\documentclass[1p]{elsarticle_modified}

\journal{CGP17004}

\usepackage{abbrmath_seonhwa} 

\usepackage{amsfonts}
\usepackage{amssymb}
\usepackage{amsmath}
\usepackage{amsthm}
\usepackage{amsbsy}

\usepackage{tikz}
\usepackage{tikz-cd}
\usetikzlibrary{arrows}

\usepackage{soul}

\let\OLDthebibliography\thebibliography
\renewcommand\thebibliography[1]{
	\OLDthebibliography{#1}
	\setlength{\parskip}{0pt}
	\setlength{\itemsep}{0.8ex} 
}

\usepackage{ulem}
\usepackage{kotex}
\usepackage{caption,subcaption}

\usepackage[colorlinks]{hyperref}

\usepackage{graphicx}

\usepackage{xcolor} 
\usepackage{float}

\theoremstyle{definition}
\newtheorem{thm}{Theorem}[section]
\newtheorem{prop}[thm]{Proposition}
\newtheorem{lem}[thm]{Lemma}
\newtheorem{cor}[thm]{Corollary}
\newtheorem{defn}[thm]{Definition}
\newtheorem{exam}[thm]{Example}
\newtheorem{rmk}[thm]{Remark}
\theoremstyle{definition} \newtheorem{notat}{Notation}

\theoremstyle{remark}   
\theoremstyle{remark}   
\theoremstyle{remark}   
\DeclareMathOperator{\Stab}{\Scal\hspace{-0.1em}\tcal\hspace{-0.1em}\acal\hspace{-0.1em}\bcal}

\def\sl{\operatorname{\textup{SL}}(2,\Cbb)}
\def\psl{\operatorname{\textup{PSL}}(2,\Cbb)}

\def\Mt{{ M \hspace{-0.96em} \widetilde{\raisebox{-0.11em}{\phantom{W}}}}}

\def\Mh{{ M \hspace{-0.96em} \widehat{\raisebox{-0.1em}{\phantom{W}}}}}

\def\Mth{{M \hspace{-0.96em} \widetilde{\raisebox{-0.15em}{\phantom{W}}} \hspace{-1.1em} \widehat{\raisebox{0.03em}{\phantom{W}}}  } }

\def\Mb{{ M \hspace{-0.96em} \overline{{\phantom{N}}} }}

\def\Mbt{{ M \hspace{-0.96em} \overline{\phantom{N}} \hspace{-0.95em} \widetilde{\raisebox{0.02em}{\phantom{W}}}}}

\def\IsoH{{\psl}}
\def\H{{\mathbb{H}^3}}
\def\Hb{{\overline{\mathbb{H}^3}}}

\def\T{{\Tscr}}
\def\Tt{{\Tscr \hspace{-0.83em} \scalebox{0.8}[1]{$\widetilde{\raisebox{-0.13em}{\phantom{W}}}$ }\hspace{-0.2em} }}

\newsavebox\CBox
\newcommand\hcancel[2][0.5pt]{%
	\ifmmode\sbox\CBox{$#2$}\else\sbox\CBox{#2}\fi%
	\makebox[0pt][l]{\usebox\CBox}%
	\rule[0.5\ht\CBox-#1/2]{\wd\CBox}{#1}}

\begin{document}

\begin{frontmatter}

\title{Octahedral developing of knot complement I: pseudo-hyperbolic structure}
	
\author{Hyuk Kim}
\address{Department of Mathematical Sciences, Seoul National University, Seoul, 08826, Korea}
\fntext[h_kim]{The first author was supported by Basic Science Research Program through the NRF of Korea funded by the Ministry of Education (2015R1D1A1A01057299).}
\ead{hyukkim@snu.ac.kr}

\author{Seonhwa Kim} 
\address{Center for Geometry and Physics, Institute for Basic Science, Pohang, 37673, Korea}
\fntext[s_kim]{The second author was supported by IBS-R003-D1.}
\ead{ryeona17@ibs.re.kr}

\author{Seokbeom Yoon}
\address{Department of Mathematical Sciences, Seoul National University, Seoul, 08826,  Korea}
\fntext[s_yoon]{The third author was supported by Basic Science Research Program through the NRF of Korea funded by the Ministry of Education (2013H1A2A1033354).}
\ead{sbyoon15@snu.ac.kr}

\begin{abstract}
It is known that a knot complement can be decomposed into ideal octahedra along a knot diagram. A solution to 
the gluing equations applied to this decomposition gives a pseudo-developing map of the knot complement, which will be called a pseudo-hyperbolic structure.  In this paper, we study these in terms of segment and region variables which are motivated by the volume conjecture so that we can compute complex volumes of all the boundary parabolic representations explicitly.
We investigate the octahedral developing and holonomy representation carefully, and obtain a concrete formula of Wirtinger generators for the representation and also of cusp shape. We demonstrate explicit  solutions for $T(2,N)$ torus knots, $J(N,M)$ knots and also for other interesting knots as examples. Using these solutions we can observe the asymptotic behavior of complex volumes and cusp shapes of these knots. We note that this construction works for any knot or link, and  reflects systematically both geometric properties of the knot complement and combinatorial aspects of the knot diagram.
\end{abstract}


\end{frontmatter}

\tableofcontents

	\section{Introduction}\label{sec:Introduction}
	
	The earlier beautiful interaction  between Riley’s work on boundary parabolic representation of a knot group  and Thurston’s work on hyperbolic geometry of a knot complement still gives us much inspiration (\cite{riley_parabolic_1972}, \cite{thurston_geometry_1977}, \cite{riley_personal}), and can be generalized further to encompass not only hyperbolic knots but also general knots (or links) if we relax the condition ``faithful and discrete'' in the representation side and ``angle sum equals $2\pi$'' in the geometric side. Although there have been some efforts and results in this direction, we would like to add more concrete and explicit geometric flavor to laying the foundation of such interaction through a series of papers starting with this one.

	If a knot complement has a hyperbolic structure, that is, it has a developing map from its universal cover to $\H$, it gives us a corresponding holonomy representation into $\psl$. Conversely suppose  we have an ideal triangulation of a knot complement and if a knot group has a representation $\rho$, then we can first construct a developing of 0-cells (or ideal points) $\rho$-equivariantly and then extend to higher skeletons using geodesics and geodesic planes, and so on to obtain a ``developing" map. (See for example \cite{zickert_volume_2009}.) Here this developing map is not a local diffeomorphism in general and we have to consider a more general type developing allowing “folding” of tetrahedra. Such a developing in this broader sense focusing on the holonomy will be called a pseudo-hyperbolic structure in this paper. To study such structures we have to discard the angle sum condition in Thurston’s edge condition. Then the remaining edge condition, ``the product of cross ratios around an edge is equal to $1$'', turns out to be the right condition for the correspondence between a representation of a knot group and a pseudo-hyperbolic structure on the knot complement.  We will call this condition the gluing equation or the hyperbolicity equation. We would like to stress again that this works for all knots or links.

     In this paper,	we will follow exactly what Thurston  did in \cite{thurston_geometry_1977} to obtain an explicit picture of the pseudo-developing map using a very specific triangulation coming from an octahedral decomposition. Then we can talk about knot invariants stemming from hyperbolic knot theory more generally applied to any knot or link. 
     The complex volume is one notable such invariant and this can be obtained from each representation of a knot group (\cite{neumann_extended_2003}, \cite{zickert_volume_2009}). We will discuss cusp shape as another such example and give an explicit formula for each boundary parabolic representation of a knot group which can be calculated from the explicit geometry of the pseudo-developing map.
	
	The real problem in such geometric arguments is the lack of the existence of a canonical triangulation for a knot complement to handle systematically, but the octahedral decomposition of a knot complement alternatively can give us a canonical method to work with, once we have a knot diagram. (Even if this gives us an ideal triangulation of the knot complement minus two points, it really doesn’t matter.) The advantage of this decomposition is obvious since we can compute an explicit formula for knot invariants through very systematic developing, and furthermore this decomposition already contains the information about the combinatorics of the knot diagram, thus opening a bridge for a possible connection between hyperbolic invariants and combinatorial ones. 
	
		The idea of connecting  an octahedron to each crossing of knot diagram  seems to  first appear in the Kashaev's study on   quantum dilogarithm and R-matrix to define his quantum link invariant (\cite{Kr1}, \cite{Kr2}) 
	and  was interpreted as an   ideal triangulation of knot complement in the effort to understand  the volume conjecture \cite{thurston_1999}.
This structure has been intensively studied especially by Y. Yokota and others  because of its relevance to the volume conjecture(\cite{murakami_intro2010}, \cite{yokota_potential_2002}, \cite{CM_2013}, \cite{OhtsukiTakata2015},...), but there are also other works in some geometric contexts (\cite{We2005},\cite{IKquandle2014},...).

In general, an octahedron can be decomposed into tetrahedra in several ways, but we use two particular decompositions called the 4-term and 5-term triangulation. These triangulations are related to the $R$-matrix of the Kashaev invariant and colored Jones polynomial respectively. 
			Precisely, the 4-term (resp., 5-term) triangulation  is a decomposition into four (resp., five)  tetrahedra as in Figure \ref{fig:four_five}  and each tetrahedron corresponds to a $q$-series term in the $R$-matrix of the Kashaev invariant (resp., colored Jones polynomial)
			 or to a dilogarithm function term in the optimistic limits (\cite{yokota_potential_2002}, \cite{CM_2013}). 
			
			Our geometric observation started from the volume fomulas given by the optimistic limits (\cite{CKK_2014},\cite{cho_optimistic_2013}).
	  As the state sum  of  the Kashaev invariant (resp., colored Jones polynomial) is expressed by the terms  indexed by  segments{\footnote{ We use the term ``segment" instead of  ``edge" for knot diagrams to avoid the confusion with an edge of a triangulation, see Section \ref{sec:OctaDeco}.} (resp., regions) of a knot diagram,   
	the cross-ratios of the 4-term (resp., 5-term) triangulation 
	are parametrized by variables assigned to segments (resp., regions) denoted by $z$-variables(resp., $w$-variables).

	We will explain that in these triangulations the $z$-variable can be interpreted geometrically as the coordinates of vertices of a hyperbolic ideal octahedron so that the cross-ratio of the  side  edges of octahedron is given by the ratio of two $z$-variables, and similarly the cross-ratio of the other edges of the octahedron can be written as ratios of $w$-variables. The hyperbolicity equation can be given in either of these variables, and with $z$(or $w$)-solutions we  can compute the pseudo-developing and holonomy representation explicitly. All these materials will be discussed in Sections \ref{sec:OctaDeco} and \ref{sec:dev}.
	
	We then discuss three applications of these constructions: boundary parabolic representations, complex volume and cusp shape. All the representations of a knot group  into $\psl$ can be obtained through the pseudo-hyperbolic structures on the knot complement as holonomy. In fact a solution of the hyperbolicity equation gives a pseudo developing and thus the corresponding holonomy. We will focus on the boundary parabolic representation in this paper and also point out what kind of subtlety arises mostly for non geometric representations when the ``pinched case" occurs in the pseudo-developing, i.e., the case where the top and bottom vertices of the octahedron coincide. There are other ways to compute the representations algebraically but the cases of possible representations bifurcate in a complicated manner and we would like to handle this problem more systematically in a subsequent paper analyzing these pinched cases. In general this method coming from pseudo-hyperbolic structures is more satisfactory since it always comes with  geometry and we can actually see what is going on. 
		
	The complex volume formula using $z$-variables was first given by Yokota 
	and he used a collapsed version of the 4-term octahedral decomposition to obtain a genuine triangulation of a knot complement and then plug in $z$-solutions to the 4-term potential function to obtain the complex volume \cite{yokota_potential_2002}. The more general non-collapsed version is done in \cite{CKK_2014}. The 5-term complex volume formula using the colored Jones polynomial was obtained by Cho and Murakami \cite{CM_2013}. Both of these are  elegant explicit formulas for complex volume which actually are anticipated from the volume conjecture as a limit, and hence this pseudo-hyperbolic structure is naturally involved in the limit algebro-analytic geometry of quantum invariants of a knot.
	
	The next invariant we consider is the cusp shape which is also a hyperbolic invariant, but now can be generalized again to any knot with a boundary parabolic representation. Although this can be calculated algebraically once we have a representation of a knot group, we can see the geometry of this representation through a pseudo-developing map. We can express this invariant as a sum of rational functions of $z$-variables defined at each crossing as in the case of the volume potential functions, but this is simpler since it doesn’t have dilogarithm functions. We will derive the formula in Section \ref{sec:Holonomy} after discussing the holonomy representation. 
	
	Since these invariants can be given by exact formulas, we can see their behaviors for a family of knots. As a sample example, we compute and show their asymptotic behavior for twist knots in Section \ref{sec:Solution}. One interesting observation for this computation is that the twist knot, $J(2,N)$ has exactly $N$ solutions  and thus has $N$ boundary parabolic representations, and as $N$ increases, the limit of the volumes of all the representations (not only the geometric one but also all the other Galois conjugates) seems to converge to the volume of the Whitehead link complement. Another interesting observation is the behavior of the Chern-Simons invariants. The volume formula gives us the values without reducing mod $\pi^2$ and this values increase almost linearly as the number of twists goes to infinity. We hope this phenomenon could be investigated more precisely. We also are wondering if this value without reducing mod $\pi^2$ suggests an $\eta$-invariant of a knot complement when it is defined appropriately  \cite{meyerhoff_eta1997}.  
	
	Now for the actual computations of these invariants and pseudo-developing and holonomy itself, we need to solve the hyperbolicity equation. At first glance this looks very complicated but as we demonstrate for the case of torus knots $T(2,N)$ and $J(N,M)$-knots in Section \ref{sec:Solution}, it has also pretty nice structure of its own, a Fibonacci type structure. In general the hyperbolicity equation in the 4-term triangulation is given by local pictures of the diagram and reveals the combinatorics of a knot in a very systematic way, and it certainly deserves to be explored further. It should have a certain inherent structure in itself just as the Ptolemy coordinates of these triangulations  have a cluster algebra structure as demonstrated in \cite{hikami_cluster_2015} and \cite{CZpersonal}.
	In fact there is a correspondence between $z$- or $w$-variables and Zickert's long edge parameters and we will discuss this correspondence along with other subjects in a subsequent paper \cite{KYptolemy}. 	
	Also the result of Thistlethwaite and Tsvietkova \cite{ThistlethwaiteTsvietkova2014}  can be reinterpreted in this setting and thereby their result works not only for hyperbolic but for any knot.
We also would like to mention that  
the behaviour of the octahedral decomposition  under Reidemeister moves is  studied  in \cite{CM_2017} giving  an elegant algebraic formula for the corresponding transitions  of $w$-solutions for a boundary parabolic representation. 
	As one already expects, there should be a relation between the character variety and the solution variety, and we will discuss this subject separately in another paper as well \cite{KP}.
	We state theorems for  a knot for simplicity but 
	 most of the discussions in this paper also work for a link.

\section*{Acknowledgement}
We should thank  our colleague Insung Park who assisted us in improving the theory. We are also grateful to C. Zickert for email discussions, in particular,  about $\psl$-representations and solutions to the gluing equations.

\section{Preliminaries : ideal triangulation and pseudo-developing}\label{sec:pseudo_dev}	
	
      We briefly review ingredients in as self-contained a manner as possible for those who are not familiar with the subject. All materials in this section are more or less known to experts in hyperbolic geometry, but we reorganize them from the literature in a manner convenient for our purpose.     
       We  denote by $\psl$  the group of orientation preserving isometries on $\H$ throughout the paper.

     In this section we would like to consider a generalized hyperbolic structure on a $3$-manifold that serves as a rigid underlying geometry whose holonomy is a given representation of its fundamental group to $\psl$. We propose to call such a geometric structure a \emph{pseudo-hyperbolic structure},  which is a generalization of a hyperbolic structure of finite volume, allowing ``folding'' and ``branched covering'' in the developing.

      We basically  refer to Zickert's setting  (in particular, Section 4 in \cite{zickert_volume_2009})
      and extend it to allow non-parabolic representations.  It would  also be helpful to see A. Champanerkar's thesis \cite{abhijit_thesis}. }
       We often use the term  \emph{pseudo-deveoloping map} (\cite{dunfield_cyclic_1999}, \cite{francaviglia_hyperbolic_2004}, \cite{segerman_pseudo-developing_2011})
     	rather than developing map    
    	to emphasize the differences : (i) we drop the condition of local diffeomorphism, commonly used for the usual developing map for $(G,X)$-structures (see Section 3 in \cite{thurston_geometry_1977}), and  (ii) we  require  continuous extendability  to the end-compactification. (The precise defnition  will be made  in the following subsection.)     
    
       \subsection{Pseudo-developing}
          Throughout the article, we always assume that a 3-manifold $M$ is orientable,  connected and  homeomorphic to the interior of a compact manifold $\Mb$ with  (possibly without) boundary components $E_1, \dots, E_k$, and such an $M$ will be called an \emph{open tame} 3-manifold.  
     We  denote  by $\Mh$, called the {\emph{end-compactification}} of $M$,  the quotient space  obtained from $\Mb$ by identifying each boundary component $E_i$ to a single point $p_i$.  
      We refer to the quotient point $p_i$ as  an \emph{ideal point} and a small deleted neighborhood of it   as  an \emph{end}  of $M$.  
      In general, $\Mh$ would be a manifold except for the ideal points which might be isolated singularities, and is sometimes called  a \emph{pseudo-manifold} {\cite{hodgson_triangulations_2014}}.
      Note that  ideal points of $M$ are contained not in   $M$, but in $\Mh$. We sometimes abuse  the terminology   ``end'' with  ``boundary''. 

     Let  $\Mt$ be the universal cover of $M$ with the covering map $\pi: \Mt \to  M$.
      Let $\Mth$ be the {end-compactification} of the universal cover $\Mt$ which is also obtained  by collapsing each boundary component of $\Mbt$ to a single point. Then the universal covering map $\pi$ admits a continuous extension $\widehat{\pi}:\Mth \to \Mh$ and an ideal point $\widetilde{p_i}$  of $\Mt$  becomes a lifting of the ideal point $p_i$ of $M$. 
      The deck transformation group $\Pi$ of $\Mt$ also acts on $\Mth$, whose orbit space is $\Mh$, but $\widehat{\pi}$ may not be a covering map because  the action may not be free on ideal points of $\Mt$.

 The fundamental group $\pi_1(\Mb,x_0)$ is isomorphic to $\Pi$ with a choice of a base point $x_0$ and its lifting $\widetilde{x_0}$.   Under this isomorphism, a peripheral subgroup of $\pi_1(\Mb,x_0)$ is identified with a subgroup of $\Pi$ preserving a boundary component of $\Mbt$. 
 	Therefore   	 	
  	we refer to the stabilizer of an  ideal point $\widetilde{p}$ of $\Mt$,  denoted by $\Stab(\widetilde{p})$, as a \emph{peripheral subgroup} of $\Pi$.

Now we define a pseudo-developing as follows.
      
       \begin{defn}[pseudo-developing]\label{defn:pseudodev}
       	A \emph{pseudo-developing} of $M$ is an equivariant pair $(\Dcal,\rho)$  where 
       	 	$\Dcal : \Mth \to  \Hb$ is a continous map 
       	 	 with $\Dcal(\Mt)\subset \H$ 
       	 	and a \emph{holonomy} homomorphism 
       	 	$\rho : \Pi \to \IsoH$  such that  $\Dcal ( g \cdot x ) = \rho(g) \cdot \Dcal(x)$ for all $g \in \Pi$ and $x \in \Mth$.
       	Two pseudo-developings $(\Dcal,\rho)$ and $(\Dcal',\rho')$ are  \emph{equivalent}  if there exists $\phi \in \IsoH$ such that $\rho=\phi \rho' \phi^{-1}$ and  there is a $\rho$-equivariant homotopy between $\Dcal$ and $\phi \circ \Dcal'$.
       \end{defn}

       	It is well-known  that for two usual developing maps $\Dcal$ and $\Dcal'$ from $\Mt$ to $\H$ with the same holonomy $\rho$ , there always exists a $\rho$-equivarient homotopy  between them
       	defined by the convex combination along  the geodesic between $\Dcal(x)$ and 
       	$\Dcal'(x)$.  However, we have to be careful for the  pseudo-developing case since the map is also defined for ideal points. (See Theorem \ref{thm:equipseudo}.)  
%
       Another issue is that the usual $\rho$-equivariant homotopy doesn't preserve some important geometric invariants such as volume in non-compact cases, because $\rho$-equivariance  does not control  the open ends well.  See \cite{dunfield_cyclic_1999} and \cite{francaviglia_hyperbolic_2004} for the details on this issue. But we remark that our definition of pseudo-developing is slightly different from theirs. 

When we consider a pseudo-developing of $M$,   sometimes it is more convenient to use  $M'$ instead of $M$ which is obtained by deleting a finite number of points of $M$.
  In that case, we get the following proposition.

\begin{prop}\label{prop:sphereboundary}
	Suppose $\overline{M'}$  has a 2-sphere boundary component and let $M$ be obtained from $\overline{M'}$ by capping off the 2-sphere boundary with a 3-ball.  Then every pseudo-developing map $\Dcal$ of $M$ is also a pseudo-developing map  of $M'$.
\end{prop}
\def\Mpt{{ M' \hspace{-1.4em} 
		\scalebox{1.13}[1]{$\widetilde{\raisebox{-0.1em}{\phantom{W'}}}$}}}
\def\Mph{{ M' \hspace{-1.4em} \scalebox{1.25}[1]{$\widehat{\raisebox{-0.05em}{\phantom{W'}}}$} }}
\begin{proof}
	The two end-compactifications $\Mph$ and $\Mh$ are essentially same and in that case we may assume $\Dcal(\Mpt)\subset \Dcal(\Mt) \subset \H$.  
\end{proof}	

\begin{rmk}
The converse of the above is also true by Proposition \ref{prop:spherecusp}   (in the sense of equivalence). Therefore we sometimes ignore  a 2-sphere  boundary component.  
We will say that   a  boundary  component of $\Mb$ is  \emph{trivial} if it is a 2-sphere, and the corresponding end or the ideal point of $M$ will be also called trivial.
	
\end{rmk}

       \subsection{Existence and equivalence through an ideal triangulation }
       
       A pseudo-developing in the previous subsection is more concretely  discussed with an ideal triangulation. 
       We first recall the notion of an ideal triangulation.
       Let $\Mh$ be 
       obtained by gluing faces  of   standard 3-simplices $\Delta_1, \Delta_2, \dots \Delta_n$ with orientation reversing identification.
       The vertices in the 0-skeleton $\Mh^{(0)}$ are the  only possible singularities as  non-manifold points. 
       Let $M$ be the complement of the 0-skeleton $\Mh \setminus \Mh^{(0)}$  and  $\Mb$ be the \emph{exterior} of the 0-skeleton $\Mh \setminus N (\Mh^{(0)})$ where $ N (\Mh^{(0)})$ is a sufficiently small open ball neighborhood  of  the 0-skeleton. Then $M$ is a connected orientable 3-manifold with the end-compactification $\Mh$, and  $\Mb$  is a compact manifold with boundary whose interior is homeomorphic to $M$. 
       We call a vertex of $\Mh^{(0)}$  an  \emph{ideal vertex} of $M$ and each simplex without vertices in $M$ an \emph{ideal simplex}, i.e., \emph{ideal tetrahedron, ideal triangle,} and \emph{ideal edge}.  
	Let us call this cell decomposition of $M$ an \emph{ideal triangulation} $\T$  and denote an ideally triangulated manifold by $(M,\T)$. 
	
		\begin{rmk}\label{rmk:existencetriangulation}
		It is well-known that we can find an ideal triangulation for any open tame  3-manifold. We can prove this fact  without much  difficulty by using a \emph{triangular pillow with a pre-drilled tube} as in \cite{We2005}. As a corollary, there exists an ideal triangulation with one ideal vertex for any closed 3-manifold $M$ (with a deleted ideal point).
	\end{rmk}
	
	Each ideal $i$-simplex of $M$  can be lifted to the universal cover $\Mt$  and we obtain an  ideal  triangulation 
	$\Tt$ of $\Mt$  where the deck transformation group $\Pi$ of $\Mt$   acts cellularly. 
 Now we can talk about a necessary and sufficient condition for the existence and equivalence of pseudo-developings with respect to a fixed representation.

        \begin{thm}[existence]\label{thm:existpseudo} 
        	
        	For a given homomorphism  $\rho :\Pi \to \IsoH$,   the following are equivalent.
        	\begin{enumerate}[(a)]
        		\item There exists a pseudo-developing  ($\Dcal,\rho)$ of $M$.
        		\item The $\rho$-image of each peripheral subgroup of $\Pi$  has a fixed point 
        		in $\overline{\Hbb^3}$. 
        \end{enumerate}
        	
        \end{thm}
        
        \begin{proof}
        	It is obvious that (a) implies (b) since for each ideal point $\widetilde{p}$ of $\Mt$, the $\rho$-image of the peripheral subgroup $\Stab(\widetilde{p})$ fixes $\Dcal(\widetilde{p})$ by equivariance. 
        	To prove the converse, we use an ideal triangulation $\T$ of $M$. Let $p_1,p_2,\dots,p_k$  be ideal points of $M$ and choose an ideal point  $\widetilde{p_i}$  of $\Mt$ among the  liftings of each $p_i$.
        	For each $\widetilde{p_i}$ we define $\Dcal(\widetilde{p_i})$ to be a fixed point  of $\rho ( \Stab (\widetilde{p_i} ) )$ in $\Hb$. Here the existence of the fixed point comes from condition (b). 
        Since every ideal point $\Mt$ is of the form $g \cdot \widetilde{p_i}$ for some $g \in \Pi$, we can define $\Dcal(g\cdot\widetilde{ p_i})$ $\rho$-equivariantly, i.e., $\Dcal(g\cdot\widetilde{ p_i}):= \rho(g)\cdot \Dcal(\widetilde{ p_i})$.
        
For each 1-cell $e_i$ of $(M,\T)$ choose a lifting $\widetilde{e_i}$ in $\Mt$. Let  the ideal vertices of  $\widetilde{e_i}$ be  $\widetilde{p_i}$ and $\widetilde{q_i}$, and define $\Dcal$ on $\widetilde{e_i}$ to be  an arbitrary  
path in $\H$ connecting $\Dcal(\widetilde{p_i})$ and $\Dcal(\widetilde{q_i})$. 
 For  each 1-cell $\widetilde{e}$ of $\Mt$, there exists a unique $g\in \Pi$ and $\widetilde{e_i}$ such that  $\widetilde{e}=g\cdot \widetilde{e_i}$. Therefore we can extend $\Dcal$ to all 1-cells of $\Mt$ $\rho$-equivariantly. Similarly, we can also define $\Dcal$  on chosen lifted $2$- or $3$-cells first and then extend $\Dcal$ to the whole set of $2$- or $3$- cells in an equivariant way, respectively. By  construction, $\Dcal$ is also a $\rho$-equivariant map on $\Mth$ and $(\Dcal,\rho)$ satisfies  the conditions to be a   pseudo-developing.  
        \end{proof}

   \begin{rmk} \label{rmk:sphereboundary}
   	Let $M'$ be obtained by deleting finitely many points from  $M$. To obtain a pseudo-developing of $M'$,  it is enough to find a pseudo-developing of $M$ by  Proposition \ref{prop:sphereboundary}. Conversely,
   	to obtain a pseudo-developing of $M$, we can use  $M'$ instead of $M$. During the construction of a pseudo-developing of $M'$ using an ideal triangulation, if we choose the $\rho$-image of trivial ideal points within $\H$,   then the resulting  $\Dcal'$ is also a pseudo-developing map of $M$ without any modification. Besides, Theorem \ref{thm:existpseudo} holds also for a closed $3$-manifold. 
   \end{rmk}
Now, we have a  corollary which is easy but fundamental. 
 \begin{cor}\label{cor:existdev}
 For a given representation $\rho$, there exists a pseudo-developing $(\Dcal,\rho)$ of $M$ if the $\rho$-image of the peripheral subgroup is abelian.  In particular, if $M$ is closed or has only torus boundary components, there exists a pseudo-developing $(\Dcal,\rho)$ for any representation $\rho$. 
 \end{cor}
\begin{proof}
	If $M$ has a boundary(=end), there always exists an ideal triangulation as we mentioned in Remark \ref{rmk:existencetriangulation}. If $M$ is closed, we consider an ideally triangulated  $M'$ with a trivial boundary instead of $M$ as Remark \ref{rmk:sphereboundary}.  Therefore  the existence of a  pseudo-developing depends only on the existence of a representation whose image of each peripheral subgroup has a fixed point in $\Hb$.
	Since any abelian subgroup of $\IsoH$ has a fixed point in $\Hb$,  we obtain a pseudo-developing of $M$ by Theorem \ref{thm:existpseudo}. 
	In particular, if $M$ has only torus or trivial boundary components, each peripheral subgroup is already abelian. Therefore there always exists a fixed point of the peripheral image of any representation.
\end{proof}

Similarly to the case of existence, we can find a necessary and sufficient condition that two pseudo-developings are equivalent when they have the same holonomy.

\begin{thm}[equivalence]\label{thm:equipseudo}
	Let $(\Dcal,\rho)$ and $(\Dcal',\rho)$ be two pseudo-developings of $M$ with the same holonomy. Then the following are equivalent.
	\begin{enumerate}[(a)]
		\item  
		The two pseudo-developing ($\Dcal,\rho$) and ($\Dcal',\rho$) are equivalent. 
		\item For a lifting  $\widetilde{p}$ of each ideal point $p_i$, there is a path $\gamma:[0,1] \to  \overline{\Hbb^3}$  connecting  $\Dcal(\widetilde{p})$ and $\Dcal'(\widetilde{p})$ such that $ \rho(\Stab(\widetilde{p}))$ fixes $\gamma(t)$ for all $t\in[0,1]$.		
	\end{enumerate}
	
\end{thm}

\begin{proof}
	If there is a $\rho$-equivariant homotopy $H(t,x) : [0,1] \times \Mth \to \Hb$ between $\Dcal$ and $\Dcal'$, then  $H(t,\widetilde{p})$ will be  a  path $\gamma(t)$ connecting $\Dcal(\widetilde{p})$ and $\Dcal'(\widetilde{p})$, which is fixed by $\rho(\Stab(\widetilde{p}))$.   Conversely, we can construct a $\rho$-equivariant homotopy on $0$-cells  first, i.e., for each ideal point $p$ we already have $\gamma(t)$, so  we define $H(t,\widetilde{p}):= \gamma{(t)}$ and then $H(t,g\cdot\widetilde{p}):=\rho(g)\cdot H(t,\widetilde{p})$ for every $g\in \Pi$. Similarly, we extend it  on $i$-cells $\rho$-equivariantly from the homotopy of $(i-1)$-cells as 
	we did in the proof of Theorem \ref{thm:existpseudo}.  
\end{proof}

Then we have the following corollary (see also Proposition \ref{prop:sphereboundary} and Remark \ref{rmk:sphereboundary}), which says that   a 2-sphere boundary component really doesn't matter when we consider a pseudo-developing.

\begin{prop}\label{prop:spherecusp}
Let $M'$ be obtained from $M$ by deleting a  point $p$ of $M$.
For any pseudo-developing map $\Dcal'$ of $M'$, there is an equivalent pseudo-developing map  of $M'$ which is also a pseudo-developing of $M$. 
(See also the spinning construction in \cite{lty_2013_spinning}.)  
\end{prop}
\begin{proof}
	If $\Dcal' (\widetilde{p}) \in  \H$ then $\Dcal'$ itself is also a pseudo-developing of $M$ as mentioned in Remark \ref{rmk:sphereboundary}.
	 Although $\Dcal' (\widetilde{p}) \in  \Hb \setminus \H$, we can modify the pseudo-developing map so that  the image of $\widetilde{p}$ is in $\H$ since   
	 $\rho(\Stab(\widetilde{p}))$ is a trivial group that fixes all of  $\Hb$ and there is no obstacle  to defining an equivariant homotopy.
\end{proof}
Similar to Corollary \ref{cor:existdev},
we obtain an obvious corollary for uniqueness of  pseudo-developings as follows. 
\begin{cor}\label{cor:equicor}
	There exists a unique pseudo-developing $(\Dcal,\rho)$ for a representation $\rho$ up to equivalence 	if the $\rho$-image of each peripheral subgroup is a parabolic, elliptic or trivial subgroup of $\IsoH$. 
\end{cor}
\begin{proof}
	The  fixed points on $\Hb$ of a parabolic, elliptic or trivial subgroup consists of only one connected component, i.e., is a unique point, a complete geodesic of infinite length or all of $\Hb$ respectively. This implies the existence and the uniqueness directly by Theorem \ref{thm:existpseudo} and Theorem \ref{thm:equipseudo}.
\end{proof}

\subsection{Straightening and pseudo-hyperbolic structure} \label{ref:straightening}
In general a pseudo-developing image can be too arbitrary to see any meaningful geometry locally and we want to consider a more  concrete geometric developing, i.e., a straightened developing. First we need a following preparation. 

\begin{defn}[degenerate pseudo-developing]
	A pseudo-developing map is \emph{degenerate} for an ideal triangulation $\T$ if there is a lifted ideal edge in $\Tt$ whose ideal vertices are mapped to the same point. 
\end{defn}
To obtain a non-degenerate pseudo-developing, the ideal triangulation requires a combinatorial property as follows.
\begin{prop}\label{prop:essentialedge}
	If there exists a non-degenerate developing for $\T$, then each ideal edge in $\T$  is \emph{essential}, which means that a lifting of the edge  has two distinct end points in $\Mth$. (An ideal triangulation is called \emph{virtually nonsingular} in \cite{segerman_pseudo-developing_2011} if it has only essential edges.) 
\end{prop}
\begin{proof}
	If there is a nonessential edge, i.e., an ideal edge in $\T$ whose liftings already has the same end point in $\Mth$, then any pseudo-developing map should be degenerate for $\T$.		
\end{proof}

On the contrary,
the developing image of each lifted $i$-cell may be better behaved. The following definition is modelled on  geometric simplices. 
\begin{defn}[straightened pseudo-developing]\label{def:straightened}
	A pseudo-developing map $\Dcal$ is \emph{straightened} for an ideal triangulation $\T$ if  $\Dcal$ is non-degenerate for $\T$ and the $\Dcal$-image of each lifted $i$-cells is 	 totally geodesic.  
\end{defn} 
In fact, the non-degeneracy condition is the only requirement to carry a straightened pseudo-developing as follows.
 \begin{prop}\label{prop:nondegstr}
If $\Dcal$ is non-degenerate for $\T$ then there exists a $\rho$-equivariantly homotopic  pseudo-developing $\Dcal'$  which is straightened for $\T$. In this case,
 $\Dcal'$ is called a \emph{straightening} of $\Dcal$. 
\end{prop}
\begin{proof}
As  in the proof of Theorem \ref{thm:existpseudo},	
define $\Dcal' =\Dcal$ on ideal points and extend $\rho$-equivariantly to  each lifted higher skeleton using  geodesics and ideal geodesic triangles and so on. The resulting pseudo-developing is obviously $\rho$-equivariantly homotopic to $\Dcal$. 
\end{proof}	
 
 Now we  define a ``pseudo-hyperbolic'' structure on $M$  as follows. 
 \begin{defn}[pseudo-hyperbolic structure]
 	A \emph{pseudo-hyperbolic structure} on $M$ is   a straightened pseudo-developing ($\Dcal$,$\rho$) of $M'$ which is obtained from  $M$   by possibly deleting a finite number of  points of $M$. 
 	We refer to $\Dcal$ and $\rho $ as  the \emph{developing map} and  the \emph{holonomy} representation of  a pseudo-hyperbolic manifold $M$,  respectively.  
 	Two pseudo-hyperbolic structures are \emph{equivalent}    if their  pseudo-developings  are equivalent with respect to  the intersection of the two $M'$'s used to define the pseudo-hyperbolic structures. 
 \end{defn}
\begin{rmk}[independence of ideal triangulation]
	As Theorem \ref{thm:equipseudo} shows,   the equivalence class of  pseudo-hyperbolic structures depends only on the holonomy representation and $\Dcal$-image of nontrivial ideal points. Therefore, although we  always need an ideal triangulation to obtain  a pseudo-hyperbolic structure concretely, it is essentially independent of the choice of ideal triangulation.	
\end{rmk}
We would like to emphasize again that   deleting  points really doesn't matter for the notion of a  pseudo-hyperbolic structure. Furthermore,
 we immediately obtain several fundamental  statements about  pseudo-hyperbolic structures as follows.

\begin{thm}
	Let $M$ be a open tame 3-manifold and $\rho$ be a $\psl$-representation.
	\begin{enumerate}[(a)]
		\item 	If the $\rho$-image of each peripheral subgroup is abelian, then  $M$  admits a pseudo-hyperbolic structure with   holonomy $\rho$. 
		\item  If $M$ is closed or has only  torus boundary components, then  there exists a pseudo-hyperbolic structure on $M$ for any representation $\rho$.
		\item If a representation $\rho$ maps   each peripheral  subgroup to a parabolic,  elliptic or  trivial  subgroup of $\IsoH$ then there exists a corresponding pseudo-hyperbolic structure unique  up to $\rho$-equivariant homotopy.
	\end{enumerate}
\end{thm} 
\begin{proof}
	 By Corollary \ref{cor:existdev} and Corollary \ref{cor:equicor} we can find a suitable  pseudo-developing  $(\Dcal,\rho)$  for  statements (a),(b) and (c) through an ideal triangulation $\T$. The only remaining part to prove is whether the straightenings exist, where the only requirement is non-degeneracy for $\T$ by  Proposition \ref{prop:nondegstr}. If $\Dcal$ is degenerate for $\T$ then we take a subdivided ideal triangulation $\T'$ from  $\T$ where each degenerate edge connecting ${p_i}$ and ${p_j}$   splits into two edges  $({p_i},{p_k})$ and $({p_k},{p_j})$  by adding a new ideal vertex  $p_k$ within the edge. Then $\Dcal$ is a pseudo-developing map of $(M',\T')$ with added new ideal vertices and  we can modify $\Dcal(\widetilde{p_k})$ within  its homotopy class to be distinct from $\Dcal(\widetilde{p_i} )$ or $\Dcal(\widetilde{p_j})$ because $p_k$ has a trivial peripheral subgroup. Then we obtain a  non-degenerate pseudo-developing for $\T'$. (In fact, a single barycentric subdivision is enough to do this. See  Theorem 4.11 in \cite{zickert_volume_2009}.)
\end{proof} 
A pseudo-hyperbolic structure may be thought of as a geometric apparatus to study a representation of a 3-manifold. A usual hyperbolic structure is  a special case of a pseudo-hyperbolic structure and  we can also consider  hyperbolic invariants of pseudo-hyperbolic structures. This gives us some advantages both in  the theoretical  approach and in calculating them  as shown by the study of  knot complements in this article.

From now on, we always assume that the $\Dcal$-image of  ideal points is contained in $\partial \Hb$. For this case  we say that  a straightened pseudo-developing  is \emph{ideally straightened} for the \emph{underlying triangulation} $\T$. 
Then   we can use the \emph{cross-ratios} to investigate the structure as in the subsequent subsections.    	

\subsection{The cross-ratio and  Thurston's gluing equation}
Let us consider a standard tetrahedron $\Delta$ with the standard orientation with respect to the vertex  labeling  $0,1,2$ and $3$. Consider a hyperbolic ideal tetrahedron in $\Hb$ which is a totally geodesic embedding of $\Delta$ where each vertex $i$  maps to $v_i$ on $\partial \Hb = \Cbb \cup \{\infty\}$. We then assign a cross-ratio 
\begin{equation}\label{eq:crossratiodef}
[v_i,v_j,v_k,v_l]:=\frac{(v_i-v_l)(v_j-v_k)}{(v_i-v_k)(v_j-v_l)}
\end{equation}  
  to an edge $(i,j)$ when $(i,j,k,l)$ gives the standard orientation. Note that the cross-ratio does not depend on the choice between $(i,j)$ and $(j,i)$, so it is assigned to an unoriented edge. An ideal tetrahedron is called \emph{degenerate} if the cross-ratio is $0$, $1$ or  $\infty$. We denote the cross-ratio at an edge $e$ of a tetrahedron $\Delta$ by $cross\text{-}ratio(\Delta,e)$. 
  The cross-ratio values  for  each edge are  preserved under hyperbolic isometries and    satisfy a well-known relation, 
\begin{equation}\label{eq:crossratiorel}
\vcenter{\hbox{	
\begingroup%
  \makeatletter%
  \providecommand\color[2][]{%
    \errmessage{(Inkscape) Color is used for the text in Inkscape, but the package 'color.sty' is not loaded}%
    \renewcommand\color[2][]{}%
  }%
  \providecommand\transparent[1]{%
    \errmessage{(Inkscape) Transparency is used (non-zero) for the text in Inkscape, but the package 'transparent.sty' is not loaded}%
    \renewcommand\transparent[1]{}%
  }%
  \providecommand\rotatebox[2]{#2}%
  \ifx\svgwidth\undefined%
    \setlength{\unitlength}{71.23224717bp}%
    \ifx\svgscale\undefined%
      \relax%
    \else%
      \setlength{\unitlength}{\unitlength * \real{\svgscale}}%
    \fi%
  \else%
    \setlength{\unitlength}{\svgwidth}%
  \fi%
  \global\let\svgwidth\undefined%
  \global\let\svgscale\undefined%
  \makeatother%
  \begin{picture}(1,0.84213355)%
    \put(0,0){\includegraphics[width=\unitlength,page=1]{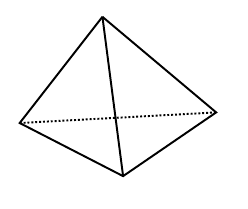}}%
    \put(-0.07395512,0.29782549){\color[rgb]{0,0,0}\makebox(0,0)[lb]{\smash{$v_0$}}}%
    \put(0.34698952,0.80082171){\color[rgb]{0,0,0}\makebox(0,0)[lb]{\smash{$v_1$}}}%
    \put(0.48260137,0.03046293){\color[rgb]{0,0,0}\makebox(0,0)[lb]{\smash{$v_2$}}}%
    \put(0.90385225,0.359075){\color[rgb]{0,0,0}\makebox(0,0)[lb]{\smash{$v_3$}}}%
    \put(4.71165232,0.46563134){\color[rgb]{0,0,0}\makebox(0,0)[lb]{\smash{}}}%
    \put(0,0){\includegraphics[width=\unitlength,page=2]{tetracross.pdf}}%
    \put(0.41349282,0.48863354){\color[rgb]{0,0,0}\makebox(0,0)[lb]{\smash{$z'$}}}%
    \put(0,0){\includegraphics[width=\unitlength,page=3]{tetracross.pdf}}%
    \put(0.55140727,0.34259342){\color[rgb]{0,0,0}\makebox(0,0)[lb]{\smash{$z'$}}}%
    \put(0,0){\includegraphics[width=\unitlength,page=4]{tetracross.pdf}}%
    \put(0.25774727,0.19787578){\color[rgb]{0,0,0}\makebox(0,0)[lb]{\smash{$z''$}}}%
    \put(0,0){\includegraphics[width=\unitlength,page=5]{tetracross.pdf}}%
    \put(0.20906628,0.53047686){\color[rgb]{0,0,0}\makebox(0,0)[lb]{\smash{$z$}}}%
    \put(0,0){\includegraphics[width=\unitlength,page=6]{tetracross.pdf}}%
    \put(0.63003384,0.53327572){\color[rgb]{0,0,0}\makebox(0,0)[lb]{\smash{$z''$}}}%
    \put(0,0){\includegraphics[width=\unitlength,page=7]{tetracross.pdf}}%
    \put(0.64864245,0.21531718){\color[rgb]{0,0,0}\makebox(0,0)[lb]{\smash{$z$}}}%
  \end{picture}%
\endgroup%
}}, \hspace{3em} z'=\dfrac{1}{1-z},~~ z''=1-\dfrac{1}{z}. 
\end{equation}
Therefore it is enough to consider a single complex number $z$, called the  $\emph{shape parameter}$, for a hyperbolic ideal tetrahedron $\Delta$ with a preferred edge.    Note that  if $z = $  $0$, $1$ or  $\infty$ then the cross-ratio at any edge is   $0$, $1$ or  $\infty$.

The shapes of the ideal tetrahedra of a pseudo-hyperbolic manifold  satisfy an algebraic relation for each edge, which will be called \emph{Thurston's gluing equation}. 
Recall that $(M,\Tcal)$ is an ideally triangulated 3-manifold where the ideal tetrahedra $\Delta_1, \dots, \Delta_n$ are glued together by  face pairing homeomorphisms. Each $\Delta_i$ has six (unoriented) edges denoted by $(\Delta_i,e_k)$ where $e_k$ is an edge of  $\Delta_i$ with $k=6i+1,\dots,6i+6$. 
Under the gluing pattern, we have a partition of the set of disjoint edges $(\Delta_i,e_k)$'s where each equivalence class, denoted by $[e]$, is the set of tetrahedral edges identified to an ideal edge $e$ in $\T$.

\begin{thm}\label{thm:thurstoneq}
Let $M$ be an ideally straightened pseudo-hyperbolic manifold  with an underlying ideal triangulation $\T$.  Then the shape parameters of hyperbolic ideal tetrahedra in $\T$ satisfy 
{Thurston's gluing equation} (simply  \emph{gluing equations} or  \emph{hyperbolocity equations}) which is  a system of equations for $(M,\Tcal)$  where  each equation is indexed by an 1-cell $e$ of $\Tcal$: 

\begin{equation}\label{eq:gluingeq}
\prod_{(\Delta_i,e_k) \in [e] }  cross\text{-}ratio(\Delta_i,e_k) = 1. 
\end{equation}
\end{thm}

\begin{proof}
%
	
	Let us consider an ideal tetrahedron $\Delta_i$ adjacent to each edge $e$ in $\T$, i.e.,  $(\Delta_i,e_k) \in [e]$.
	Once we choose a lifting $\widetilde{e}$ of $e$ in $\Mth$, we can take  a unique lifting  $\widetilde{\Delta_{i}}$ whose lifted edge $\widetilde{e_k}$ is identified to $\widetilde{e}$, i.e., $(\widetilde{\Delta_i},\widetilde{e_k}) \in [\widetilde{e}]$. 
	The link of $\widetilde{e}$ in $\Tt$ is an  edge path cycle which consists of the  opposite edges of $\widetilde{e_k}$ in each $\widetilde{\Delta_i}$. It is denoted by a cyclically consecutive  ideal points $\widetilde{p_0},...,\widetilde{p_m}(=\widetilde{p_0})$ in $\Tt$. 
	By conjugation of an element of $\psl$, we can assume $\Dcal$-image of the end points of $\widetilde{e}$ are  $0$ and $\infty$ in $ \Cbb \cup \{\infty\}$. Then
	\begin{align*}
	\prod_{(\Delta_i,e_k) \in [e] } cross\text{-}ratio(\Delta_i,e_k) &= \prod_{(\widetilde{\Delta_i},\widetilde{e_k}) \in [\widetilde{e}] } 
	 cross\text{-}ratio(\widetilde{\Delta_i},\widetilde{e_k}) \\
	&= \prod_{j=1}^m 
	 \frac{\Dcal(\widetilde{p_j})}{\Dcal(\widetilde{p_{j-1}})}=1.
	\end{align*}	
\end{proof}

\begin{rmk}
	``Thurston's gluing equation'' is  sometimes taken to mean  not only the above gluing equation at each edge, but also  a condition at each boundary for  complete hyperbolic structure or for a boundary parabolic representation. See Chapter 4 in \cite{thurston_geometry_1977} or \cite{neumann_zagier_1985}.
\end{rmk}
In a sense, Theorem \ref{thm:thurstoneq} is trivial by  definition. But the converse is also true and has more fundamental meaning.  
\subsection{Thurston's construction}\label{sec:TY_construction}
Thurston's gluing equation is a necessary and sufficient condition to define an ideally straightened pseudo-hyperbolic structure, i.e. we can construct a pseudo-developing map and holonomy representation directly from   a solution to  Thurston's gluing equation. 
 
%
The essential idea is of course due to W. Thurston \cite{thurston_geometry_1977} and an elaborated exposition for the generalized case was given by T. Yoshida  \cite{yoshida_ideal_1991}. Although he only considers a hyperbolic manifold decomposed into  hyperbolic ideal tetrahedra, the method is also applicable for a topological ideal triangulation (for example, see \cite{luofeng_volume_optimization}, \cite{segerman_pseudo-developing_2011}). It will be worthwhile to see the slight difference from Yoshida's proof in  Section 5 of \cite{yoshida_ideal_1991}. In our proof, we don't need any requirement  on the  ideal triangulation. (See also Corollary \ref{rmk:Yoshidadiff}).
\begin{thm}\label{thm:thurston_construction}
	Let $(M,\Tcal)$ be an ideally triangulated 3-manifold. If  there is a solution to Thurston's gluing equation, then there exists a straightened pseudo-developing map $\Dcal$ associated to the solution, and the holonomy $\rho$ is unique up to conjugation. 
\end{thm}

\begin{proof}
Let  $z:=(z_1, z_2, \dots, z_n)$ be  a solution to Thurston's gluing equation for $(M,\Tcal)$. 
Let  $\Fcal$ be the set of face pairing homeomorphisms, i.e.,
	$M = \bigcup \Delta_i = \coprod \Delta_i \Big/ {\sim}$ where $x \sim f(x)$ for $f \in \Fcal$.
 Note that   each $\Delta_i$ is an ideal 3-simplex and let each  $ \Delta_i{(z_i)}$ (or simply $\Delta_i{(z)}$  for  notational simplicity) be a hyperbolic ideal tetrahedron of $\Delta_i$ whose  cross-ratio is given by $z_i$.  
 At first, we construct a topological space $M_z$ with respect to  $z$ which is obtained by gluing the hyperbolic ideal tetrahedra $\Delta_i(z)$   instead of   $\Delta_i$ with the same gluing pattern. Each gluing map, denoted by adding the subscript $z$ (as in $f_z$ and $\Fcal_z$), is canonically chosen by the unique isometry  between two hyperbolic ideal triangles. So we obtain
 $M_z = \bigcup_i \Delta_i(z) = \coprod \Delta_i(z)  \Big/ {\sim}$
 where $x \sim f_z(x)$ for a face gluing map $f_z \in \Fcal_z$.

Note that  $M_z$  may  not be homeomorphic to   $M$  since there might be a flat tetrahedron with real-valued cross-ratio. There is a continuous map $\Ccal: M \to M_z$ using an  obvious  map for each simplex $\Delta_i$ to $\Delta_i(z)$ and then redefining $f$ to satisfy the gluing  compatibility,  $\Ccal \circ f = f_z \circ \Ccal$.

Let us choose an ideal tetrahedron  $\Delta_{i_0}$, called a \emph{base tetrahedron}, and   a base point $x$ in  the interior of $\Delta_{i_0}$  and consider a continuous  curve $\alpha:[0,1] \to M$ with $\alpha(0)=x$. After a slight perturbation fixing the end points, we can divide $\alpha$ into finitely many  curves $\alpha_0, \alpha_1,\dots,\alpha_k$
 such that 
 \begin{enumerate}[(i)]
 	\vspace{-1ex}
 	\item $\alpha_0*\alpha_1*\cdots*\alpha_k=\alpha$,
 	\item each $\alpha_j$  is contained in the interior of  $\Delta_{i_j}$ away from its end points,
 	\item $\alpha_j(1)=\alpha_{j+1}(0)$ is placed in the interior of a common face  $\Delta_{i_j} \cap \Delta_{i_{j+1}}$.
 \end{enumerate}
Thus, we get a finite sequence of  ideal tetrahedra $\Delta_{i_0},\Delta_{i_1},\dots,\Delta_{i_k}$ along $\alpha$ such that each $ \text{Im}(\alpha_j) \subset \Delta_{i_j}$  and $ \alpha \subset \bigcup \Delta_{i_j}$.

 Now we choose   a hyperbolic ideal triangle in $\H$, called an \emph{initial triangle}, and a face of the base tetrahedron $\Delta_{i_0}$, called a $\emph{base triangle}$.
We can take an isometry $g_{0}:\Delta_{i_0}(z) \to \H$ so that    $g_0\circ\Ccal$ maps the base triangle in $\Delta_{i_0}$ to the initial triangle in $\H$. 
Since $\text{Im}(\alpha_0)\subset \Delta_{i_0}$, we   obtain a path $\widetilde\alpha_0 = g_{0} \circ \Ccal \circ \alpha_0$ and 
then the developing of $\Delta_{i_1}$ is determined by the embedded face of $\Delta_{i_0}$ containing $\alpha_0(1)=\alpha_1(0)$.
Note that every shape  of  $\Delta_{i_j}(z)$ is  determined by the cross-ratio  $z_{i_j}$, that is there is a unique isometric embedding $g_j: \Delta_{i_j}(z) \to \H$  whenever a face of $\Delta_{i_j}(z)$ is already embedded in $\H$. Therefore each $\Delta_{i_j}(z)$ and the  path $\Ccal \circ \alpha_j$ is consecutively  embedded into $\H$  so as to form 
a curve $\widetilde{\alpha} := \widetilde{\alpha_0} * \widetilde{\alpha_1} * \cdots * \widetilde{\alpha_k}:  [0,1] \to \Hbb^3$ where   $\widetilde\alpha_j = g_{j} \circ \Ccal \circ \alpha_j$.
Consequently,   we obtain a developing $\widetilde{\alpha} : [0,1] \to  \Hbb^3$ of a curve $\alpha: [0,1] \to M$ in a unique way up to a choice of  initial triangle (when the base tetrahedron with  base triangle is fixed).

Now we construct a developing map $\Dcal : \widetilde{M} \to \Hbb^3$.
The fact that each cross-ratio is a solution of  Thurston's gluing equation (\ref{eq:gluingeq}) exactly implies that  $\widetilde{\alpha}(1)$  depends only on the homotopy class $[\alpha]$  fixing the end points and $\Dcal$ is continuous. Therefore,
 by identifying the universal cover $(\widetilde{M},\widetilde{x})$ with the set of homotopy classes of paths starting at $x$ in $M$, 
we obtain a well-defined continuous map,
\begin{equation}\label{eq:TYconstruction}
\Dcal : \Mt \to \Hbb^3 ~~\text{ by } [\alpha] \mapsto \widetilde{\alpha}(1).
\end{equation}
A deck transformation $g \in \Pi$ is identified with a loop $\gamma_g$ in $ \pi_1(M,x) $. Then  $\Dcal \circ g$  is obtained by the same construction using $\gamma_g * \alpha$ instead of $\alpha$, i.e.,  $\Dcal \circ g ( [\alpha]) = (\widetilde{\gamma_g * \alpha })(1)$  and  $\Dcal \circ g$ is also a developing with a different initial triangle obtained by  development along $\gamma_g$.  
We define the difference between the  the  initial triangles as $\rho(g)$ and  obtain a map $\rho:\Pi\to\IsoH$ and  $\Dcal \circ g = \rho(g) \circ \Dcal$ by the uniqueness of developing\footnote{If one choose a different $\Ccal:M\to M_z$, the developing is not unique as a map. But the simplex-wise image of the developing doesn't change.}. Then it  follows that $\rho(g_1  g_2) = \rho(g_1) \rho(g_2)$ as  usual. 
By construction it is obvious that  the developing map  can be extended to the end-compatification of $\Mt$ and hence we obtain  a 
pseudo-developing map from $\Mth$ to $\Hb$.  Moreover the resulting pseudo-developing is  unique up to equivalence,  and  the choice of initial triangle  exactly corresponds to the choice of a  representative within the conjugacy  class of the holonomy.
\end{proof}

     A pseudo-developing  from the Thurston construction can be thought of as  a geometric developing of a pseudo-hyperbolic manifold. We emphasize again that the resulting pseudo-hyperbolic structure from a solution is unique up to equivalence.

     Furthermore, as  mentioned before, we don't need any combinatorial assumption for the ideal triangulation $\T$ to construct a pseudo-developing map from a solution and hence we immediately have the following corollary. 
     
     \begin{cor}\label{rmk:Yoshidadiff}
     	If there is a solution of Thurston's gluing equation for $(M,\T)$ then the ideal triangulation $\T$ has only essential edges, i.e., is virtually non-singular.  (see also Theorem~1 in  \cite{segerman_pseudo-developing_2011}.)
     \end{cor}
     \begin{proof}
     	If we have a solution, we can construct a straightened pseudo-developing. Hence the resulting pseudo-developing  is   non-degenerate for $\T$ by Definition \ref{def:straightened}. Then Proposition~\ref{prop:essentialedge} completes the proof. 
     \end{proof}

	Note that a developing from the above construction may naturally have a ``folding''  which occurs at the common face between adjacent tetrahedra whose cross-ratios have positive and negative imaginary parts, respectively. 
	This phenomenon was already observed by W.~Thurston under the name   \emph{degree one ideal triangulation}  and also was studied from the point of the scissors congruence and Bloch group in \cite{neumann_bloch_1999}.

  \section{The octahedral decomposition of a knot complement}\label{sec:OctaDeco}
	In this section, we review the octahedral decomposition of a knot complement minus two points to set-up our discussion. 
	Throughout the paper, we fix $K$ to denote a knot in $S^3$ and $D$ to denote an oriented diagram of $K$. Let $N$ be the number of crossings of $D$ and denote the
 	crossings of $D$ by $c_1,\cdots,c_N$. We will use terms and notation about $D$ as follows.
 	\begin{itemize}
	 	 	\item  Deleting all crossings in the projection plane, $D$ is separated into the disjoint union of open intervals. We call the closure of each open interval a \emph{segment} of $D$. There are $2N$ segments of $D$ and we denote them by $s_1, \cdots, s_{2N}$. We use the term ``segment" instead of the usual term ``edge" of a graph to avoid confusion with an edge of a tetrahedron.

	 	 	\item Along the oriented diagram $D$, the union of segments from an under-passing crossing to the following under-passing crossing is called an \emph{over-arc} of $D$. Similarly, an \emph{under-arc} of $D$ is the union of segments from an over-passing crossing to the following over-passing crossing. Note that $D$ has $N$ over-arcs and $N$ under-arcs.
	 	 	
	 	 	\item A \emph{region} of $D$ is a connected component of the complement of $D$ in the projection plane. There are $N+2$ regions of $D$ and we denote them by $r_1, \cdots, r_{N+2}$. 
 	 \end{itemize}  
 	 \begin{figure}[!h]
	 	 	\centering
	 	 	\scalebox{1}{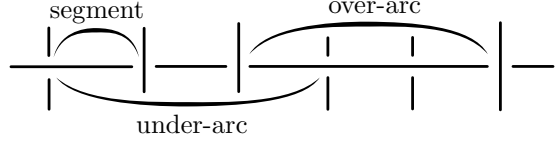}
	 	 	\label{fig:notation}
	 	 	\caption{Segments, over-arcs and under-arcs.}
 	 \end{figure}
   \subsection{Ideal triangulations of $S^3\setminus(K \cup \{\textrm{two points}\})$}      \label{sec:toptriangulation}
 We briefly recall the octahedral decomposition of $S^3\setminus(K \cup \{\textrm{two points}\})$. We first put an ideal octahedron $o_k$ on each crossing $c_k$ such that the top and the bottom vertices $o_k$ touch the over- and the under-passing arcs of $c_k$,   respectively, as in Figure \ref{fig:octaCrossing}.
   \begin{figure}[H]
	\centering
	\begin{align*}
		~&~\hspace{1em}  \vcenter{\hbox{{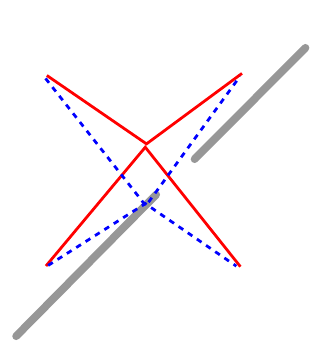}}} &\hspace{2em} \hspace{2em}   ~&~ 	\vcenter{\hbox{{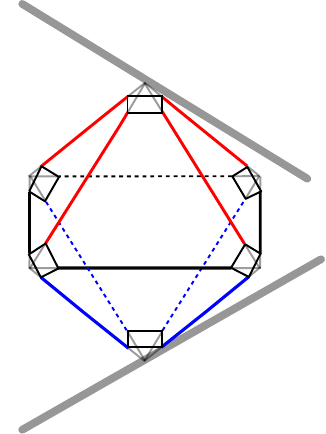}}} 
     \end{align*}
     \caption{An ideal octahedron at a crossing.}
      \label{fig:octaCrossing}
   \end{figure}    
We then glue two pairs of hypotenuses of $o_k$ as follows. Here we denote ideal vertices of $o_k$ by $A,B, \cdots,F$, temporarily.
\begin{itemize}
	\item Identify two upper-hypotenuses $AC$ and $AE$ to form a single edge above $A$ as in Figure~\ref{fig:identified_octa}(a).
	\item Identify two lower-hypotenuses $BF$ and $DF$ to form a single edge below $F$ as in Figure~\ref{fig:identified_octa}(b).
\end{itemize} 
    \begin{figure}[!h]
    	\centering
    	\begin{subfigure}[H]{0.45\textwidth}
    		\centering
    		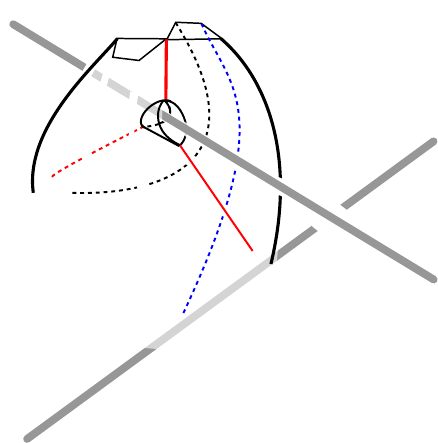
    		\caption{Identify $AC$ and $AE$.}
    	\end{subfigure}
    	\begin{subfigure}[H]{0.45\textwidth}
    		\centering
    		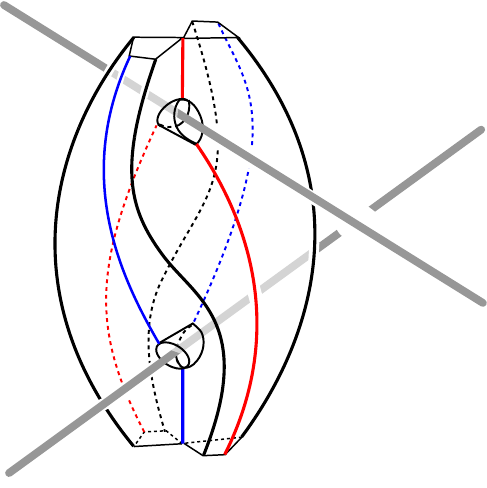
    		\caption{Identify $BF$ and $DF$.}
    	\end{subfigure}        \caption{A twisted octahedron.}
    	\label{fig:identified_octa}
    \end{figure} 

	This results in a \emph{twisted octahedron} whose boundary consists of four \textit{leaves}. A leaf is an once-punctured ideal bigon with two interior edges and two boundary edges. See Figure~\ref{fig:twsitedOctaLeaf}. We follow  \cite{murakami_kashaevs_2001} for the terminology.
	
	For a segment of $D$ we glue two leaves coming from each end of the segment so that the interior and boundary edges of a leaf are identified with those of the other leaf, respectively.  Performing the identification for every segment of $D$, $N$ twisted octahedra form $S^3\setminus(K\cup\{\textrm{two points} \})$. We denote this octahedral decomposition $\{o_1,\cdots,o_N\}$  by $\Ocal_D$. We stress that an octahedral decomposition works for any knot diagram and a link diagram without a component which has only over- or under-passing crossings.
	\begin{figure}[!h]
		\centering
		\includegraphics[scale=0.4]{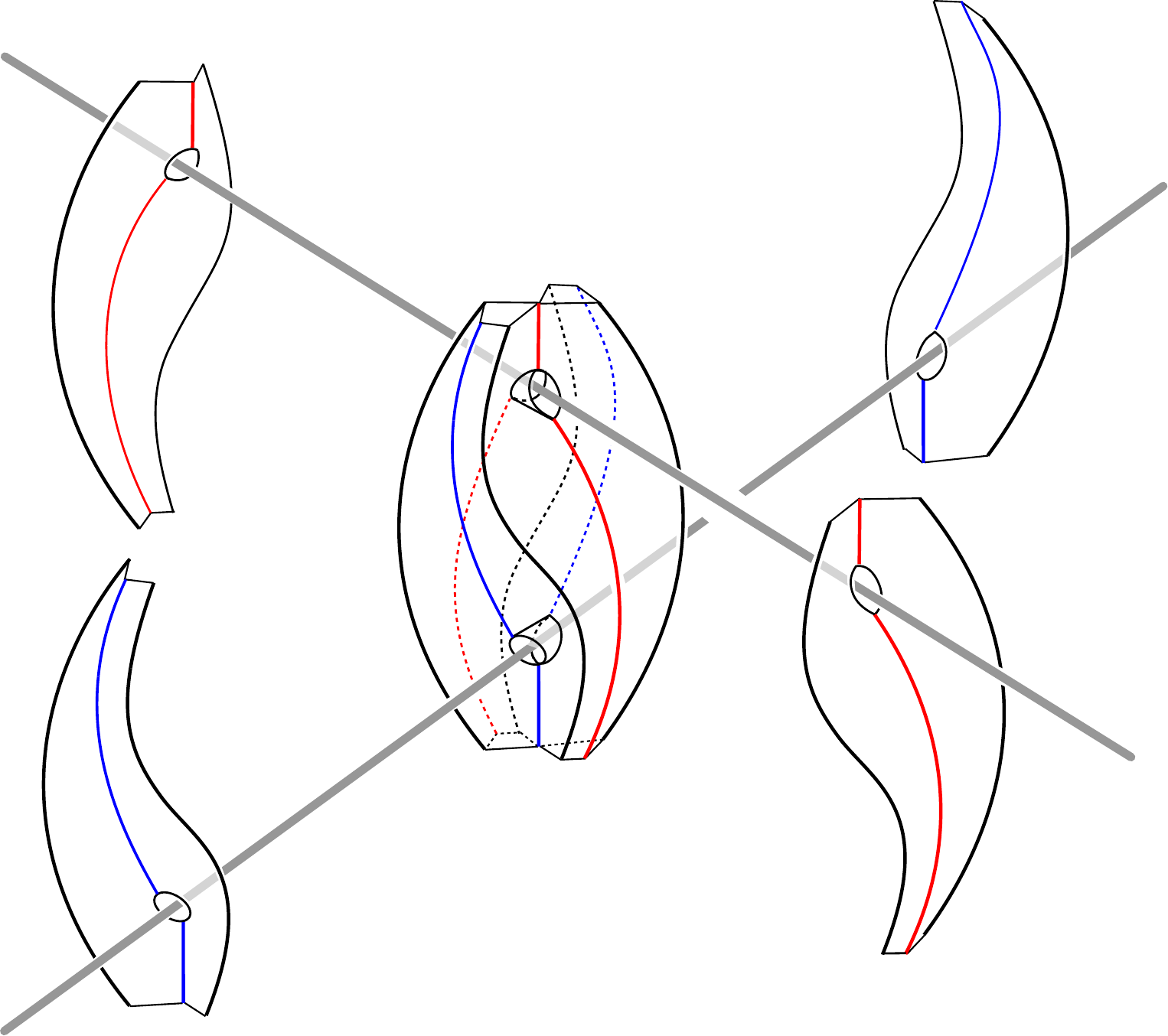}
		\caption{A twisted octahedron with four leaves.}
		\label{fig:twsitedOctaLeaf}
	\end{figure} 

	 There are three ideal points of $\Ocal_D$. We denote the ideal point corresponding to the knot by $p^\circ$. We also denote the one above $D$ by $p^{+}$ and the one below $D$ by $p^-$. There are three types of edges of $\Ocal_D$ : an edge joining  $p^\circ$ and $p^+$(resp., $p^-$) is called an \emph{over-edge}(resp., \emph{under-edge}) and an edge joining $p^+$ and $p^-$ is called a \emph{regional edge}. We remark that over-edges, under-edges and regional edges are in one-to-one correspondence with over-arcs, under-arcs and regions of $D$, respectively.
	    	\begin{defn} Ideal triangulations of $S^3 \setminus (K \cup \{ \textrm{two points}\})$ obtained by dividing each octahedron in $\Ocal_D$ into four and five tetraheda as in Figure \ref{fig:four_five} are called the \emph{four-term triangulation} $\Tcal_{4D}$ and the \emph{five-term triangulation} $\Tcal_{5D}$, respectively. For both cases the additional edges created in the division are called \emph{central edges}. 
		    \end{defn}
	    
		 \begin{figure}[H]
		 	\centering
		 	\tikzset{
		 		commutative diagrams/.cd,
		 		arrow style=tikz,
		 		diagrams={>=stealth}}
		 	\begin{tikzcd}[column sep=2em]	
		 	\vcenter{\hbox{\includegraphics{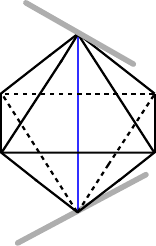}}} \ar[r]  &  \vcenter{\hbox{\includegraphics{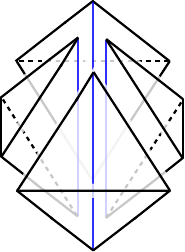}}} &  & 
		 	\vcenter{\hbox{\includegraphics{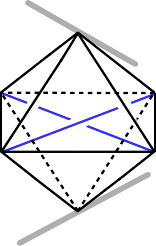}}}\ar[r]  &  \vcenter{\hbox{\includegraphics{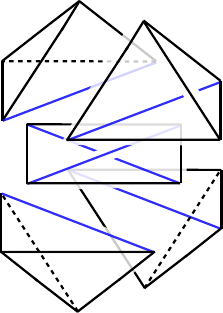}}}
			\end{tikzcd}
		 \caption{Triangulations of an octahedron.}
		 	\label{fig:four_five}
		 \end{figure}

	   \begin{rmk} \label{rmk:five_term} 
	 	We remark that the four term triangulation is related to   the Kashaev invariant  and the five term triangulation to the colored Jones polyonomial  as mentioned in the introduction. Each tetrahedron corresponds to a $q$-series  or a quantum factorial term of the $R$-matrices (\cite{yokota_potential_2002},\cite{CM_2013}).
	 \end{rmk}

		We finally recall the \emph{link} $l(p^\circ)$ of the ideal point $p^\circ$ of $\Ocal_D$. The top and bottom vertices of $o_1, \cdots, o_N$ are identified to  $p^\circ$ and hence each $o_k$ contributes two quadrilaterals, the quadrilaterals near $A$ and $F$ respectively in Figure~\ref{fig:octaCrossing}, to $l(p^\circ)$. We denote the one in the upper-pyramid of $o_k$ by $l_k$ and the other one in the lower-pyramid by $\widehat{l}_k$. Along $D$, these $2N$ quadrilaterals $l_1,\widehat{l}_1,\cdots,l_N,\widehat{l}_N$ form the link $l(p^\circ)$.
		
		\begin{exam}[Figure-eight knot] \label{ex:figure_eight} For a diagram of the figure-eight knot given as in Figure \ref{fig:figure_eight}, the link $l(p^\circ)$ consists of  $8$ quadrilaterals $l_1,\widehat{l}_1,\cdots,l_4,\widehat{l}_4$, which is indeed a torus. Here $a_i$'s indicate the edge identifications. 

		\begin{figure}[!h]
		 	\centering
		 	\scalebox{1}{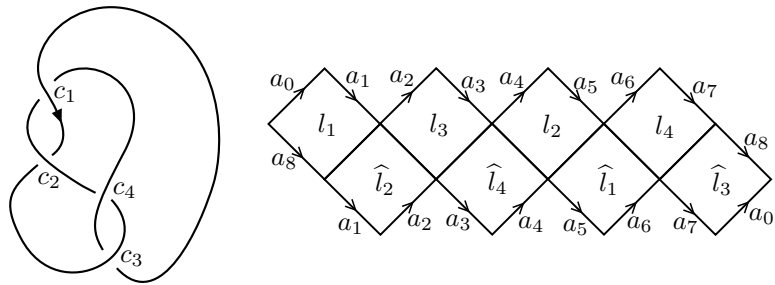} 
		 	\caption{The figure-eight knot diagram and $l(p^\circ)$.}
	 	\label{fig:figure_eight}			
		\end{figure}
		\end{exam}
	
		\begin{rmk} For the link $l(p^{+})$, the octahedron $o_k$ contributes a pair of quadrilaterals $l_k^+$ which is the pair of quadrilaterals near $C$ and $E$ in Figure \ref{fig:identified_octa}(b). One can verify that $l(p^{+})$ (and similarly $l(p^{-})$) is a 2-sphere by putting every $l^{+}_k$ on a plane and expanding them so that they cover the entire plane as in Figure~\ref{fig:figure_eight_bubble}.
		\begin{figure}[!h]
		 	\centering
		 	\scalebox{1}{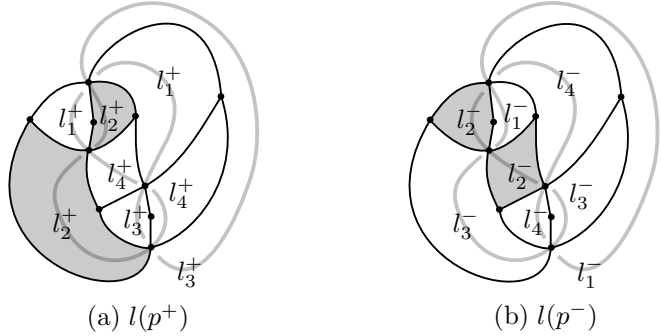} 
		 	\caption{$l(p^{+})$ and $l(p^{-})$ viewed from $p^{+}$ and $p^{-}$, respectively.}
		 	\label{fig:figure_eight_bubble}			
		\end{figure}
		\end{rmk}
 \subsection{Thurston's gluing equations} \label{subsec:gluing}
         From now on we consider each of $o_1, \cdots, o_N$ to be a hyperbolic ideal tetrahedron so that the vertices of $o_k$ lie in  $\Cbb \cup \{\infty \}$. Throughout the paper, we will use indices of the segments and regions around a crossing $c_k$ as in Figure \ref{fig:reg_crossing}.          	
         	
         	\begin{figure}[!h]
         		\centering
         		\scalebox{1}{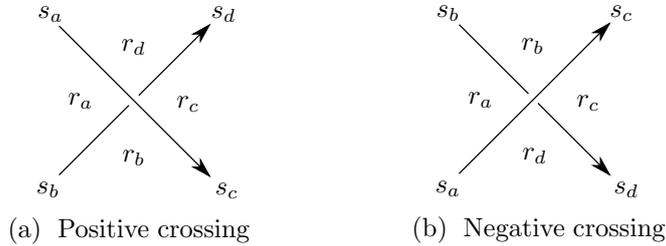}
         		\caption{Indices of segments and regions around a crossing $c_k$. }
         		\label{fig:reg_crossing}
         	\end{figure} 
         	
			\begin{notat} \label{notat:edge} (a)  We denote the upper-hypotenuses of $o_k$ by $\alpha_k,\beta_k,\gamma_k,$ and $\delta_k$  counter-clockwise (resp., clockwise) from the direction of the incoming over-arc of a positive (resp., negative) crossing  $c_k$. We also denote the lower-hypotensuses of $o_k$ by $\widehat{\alpha}_k,\widehat{\beta}_k,\widehat{\gamma}_k,$ and $\widehat{\delta}_k$ clockwise (resp., counter-clockwise) from the direction of the incoming under-arc. See Figure \ref{fig:edge_notation}.\\ 
			(b) We denote the side edge of $o_k$ corresponding to a region $r_j$ by $\tau_{k, j}$. We often denote a side edge by $\tau$ if we don't need to specify the indices.\\
			(c) For an edge of $o_k$ we consider the tetrahedron determined by two faces of $o_k$ containing the edge and the cross-ratio of the tetrahedron associated to the edge is said to be the  \emph{cross-ratio of the edge in $o_k$}.
			Abusing notation, we will denote the cross-ratio of an edge in $o_k$ by the same symbol as the edge itself. For instance,  $\alpha_k = [A,B,E,C]$ and $\tau_{k,b}=[C,D,F,A]$ for the left octahedron in Figure \ref{fig:edge_notation}.
				\begin{figure}[!h]
					 	\centering	
					 	\begin{subfigure}{0.48\textwidth}
					 		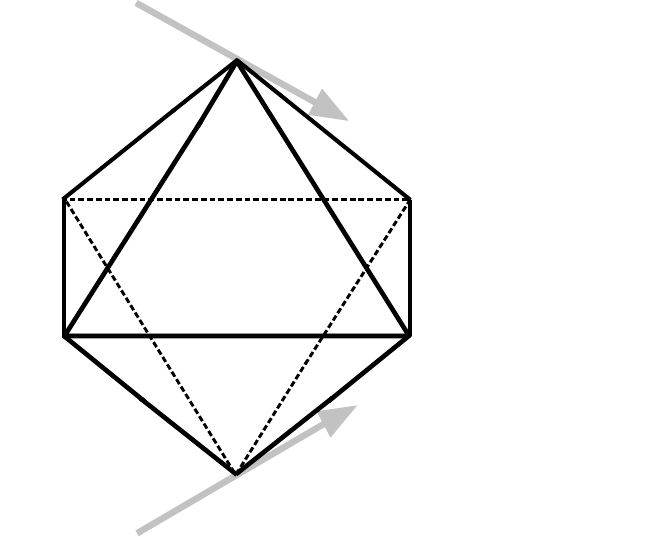
					 		\caption{Positive crossing} 				 		
					 	\end{subfigure}		 	
				 		\begin{subfigure}{0.48\textwidth}
				 			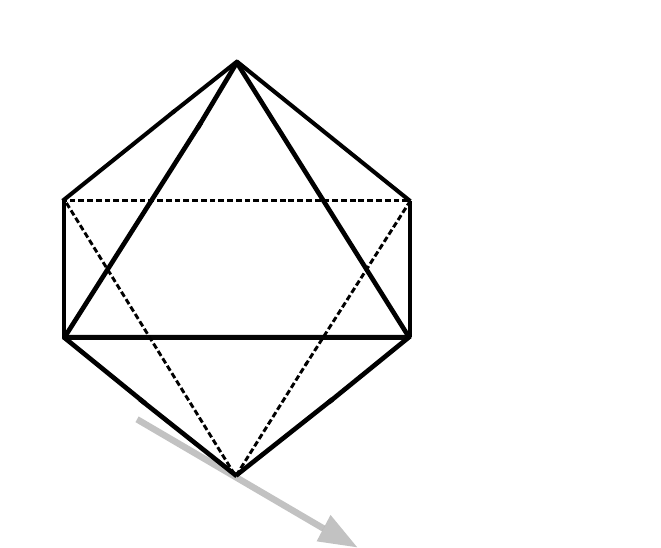
				 			\caption{Negative crossing}
				 		\end{subfigure}					 	
					 	\caption{Edge-notation of $o_k$ for Figure \ref{fig:reg_crossing}.}
					 	\label{fig:edge_notation}			
				\end{figure}
           \end{notat}
         
           Every hyperbolic tetrahedron in this paper is always assumed to be non-degenerate. We therefore always assume the cross-ratios $\tau_{k,a}$, $\tau_{k,b}$,  $\tau_{k,c}$, and $\tau_{k,d}$ are not $0,1$ or $\infty$ for the four-term triangulation $\Tcal_{4D}$, since the octahedron $o_k$ in $\Tcal_{4D}$ divides into tetrahedra of these cross-ratios. Similarly, we assume $\alpha_k$, $ \gamma_k$,  $ \widehat{\alpha}_k$, $ \widehat{\gamma}_k, $ or $ (  \alpha_k   \gamma_k)^{-1}$ are not $0,1$ and $\infty$ for the five-term triangulation $\Tcal_{5D}$. We refer to these condition as \emph{non-degeneracy conditions}.

          Now we describe Thurston's gluing equations for $\Tcal_{5D}$. (The case for $\Tcal_{4D}$ is the same.) 
          If each $o_k$ is an hyperbolic ideal octahedron,  then  the cross-ratios of the tetrahedra in  $\Tcal_{5D}$ automatically satisfy the gluing equation for every central edge, and vice versa.   		 
          We can  express  Thurston's gluing equation for $\Tcal_{5D}$ in terms of  $\alpha_k$, $ \gamma_k$,  $ \widehat{\alpha}_k$, and $ \widehat{\gamma}_k$,
         since the other cross-ratios in $o_k$ are determined easily by these four variables as follows : 
	\begin{equation}\label{eqn:tau_a}
    	\begin{array}{cc}
		\left\{
		\begin{array}{rcl}
		\beta_k &=& (  \alpha_k)' (  \gamma_k)'',\ ~~~ \widehat{\beta}_k ~~=~~ (  \widehat{\alpha}_k)' (  \widehat{\gamma}_k)'' \\[5pt]
\delta_k &=& (  \alpha_k)'' (  \gamma_k)',\ ~~~ \widehat{\delta}_k ~~=~~ (  \widehat{\alpha}_k)'' (  \widehat{\gamma}_k)' \\[5pt]
			  \tau_{k,a} &=&  (  \alpha_k)''\big(  (\alpha_k\gamma_k)^{-1}\big)' (  \widehat{\alpha}_k)'' \\[5pt]
			  \tau_{k,c} &=&  (  \gamma_k)''\big(  (\alpha_k\gamma_k)^{-1}\big)' (  \widehat{\gamma}_k)'' \\[5pt]
			  \tau_{k,b} &=&  (  \gamma_k)'\big(  (\alpha_k\gamma_k)^{-1}\big)'' (  \widehat{\alpha}_k)' \\[5pt]
			  \tau_{k,d} &=&  (  \alpha_k)'\big( (\alpha_k\gamma_k)^{-1}\big)'' (  \widehat{\gamma}_k)'
		\end{array}
		\right.
        & ~~~~\textrm{ for Figure \ref{fig:reg_crossing}(a)}
        \end{array}
	\end{equation}
	\begin{equation}\label{eqn:tau_b}
    	\begin{array}{cc}
		\left\{
		\begin{array}{rcl}
		\beta_k &=& (  \alpha_k)'' (  \gamma_k)',\ ~~~ \widehat{\beta}_k ~~=~~ (  \widehat{\alpha}_k)'' (  \widehat{\gamma}_k)' \\[5pt]
		\delta_k &=& (  \alpha_k)' (  \gamma_k)'',\ ~~~ \widehat{\delta}_k ~~=~~ (  \widehat{\alpha}_k)' (  \widehat{\gamma}_k)'' \\[5pt]
		\tau_{k,a} &=&  (  \alpha_k)'\big(  (\alpha_k\gamma_k)^{-1}\big)'' (  \widehat{\alpha}_k)' \\[5pt]
		\tau_{k,c} &=&  (  \gamma_k)'\big(  (\alpha_k\gamma_k)^{-1}\big)'' (  \widehat{\gamma}_k)' \\[5pt]
		\tau_{k,b} &=&  (  \gamma_k)''\big(  (\alpha_k\gamma_k)^{-1}\big)' (  \widehat{\alpha}_k)'' \\[5pt]
		\tau_{k,d} &=&  (  \alpha_k)''\big( (\alpha_k\gamma_k)^{-1}\big)' (  \widehat{\gamma}_k)''
		\end{array}
		\right.
        & ~~~~\textrm{ for Figure \ref{fig:reg_crossing}(b)}
        \end{array}
	\end{equation}
 We remark that  there is a symmetry in the above cross-ratio expressions, which exchanges $'$ and $''$ when the crossing sign changes. 

We have the following obvious identities. 
	 \begin{lem}\label{lem:dih_lem} (a) $   \alpha_k   \gamma_k =   \widehat{\alpha}_k  \widehat{\gamma}_k$.\\[2pt]
	(b) $   \alpha_k   \beta_k  \gamma_k   \delta_k =    \widehat{\alpha}_k  \widehat{\beta}_k   \widehat{\gamma}_k \widehat{\delta}_k =1$. \\[2pt]
	(c) $   \tau_{k,a}   \tau_{k,b}  \tau_{k,c}  \tau_{k,d}=1$ where $a,b,c$ and $d$ are the indices of the regions around $c_k$. 
	 \end{lem}
	 \begin{proof}
		  The cross-ratios of the tetrahedra of $\Tcal_{5D}$ in the middle of $o_k$ associated to the central edges are $(\alpha_k \gamma_k)^{-1}$ and $(\widehat{\alpha}_k \widehat{\gamma}_k)^{-1}$. Since the central edges are in opposite positions in the tetrahedron, we have (a). The equalities (b) and (c)  directly follow from equations  (\ref{eqn:tau_a}) and (\ref{eqn:tau_b}).
	 \end{proof}
 
	We know that the gluing equations for the central edges are already satisfied and  
	thus only consider the gluing equation for the other edges of $\Tcal_{5D}$: the regional edges, over-edges and under-edges of $\Ocal_D$. \\[4pt]
	 (1) For a regional edge of $\Ocal_{D}$ let $r_j$ be the corresponding region of $D$. To simplify notation,  let $c_{1}, c_{2}, \cdots, c_{n}$ be the corner crossings of $r_j$. Then each octahedron $o_{k}$ contributes the cross-ratio $\tau_{k,j}$ of the regional edge for $1 \leq k \leq n$ and hence its gluing equation is

      \begin{equation}\label{eqn:region_edge}
      	  \displaystyle\prod_{k=1,\cdots,n} \tau_{k,j}  =1.
      \end{equation} Plugging equations (\ref{eqn:tau_a}) and (\ref{eqn:tau_b}) to equation (\ref{eqn:region_edge}), we have the gluing equation for the regional edge in  $\alpha_k,\gamma_k,\widehat{\alpha}_k$, $\widehat{\gamma}_k$. \\[4pt]
	  (2) For an over-edge of $\Ocal_{D}$, suppose the corresponding over-arc of $D$ passes through crossings $c_{1},c_{2}, \cdots, c_{{n+1}}$ as in Figure \ref{fig:over_edge}. Then the gluing equation for the over-edge is 
	  \begin{equation}\label{eqn:over_edge}
	  		\begin{array}{rcl}
	  		  1&=& \widehat{\gamma}_1  \beta_2  \delta_2 \cdots  \beta_{n}  \delta_{n}   \widehat{\alpha}_{n+1} \\[2pt]
	&=&	  		  \widehat{\gamma}_1  ( \alpha_2  \gamma_2)^{-1} \cdots ( \alpha_{n}  \gamma_{n})^{-1}   \widehat{\alpha}_{n+1}.
			\end{array}
	  \end{equation} Here the last equality follows from Lemma \ref{lem:dih_lem}(b).
	  \begin{figure}[!h]
	 	\centering
	 	\scalebox{1}{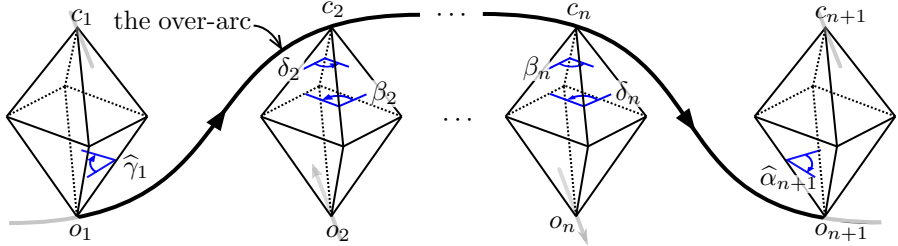}
	 	\caption{Cross-ratios around an over-edge.}
	 	\label{fig:over_edge}
	 \end{figure}
 
		\noindent (3) For an under-edge of $\Ocal_D$, suppose the corresponding under-arc passes through crossings $c_1,c_2,\cdots ,c_{n+1}$. Then the gluing equation for the under-edge is similarly given by 
	\begin{equation}\label{eqn:under_edge}
		\begin{array}{rcl}
		  1&=&\gamma_1  \widehat{\beta}_2  \widehat{\delta}_2 \cdots  \widehat{\beta}_{n}  \widehat{\delta}_{n}   \alpha_{n+1}\\[5pt]
	 &=&\gamma_1  ( \widehat{\alpha}_2  \widehat{\gamma}_2)^{-1} \cdots ( \widehat{\alpha}_{n}  \widehat{\gamma}_{n})^{-1}   \alpha_{n+1}.
	 		\end{array}
	\end{equation}
	 \begin{rmk} \label{rmk:one_segm} For an under-edge corresponding to an under-arc consisting of a single segment as in Figure \ref{fig:hyp_segm}(c), the gluing equation is $ \gamma_k    \alpha_{k+1}=1$. Similarly, for an over-edge corresponding to an over-arc consisting of a single segment as in Figure \ref{fig:hyp_segm}(d), we have $\widehat{\gamma}_k   \widehat{\alpha}_{k+1} =1$.
	 \end{rmk}
	 
	 \begin{defn} The set of hyperbolic ideal octahedra $\{o_1, \cdots ,o_N\}$ is said to be a \emph{solution} to $\Tcal_{5D}$(resp., $\Tcal_{4D}$) if all tetrahedra of $\Tcal_{5D}$(resp.,  $\Tcal_{4D}$) are non-degenerate and their cross-ratios satisfy the gluing equation for every edge of $\Ocal_D$.
	 \end{defn}

	 A solution to $\Tcal_{4D}$ or $\Tcal_{5D}$ induces a pseudo-hyperbolic structure  on 
	 $M=S^3\setminus K$ and a holonomy representation $\rho: \pi_1(M) \rightarrow \psl$.

	 \begin{defn}  We say that a  solution to $\Tcal_{4D}$ or $\Tcal_{5D}$ is an  $m$\emph{-deformed solution} for some $m \in \Cbb \setminus \{0 \}$ if $\rho(\mu)$ is conjugate to ${\begin{pmatrix} \sqrt{m} & 1 \\ 0 &\sqrt{m}^{-1} \end{pmatrix}}$ where $\mu$ is a meridian of the knot. When $m=1$, an $1$-deformed solution will be called  \emph{boundary parabolic solution}.
	 \end{defn}
	 We remark that the value $m\in \Cbb \setminus \{0\}$ in the definition is determined up to $z \mapsto 1/z$. 

	\section{Pseudo-developing maps and variables} \label{sec:dev}
	In this section, we describe a pseudo-developing map for an $m$-deformed solution and present the defining equations of an $m$-deformed solution to $\Tcal_{4D}$ and $\Tcal_{5D}$ in complex variables related to the segments and regions of a diagram $D$.
  \subsection{Pseudo-developing maps and deformed solutions}\label{subsec:developing}
  
  Let hyperbolic octahedra $\{o_1,\cdots,o_N\}$ be an $m$-deformed solution to $\Tcal_{4D}$ or $\Tcal_{5D}$. As we have seen in Section \ref{sec:pseudo_dev},   the solution gives a pseudo-hyperbolic structure $(\Dcal,\rho)$ on $M=S^3 \setminus K$, where $\Dcal : \Mth \rightarrow \overline{\Hbb^3}$ is a developing map and $\rho : \pi_1(M) \rightarrow \psl$ is a holonomy representation. For notational convenience we often confuse an object in $\Mth$ and its developing image under $\Dcal$.

  We first fix the base point $P_1$ of $\pi_1(M)$ to be a point in $l_1$. Recall that the quadrilateral $l_1$ is the intersection of the upper-pyramid of $o_1$ with the link $l(p^{\circ})$. Thus when we consider the Wirtinger generator whose usual base point is  ``$\infty$'', we conjugate it by a path joining $P_1$ to ``$\infty$'' passing through the right-hand side of the over-arc of $c_1$. See the left side of Figure~\ref{fig:figure_eight_dev}. 

  We will depict a developing map $\Dcal$, especially along the particular loop $\lambda_\circ$ which is obtained by pushing the knot $K$ parallelly to the right-hand side.  We homotope $\lambda_\circ$ to a loop in $l(p^\circ)$ passing through the $l_k$'s and $\widehat{l}_k$'s consecutively as in Figure~\ref{fig:figure_eight_dev}. (The reader should not pay atttention to the points $P_k$ and $\widehat{P}_k$ besides $P_1$ in Figure~\ref{fig:figure_eight_dev}. We will specify these points in Section \ref{sec:Holonomy}.)
	 \begin{figure}[H]
	 	\centering
	 	\scalebox{1}{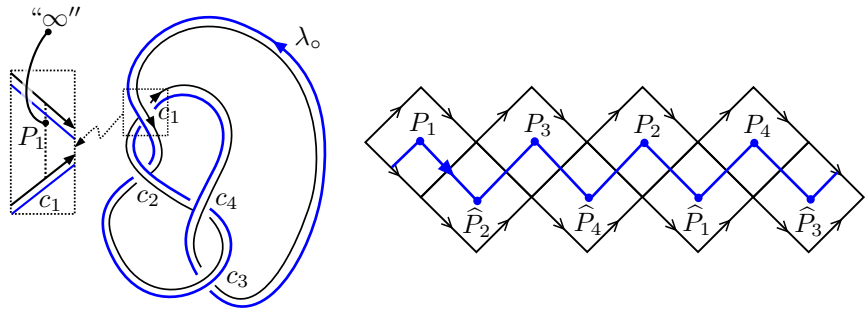}
	 	\caption{Loop $\lambda_\circ$ in $l(p^\circ)$ for Example \ref{ex:figure_eight}.}
	 	\label{fig:figure_eight_dev}
	 \end{figure}	A developing map $\Dcal$ is uniquely determined whenever an initial octahedron is placed in $\overline{\Hbb^3}$. We take $o_1$ as an initial octahedron and place it in $\overline{\Hbb^3}$ so that the top vertex of $o_1$ maps to $\infty \in \Cbb \cup \{ \infty \}$. (If we consider the four-term triangulation, then we further assume that the bottom vertex of $o_1$ maps to the origin of $\Cbb \cup \{\infty\}$.) In the developing map $\Dcal$ along $\lambda_\circ$, each octahedron $o_k$ appears twice as hyperbolic octahedra $O_k$ and $\widehat{O}_k$ when $\lambda_\circ$ passes through $l_k$ and $\widehat{l}_k$, respectively. We remark that $\widehat{O}_k$ can be obtained by ``flipping'' $O_k$ via $z \mapsto \frac{1}{z}$ and applying an appropriate similarity. Note that the ``bottom'' vertex of $\widehat{O}_k$ corresponds to the top vertex of $o_k$. 
	 
	\begin{figure}[!h]
		\centering
		\scalebox{1}{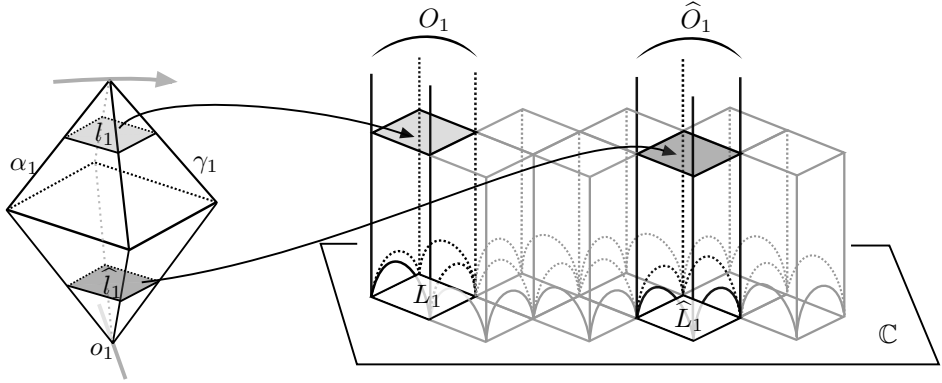}
		\caption{A developing map along $\lambda_\circ$ for Example \ref{ex:figure_eight}.}
		\label{fig:figure_eight_OctaDev}
	\end{figure}   
	
	 We now choose a meridian $\mu \in \pi_1(M)$ as the Wirtinger generator winding the over-arc of the crossing $c_1$. One can check that  $\mu$  is homotopic to the diagonal loop in $l_1$ joining $\delta_1$ to $\beta_1$ after the base point change conjugation as described at the very beginning of this subsection. Therefore, when we choose the lifting of the base point $P_1$ in $O_1$, the holonomy action of $\mu$ sends the hypotenuse of $O_1$ corresponding to  $\beta_1$(resp., $\delta_1$) to that of $\delta_1$(resp., $\beta_1$), if $c_1$ is positive (resp., negative). In particular, it fixes the top vertex of $O_1$, $\infty$, and hence is a similarity. Restricting the similarity to the boundary plane $\Cbb$, we obtain Figure~\ref{fig:figure_eight_hol}. 
	 Here the quadrilateral $L_k$(resp., $\widehat{L}_k$) is the projection image of $O_k$(resp., $\widehat{O}_k$) to $\Cbb$ as in Figure~\ref{fig:figure_eight_OctaDev}.
 	\begin{figure}[!h]
 		\centering
 		\scalebox{1}{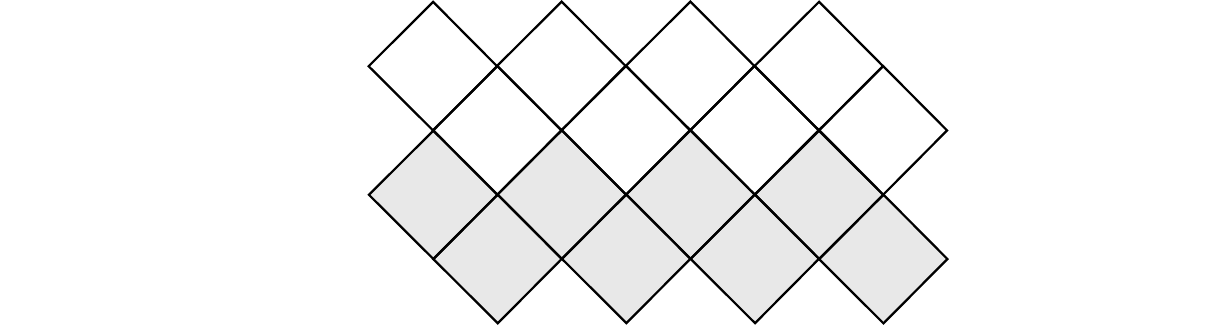}
 		\caption{The similarity map for $\mu$ on $\Cbb$.}
 		\label{fig:figure_eight_hol}
 	\end{figure}
 	Since the given solution is $m$-deformed, the similarity map for  $\mu$ has scaling factor either $m$ or $1/m$. Suppose the similarity has scaling factor $m$. We then obtain relations on $\alpha_k, \gamma_k, \widehat{\alpha}_k$ and $\widehat{\gamma}_k$ from the configuration of $L_k$'s and $\widehat{L}_k$'s. For instance, we have $\gamma_1^{-1}  \widehat{\alpha}_2= m$ and $\widehat{\gamma}_1 \alpha_4^{-1} = m^{-1}$ from Figure \ref{fig:figure_eight_hol}. Considering general cases, we obtain $   \gamma_k^{-1} \widehat{\alpha}_{k+1} = m $ for a segment as in Figure \ref{fig:hyp_segm}(a) and $  \widehat{\gamma}_k\alpha_{k+1}^{-1}  = m^{-1}$ for a segment as in Figure \ref{fig:hyp_segm}(b). Together with Remark \ref{rmk:one_segm}, we obtain a single equation from each segment of $D$ :
	 \begin{equation}\label{eqn:hyp_angle}
	 	\left\{
	 	\begin{array}{cccc}
	    \gamma_k^{-1}   \widehat{\alpha}_{k+1} &=& m & \textrm{for Figure \ref{fig:hyp_segm}(a)} \\[5pt]
	     \widehat{\gamma}_k    \alpha_{k+1}^{-1} &=& m^{-1} & \textrm{for Figure \ref{fig:hyp_segm}(b)} \\[5pt]
	   (  \gamma_k  \alpha_{k+1})^{-1} &=& 1 & \textrm{for Figure \ref{fig:hyp_segm}(c)} \\[5pt]
	     \widehat{\gamma}_k     \widehat{\alpha}_{k+1} &=& 1 & \textrm{for Figure \ref{fig:hyp_segm}(d)}
	 	\end{array}
	 	\right.
	 \end{equation}	 
	 \begin{figure}[!h]
	 	\centering
	 	\scalebox{1}{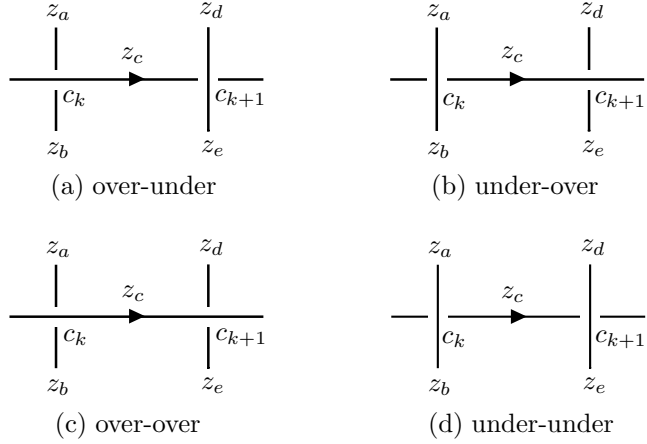}
	 	\caption{Four types of a segment.}
	 	\label{fig:hyp_segm}	 	
	 \end{figure}
	 We call equation (\ref{eqn:hyp_angle}) the $m$\textit{-hyperbolicity equation for a segment}, omitting ``$m$-'' if $m=1$.
	 \begin{prop} The gluing equation for an over(resp., under)-edge of $\Ocal_D$ agrees with the product of the $m$-hyperbolicity equations (\ref{eqn:hyp_angle}) for all segments contained in the corresponding over(resp., under)-arc of $D$.
	 \end{prop}	 
	 \begin{proof}
	 		 Comparing the gluing equations (\ref{eqn:over_edge}) and (\ref{eqn:under_edge}) with equation (\ref{eqn:hyp_angle}), we conclude the proposition.
	 \end{proof}
	 In conclusion, hyperbolic ideal octahedra $\{o_1, \cdots, o_N\}$ form an $m$-deformed solution if and only if their cross-ratios satisfy the gluing equation (\ref{eqn:region_edge}) for every regional edge of $\Ocal_D$ and the $m$-hyperbolicity equation (\ref{eqn:hyp_angle}) for every segment of $D$. 
	 We denote the set of all gluing equations of regional edges of $\Ocal_D$ by $E_R$ and the set of all $m$-hyperbolicity equations of segments of $D$ by $E_{S;m}$.
	 \begin{defn}[Re-definition of an $m$-deformed solution] The set of hyperbolic ideal octahedra $\{o_1, \cdots, o_N\}$ is an $m$\textit{-deformed solution} to $\Tcal_{5D}$(resp., $\Tcal_{4D}$) if all tetrahedra of $\Tcal_{5D}$(resp., $\Tcal_{4D}$) are non-degenerate and their cross-ratios satisfy both $E_R$ and $E_{S;m}$.
	 \end{defn} 
	 \begin{prop}\label{prop:redundant} Each $E_R$ and $E_{S;m}$ have a redundant equation.
	 \end{prop}
	 \begin{proof} By Lemma \ref{lem:dih_lem}(c), the product of all $\tau$'s is $1$ and hence the product of all equations in $E_R$ is trivial. On the other hand, each $(\alpha_k)^{-1}$, $\widehat{\alpha}_k$, $(\gamma_k)^{-1}$ and $\widehat{\gamma}_k$ appears exactly once in $E_{S;m}$. Thus the product of all equations in $E_{S;m}$ is also trivial by Lemma \ref{lem:dih_lem}(a).
	 \end{proof}
	\subsection{Segment variables} \label{subsec:segment}
   	 Let hyperbolic ideal octahedra $\{o_1,\cdots,o_N\}$ be an $m$-deformed solution to the four-term triangulation $\Tcal_{4D}$. We describe  the coordinates of vertices of $o_k$ as follows. Here we mean the coordinate by a value of $\Cbb \cup \{\infty\}=\partial \overline{\Hbb^3}$.
	\begin{notat}\label{notat:octa_coord} We denote the coordinates of the top and the bottom vertices of $o_k$ by $z_{k,T}$ and $z_{k,B}$, respectively. The side vertices  of $o_k$ are in one-to-one correspondence with the segments around the crossing $c_k$. We thus denote the coordinates of side vertices by $z_{k,i_1}, z_{k,i_2}, z_{k,i_3}$ and $z_{k,i_4}$ where $i_1,i_2,i_3$ and $i_4$ are the indices of the segments around $c_k$. 
	\end{notat} 
	\begin{prop}\label{prop:seg} For a ($m$-deformed) solution $\{o_1,\cdots,o_N\}$ there are unique coordinates of vertices of $o
_1, \cdots, o_N$ up to scalar multiplication satisfying (a) $z_{k,B}=0$ and $z_{k,T}=\infty$ for all $1 \leq k \leq N$;	(b) For each segment $s_i$ of $D$ we have $z_{k_1,i}=z_{k_2,i}$ where $k_1$ and $k_2$ are the indices of the crossings attached to $s_i$.
	\end{prop}
	\begin{proof} 
	Since the central edge of $o_k$ in $\Tcal_{4D}$ joins the top and bottom vertices, we have $z_{k,B} \neq z_{k,T}$ and may assume $z_{k,B}=0$ and $z_{k,T}=\infty$, respectively, for all $ 1 \leq k \leq N$. Fixing $z_{k,B}=0$ and $z_{k,T}=\infty$, we can only change the coordinates of side vertices of $o_k$ by multiplying by a non-zero complex number simultaneously. Now we choose $o_1$ as an initial octahedron and fix the coordinates of the side vertices of $o_1$. Then the coordinates of any adjacent octahedra to $o_1$ determined by the condition (b). (Here we say two octahedra are adjacent if the corresponding crossings are connected by a segment of $D$.) Continuing the determination, we can fix every coordinates of octahera $o_1,\cdots,o_N$. However, it remains to prove well-definedness of the coordinates through this process.

	For a region $r_j$ of $D$ suppose $c_1, \cdots, c_n$ are the corner crossings of $r_j$ in counter-clockwise order. Let $s_{i_k}$ be the segment in the boundary of $r_j$ joining crossings $c_k$ and $c_{k+1}$ for $1 \leq k \leq n$. (The index $k$ is taken modulo $n$.) Then a simple cross-ratio computation gives $  \tau_{k,j} = \frac{z_{k,i_{k-1}}}{z_{k,i_{k}}}$ and thus the gluing equation (\ref{eqn:region_edge}) of the regional edge corresponding to $r_j$ is $$ \dfrac{z_{1,i_{n}}}{z_{1,i_{1}}} \cdot \dfrac{z_{2,i_{1}}}{z_{2,i_{2}}} \cdots \dfrac{z_{n,i_{n-1}}}{z_{n,i_{n}}} =1.$$ This implies that conditions (b) for the segments $s_{i_1},\cdots,s_{i_n}$ are compatible, i.e., if one determines the coordinates of octahedra from $o_1$ to $o_n$ by condition (b) for $s_{i_1}, \cdots, s_{i_{n-1}}$, which is $z_{k,i_k} = z_{k+1,i_k}$ for $k=1, \cdots, n-1$, then the gluing equation of the regional edge corresponding to $r_j$ implies that the condition (b) for $s_{i_n}$ is automatically satisfied.  
	Now the determination of coordinates of octahedra is independent of the choice of a path from $o_1$ since it is consistent along every cycle bounding a region. This proves the proposition and uniqueness is clear from the proof.
	\end{proof}
	From condition (b), one can record the coordinates of octahedra on segments of $D$ by assigning the common value given in (b), i.e., $z_i=z_{k,i}$ on a segment $s_i$ where $k$ is the index of a crossing attached to $s_i$. Therefore, hyperbolic octahedra $o_1, \cdots, o_N$ of the following description are enough when we consider a solution to $\Tcal_{4D}$.
	\begin{itemize}
		\item Assign a variable $z_i$ on each segment $s_i$ of $D$.
		\item The coordinates of vertices of $o_k$ is determined by $z$-variables : $z_{k,B}=0$, $z_{k,T}=\infty$ and $z_{k,i}=z_i$ where $i$ is the index of a segment attached to $c_k$. 
	\end{itemize} 
	 From the construction,  these hyperbolic octahedra $o_1, \cdots o_N$ automatically satisfy $E_R$. Therefore, they form an $m$-deformed solution if and only if they satisfy $E_{S;m}$. One can express equation (\ref{eqn:hyp_angle}) in $z$-variables as follows.
	\begin{equation} \label{eqn:hyp_z}
	 	\left\{
	 	\begin{array}{cccc}
	   \dfrac{z_c-z_a}{z_c-z_b} \cdot \dfrac{z_d(z_c-z_e)}{z_e(z_c-z_d)} &=& m & \textrm{for Figure \ref{fig:hyp_segm}(a)} \\[12pt]
	  \dfrac{z_a(z_c-z_b)}{z_b(z_c-z_a)} \cdot \dfrac{z_c-z_d}{z_c-z_e} &=& m & \textrm{for Figure \ref{fig:hyp_segm}(b)} \\[12pt]
	  \dfrac{z_c-z_a}{z_c-z_b} \cdot \dfrac{z_c-z_e}{z_c-z_d} &=& 1 & \textrm{for Figure \ref{fig:hyp_segm}(c)} \\[12pt]
	  \dfrac{z_a(z_c-z_b)}{z_b(z_c-z_a)} \cdot \dfrac{z_e(z_c-z_d)}{z_d(z_c-z_e)}    &=& 1 & \textrm{for Figure \ref{fig:hyp_segm}(d)}
	 	\end{array}
	 	\right.
	\end{equation}
	We finally check that the non-degeneracy condition is converted to ``the variable $z_i$ is non-zero and two adjacent $z$-variables are distinct''. Here we say two $z$-variables are adjacent if the corresponding segments share a corner in $D$.
	\begin{defn}[Segment variables] A non-zero variable $z_i$ assigned to a segment $s_i$ of $D$ is called a \emph{segment variable}. We say $z=(z_1, \cdots ,z_{2N} )$ is an \textit{$m$-deformed solution to $\Tcal_{4D}$} if (a) it satifies the $m$-hyperbolicity equation (\ref{eqn:hyp_z}) for every segment of $D$ and (b) each pair of adjacent segment variables is distinct. 
	\end{defn}
	We often confuse a segment of $D$ and the corresponding segment variable.

	\begin{exam}[Figure-eight knot]\label{ex:fig_seg_var}
	Let $s_i$ be a segment of the figure-eight knot diagram labeled as in Figure \ref{fig:FigureEightSeg} and let $z_i$ be a segment variable assigned to $s_i$ for $1 \leq i\leq 8$.
	\begin{figure}[!h]
		\centering
		\scalebox{1}{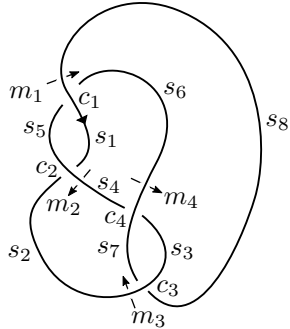}
		\caption{Labelings for the figure-eight knot diagram.}
		\label{fig:FigureEightSeg}
	\end{figure}
	Then $z=(z_1,\cdots,z_8)$ is an $m$-deformed solution if it satisfies $E_{S;m}$ :
		\allowdisplaybreaks
		\begin{align*}
			m&=\dfrac{z_1-z_6}{z_1-z_5} \cdot \dfrac{z_4(z_1-z_5)}{z_5(z_1-z_4)} = \dfrac{z_4(z_2-z_5)}{z_5(z_2-z_4)} \cdot \dfrac{z_2-z_7}{z_2-z_8} \\
&=\dfrac{z_3-z_7}{z_3-z_8} \cdot \dfrac{z_7(z_3-z_6)}{z_6(z_3-z_7)} = \dfrac{z_7(z_4-z_6)}{z_6(z_4-z_7)} \cdot \dfrac{z_4-z_2}{z_4-z_1} \\
&=	\dfrac{z_5-z_2}{z_5-z_1} \cdot \dfrac{z_8(z_5-z_1)}{z_1(z_5-z_8)} = \dfrac{z_8(z_6-z_1)}{z_1(z_6-z_8)} \cdot \dfrac{z_6-z_3}{z_6-z_4} \\
&=\dfrac{z_7-z_3}{z_7-z_4} \cdot \dfrac{z_3(z_7-z_2)}{z_2(z_7-z_3)} = \dfrac{z_3(z_8-z_2)}{z_2(z_8-z_3)} \cdot \dfrac{z_8-z_6}{z_8-z_5}.		
		\end{align*} 
	It is not difficult to check that 
	\begin{equation*}
		\begin{array}{ccl}
			(z_1,z_2,\cdots,z_8)&=& \left(pr,\ pr(1+q \Lambda),\ -\dfrac{pr\Lambda(1+q \Lambda)}{1-p}, \  \dfrac{pqr}{1-p} \right.\\[10pt]
			& & \quad \quad \quad \quad \left. ,\ -qr,\ r-qr,\ -\dfrac{pr(1-q)\Lambda^2}{1+p\Lambda},\ \dfrac{pr}{1+p\Lambda}\right)
		\end{array}
	\end{equation*} is a boundary parabolic solution (i.e., $m=1$) where $\Lambda^2+\Lambda	+1=0$. The solution has 2 choices of $\Lambda= \frac{-1 \pm \sqrt{-3}}{2}$ and ``free'' choices of variables $p,q$ and $r$ but satisfying the non-degeneracy condition. These three degrees of freedom come from the choice of the developing image of the ideal points $p^+$ and $p^-$, and the homogeneity of the hyperbolicity equations. For the case of $m \neq 1$ the computation becomes much more complicated and we will not discuss it here.
	\end{exam}
	\begin{prop}\label{prop:SegmRule} For a boundary parabolic solution, four segment variables sharing a crossing  are mutually distinct.
	\end{prop}
	\begin{proof} We only need to check that two segment variables in opposite position  sharing a common crossing are distinct. Suppose $z_a=z_b$ in Figure \ref{fig:hyp_segm}. Then from the hyperbolicity equation (\ref{eqn:hyp_z}) of $z_c$ we obtain $z_d=z_e$. Along the diagram, we conclude that all segment variables are  to be equal, which violates the non-degeneracy condition.
	\end{proof}
	We remark that Proposition \ref{prop:SegmRule} does not hold for a link. Indeed, there is a simple counter-example  for the Hopf link as follows.
	\[
	\vcenter{\hbox{
\begingroup%
  \makeatletter%
  \providecommand\color[2][]{%
    \errmessage{(Inkscape) Color is used for the text in Inkscape, but the package 'color.sty' is not loaded}%
    \renewcommand\color[2][]{}%
  }%
  \providecommand\transparent[1]{%
    \errmessage{(Inkscape) Transparency is used (non-zero) for the text in Inkscape, but the package 'transparent.sty' is not loaded}%
    \renewcommand\transparent[1]{}%
  }%
  \providecommand\rotatebox[2]{#2}%
  \ifx\svgwidth\undefined%
    \setlength{\unitlength}{84.53960507bp}%
    \ifx\svgscale\undefined%
      \relax%
    \else%
      \setlength{\unitlength}{\unitlength * \real{\svgscale}}%
    \fi%
  \else%
    \setlength{\unitlength}{\svgwidth}%
  \fi%
  \global\let\svgwidth\undefined%
  \global\let\svgscale\undefined%
  \makeatother%
  \begin{picture}(1,0.47888655)%
    \put(0,0){\includegraphics[width=\unitlength,page=1]{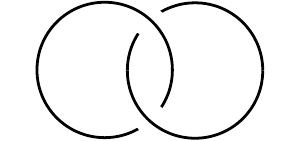}}%
    \put(-0.00211616,0.22558261){\color[rgb]{0,0,0}\makebox(0,0)[lb]{\smash{$z_1$}}}%
    \put(0.30473196,0.22693433){\color[rgb]{0,0,0}\makebox(0,0)[lb]{\smash{$z_2$}}}%
    \put(0.61522589,0.22524441){\color[rgb]{0,0,0}\makebox(0,0)[lb]{\smash{$z_3$}}}%
    \put(0.53822336,0.35470039){\color[rgb]{0,0,0}\makebox(0,0)[lb]{\smash{}}}%
    \put(0.91606346,0.22524441){\color[rgb]{0,0,0}\makebox(0,0)[lb]{\smash{$z_4$}}}%
  \end{picture}%
\endgroup%
}}
	~~~~~
	\left\{
	\begin{aligned}
	z_1&=z_3=p \\
	z_2&=z_4=q \\
	p &\neq q
	\end{aligned}
	\right.
	\]
	
	\begin{prop} \label{prop:mirror_seg} 
		 Let $(z_1,\cdots,z_{2N})$ be segment variables of $D$ and let $(z^*_1,\cdots,z^*_{2N})$ be segment variables of $D^*$ where  $D^*$ is the mirror diagram of $D$ and the indices of the segment in the same position of $D$ and $D^*$ are taken to be same. Then the set map $(z_1,\cdots,z_{2N}) \mapsto \left( 1/z_1, \cdots, 1/z_{2N}\right)$ is a bijection
		  between the set of $m$-deformed solutions to $\Tcal_{4D}$ and the set of $1/m$-deformed solutions to $\Tcal_{4D^*}$.
	\end{prop}
	\begin{proof} The proof follows by verifying equation (\ref{eqn:hyp_z}) directly.
	\end{proof}
	\subsection{Region variables}\label{subsec:region}
	In this subsection, we do similar work as in the previous subsection for the five-term triangulation $\Tcal_{5D}$. Let the hyperbolic ideal octahedra $\{o_1, \cdots, o_N\}$ be an $m$-deformed solution to $\Tcal_{5D}$.	To introduce region variables, let us record cross-ratios $\alpha_k,  \gamma_k,  \widehat{\alpha}_k$ and $ \widehat{\gamma}_k$ of $o_k$ around the crossing $c_k$ as follows.
	We locate a small arrow  on each segment  attached to $c_k$  with the right-handed orientation as in Figure \ref{fig:reg_cross}. We then assign $(\sqrt{m} \:   \gamma_k)^{-1}$ and $\sqrt{m} \:   \widehat{\gamma}_k$  to the small arrows on the out-going over-arc and the out-going under-arc, respectively, and also assign $ \alpha_k/\sqrt{m}$ and $\left( \widehat{\alpha}_k/\sqrt{m}\right)^{-1}$  to the small arrows on the incoming over-arc and the incoming under-arc, respectively, as in Figure \ref{fig:reg_cross}(a). 
	\begin{figure}[!h]
		\centering
		\scalebox{1}{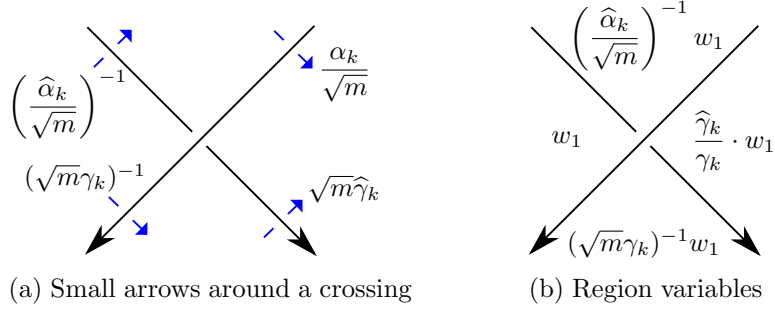}
		\caption{Cross-ratios of $o_k$ around a crossing $c_k$.}
		\label{fig:reg_cross}
	\end{figure}
	Then there are two small arrows on each segment of $D$. However, the $m$-hyperbolicity equation (\ref{eqn:hyp_angle}) is equivalent to condition that those two small arrows on each segment have the same value. Hence we regard that a single small arrow is assigned to each segment of $D$. 
	\begin{prop} \label{prop:region}
		There is a tuple  $(w_1,\cdots,w_{N+2})$ of nonzero complex numbers unique up to scalar multiplication such that $w_{j_2}/w_{j_1}$ is equal to the value on the small arrow pointing from $r_{j_1}$ to $r_{j_2}$ for every adjacent regions $r_{j_1}$ and $r_{j_2}$ of $D$.
	\end{prop}
	\begin{proof} Assign a non-zero complex number $w_1$ to an initial region $r_1$ of $D$. Then assign complex numbers to other regions by multiplying the value on an small arrow whenever we go across to the region through the small arrow. See Figure \ref{fig:reg_cross}(b). Then the values on regions of $D$ are well-defined by Lemma \ref{lem:dih_lem}(a), $\dfrac{ \widehat{\alpha}_k}{\sqrt{m}} \: \left(\dfrac{ \alpha_k}{\sqrt{m}}\right)^{-1} = (\sqrt{m} \gamma_k) \: (\sqrt{m} \widehat{\gamma}_k)^{-1}$.
	\end{proof}
	 Proposition \ref{prop:region} tells us that hyperbolic octahedra $o_1, \cdots, o_N$ of the following description are enough for considering an $m$-deformed solution to $\Tcal_{5D}$.
	 \begin{itemize}
	 	\item Assign a non-zero complex variable $w_j$ to a region $r_j$ of $D$ for $1 \leq j \leq N+2$.
	 	\item The shape parameters $\alpha_k, \gamma_k, \widehat{\alpha}_k$ and $\widehat{\gamma}_k$ of $o_k$ are determined by the ratio of $w$-variables as in Figure \ref{fig:reg_cross}.
	 \end{itemize}
	  From the construction, these hyperbolic octahedra $o_1,\cdots,o_N$ automatically satisfy $E_{S;m}$. Hence they form an $m$-deformed solution if and only if their cross-ratios satisfy $E_R$. To express $E_R$ in $w$-variables, we give $\tau$'s in terms of  $w$-variables as follows.
	\begin{equation}\label{eqn:tau_w_a}
	\begin{array}{cc}
	\left\{
	\begin{array}{ccc}
	\tau_{k,a} &=& \dfrac{( w_b- \sqrt{m}^{-1} w_a) ( w_d- \sqrt{m}^{-1} w_a)}{w_b w_d - w_a w_c} \\[10pt]
	\tau_{k,b} &=& \dfrac{ (w_a w_c -w_b w_d)}{( \sqrt{m}^{-1} w_a-  w_b) (\sqrt{m} w_c - \ w_b)}\\[12pt]
	\tau_{k,c} &=&\dfrac{( w_b-  \sqrt{m} w_c) ( w_d-  \sqrt{m}w_c)}{w_b w_d-w_a w_c} \\[10pt]
	\tau_{k,d} &=&\dfrac{ (w_a w_c - w_b w_d)}{( \sqrt{m}^{-1} w_a -  w_d ) ( \sqrt{m} w_c -  w_d)}  	  
	\end{array}
	\right. & \textrm{ for Figure \ref{fig:reg_crossing}(a)}
	\end{array}
	\end{equation}	
	 and
	\begin{equation} \label{eqn:tau_w_b}
    \begin{array}{cc}
	\left\{
		\begin{array}{ccc}
	\tau_{k,a} &=& \dfrac{( w_b- \sqrt{m} w_a) ( w_d- \sqrt{m} w_a)}{w_b w_d - w_a w_c} \\[10pt]
	\tau_{k,b} &=& \dfrac{ (w_a w_c -w_b w_d)}{( \sqrt{m} w_a-  w_b) (\sqrt{m}^{-1} w_c - \ w_b)}\\[12pt]
	\tau_{k,c} &=&\dfrac{( w_b-  \sqrt{m}^{-1} w_c) ( w_d-  \sqrt{m}^{-1} w_c)}{w_b w_d-w_a w_c} \\[10pt]
	\tau_{k,d} &=&\dfrac{ (w_a w_c - w_b w_d)}{( \sqrt{m} w_a -  w_d ) ( \sqrt{m}^{-1} w_c -  w_d)}  	  
	\end{array}
	\right. & \textrm{ for Figure \ref{fig:reg_crossing}(b)}.
    \end{array}
	\end{equation} Plugging equations (\ref{eqn:tau_w_a}) and (\ref{eqn:tau_w_b}) to the equation (\ref{eqn:region_edge}), we obtain a gluing equation in $E_R$ in $w$-variables.

	We finally consider the non-degeneracy condition for $\Tcal_{5D}$ : ``$\alpha_k,\gamma_k,\widehat{\alpha}_k,\widehat{\gamma}_k$ and $\alpha_k \gamma_k$ are not $0,1$ and $\infty$''. One can check that the non-degeneracy condition is equivalent to 
	\begin{equation} \label{eqn:non_d1}
		\left\{
		\begin{array}{l}
			w_a w_c - w_b w_d \neq 0\\[1pt]
			\sqrt{m}^{-1}w_a - w_b \neq 0,\ \sqrt{m}^{-1}w_a-w_d \neq 0\\[1pt]
		    \sqrt{m}w_c - w_b \neq 0,\  \sqrt{m}w_c-w_d \neq 0
		\end{array}
		\right. 
	\end{equation} for Figure \ref{fig:reg_crossing}(a) and
	\begin{equation} \label{eqn:non_d2}
		\left\{
		\begin{array}{l}
		w_a w_c - w_b w_d \neq 0\\[1pt]
		\sqrt{m}w_a - w_b \neq 0,\ \sqrt{m}w_a-w_d \neq 0\\[1pt]
		\sqrt{m}^{-1}w_c - w_b \neq 0,\  \sqrt{m}^{-1}w_c-w_d \neq 0
		\end{array}
		\right. 
	\end{equation} for Figure \ref{fig:reg_crossing}(b). Recall that $w_a,\cdots,w_d$ are non-zero complex numbers.
	\begin{defn}
	[Region variables] A nonzero variable $w_j$ assigned to a region $r_j$ of $D$ is called a \textit{region variable}. We say $w=(w_1,\cdots ,w_{N+2})$ is an $m$\textit{-deformed solution to $\Tcal_{5D}$} if it satisfies $E_R$ and conditions (\ref{eqn:non_d1}) and (\ref{eqn:non_d2}) for every crossing of $D$.
	\end{defn}
	
	\begin{exam}[Trefoil knot] Let $r_j$ be a region of the trefoil knot diagram labeled as in Figure \ref{fig:trefoil_reg} and let $w_j$ be the region variable assigned to $r_j$ for $1 \leq j \leq 5$.
		\begin{figure}[!h]
			\centering
			\scalebox{1}{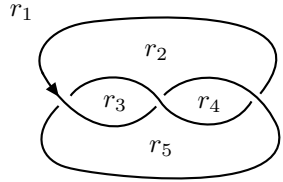}
			\caption{Labeling of regions of the trefoil knot diagram.}
			\label{fig:trefoil_reg}
		\end{figure}
	Then $w=(w_1,\cdots,w_5)$ is an $m$-deformed solution if it satisfies $E_R$
	\begin{equation*}
		\begin{array}{ccl}
			1&=& \dfrac{w_5w_2-w_1w_3}{(\sqrt{m}w_1-w_5)(\sqrt{m} w_1-w_2)} \cdot \dfrac{w_5 w_2 -w_1w_4}{(\sqrt{m}^{-1}w_1-w_5)(\sqrt{m}^{-1} w_1-w_2)} \\[13pt]
			&=& \dfrac{(w_2 -\sqrt{m}w_1)(w_2 -\sqrt{m}^{-1}w_3)}{w_1w_3-w_2w_5} \cdot \dfrac{(w_2 -\sqrt{m}w_3)(w_2 -\sqrt{m}^{-1}w_4)}{w_3w_4-w_2w_5} \\[10pt]
			&& \quad \quad \quad \quad \quad \quad \cdot \dfrac{(w_2 -\sqrt{m}w_4)(w_2 -\sqrt{m}^{-1}w_1)}{w_4w_1-w_2w_5} \\[13pt] 
			&=& \dfrac{w_5 w_2 -w_1w_3}{(\sqrt{m}^{-1} w_3-w_5)(\sqrt{m}^{-1} w_3-w_2)} \cdot \dfrac{w_5 w_2 -w_3w_4}{(\sqrt{m} w_3-w_5)(\sqrt{m} w_3-w_2)} \\[13pt]
			&=& \dfrac{w_5 w_2 -w_3w_4}{(\sqrt{m}^{-1} w_4-w_5)(\sqrt{m}^{-1} w_4-w_2)} \cdot \dfrac{w_5 w_2 -w_4w_1}{(\sqrt{m} w_4-w_5)(\sqrt{m} w_4-w_2)} \\[13pt]	
			&=& \dfrac{(w_5 -\sqrt{m}^{-1}w_3)(w_5 -\sqrt{m}w_1)}{w_1w_3-w_2w_5} \cdot \dfrac{(w_5 -\sqrt{m}^{-1}w_4)(w_5 -\sqrt{m}w_3)}{w_3w_4-w_2w_5} \\[10pt]
			&& \quad \quad \quad \quad \quad \quad \cdot \dfrac{(w_5 -\sqrt{m}^{-1}w_1)(w_5 -\sqrt{m}w_4)}{w_4w_1-w_2w_5}
		\end{array}
	\end{equation*} and the non-degeneracy condition.
	One can check that 
	$$(w_1,w_2,w_3,w_4,w_5)=\left(\dfrac{(1+q-p)r}{1+q+pq},\dfrac{(q-p)r}{1+q+pq} , \dfrac{(q-p+ pq)r}{1+q+pq}, \dfrac{(1+2q+pq)r}{1+q+pq}, r \right)$$ and 
		$$(w_1,w_2,w_3,w_4,w_5)=\left( p,\; q,\;p,\; p, \; 2p-q \right)$$ 
	are boundary parabolic solutions. Here we choose variables $p,q$ and $r$ in the solutions freely but satisfying the non-degeneracy condition. In the following subsection, we will see that the latter solution is different from the former one, since it consists of \emph{pinched} octahedra. Indeed, the latter solution gives an abelian representation while the former one gives an irreducible representation.
	\end{exam}		

	\subsection{Pinched octahedra}\label{subsec:pinched}
	 One of  the essential differences between $\Tcal_{4D}$ and $\Tcal_{5D}$ is that a solution to $\Tcal_{5D}$ allows the top and the bottom vertices of an octahedron to coincide while a solution to $\Tcal_{4D}$ does not.		
		\begin{defn}  We say that a hyperbolic ideal octahedron is \textit{pinched} if the top and the bottom vertices of the octahedron coincide.
		\end{defn}
		Let hyperbolic ideal octahedra $\{o_1,\cdots,o_N\}$ be an $m$-deformed solution to $\Tcal_{5D}$. We will keep using Notation \ref{notat:edge} and \ref{notat:octa_coord} in Sections \ref{subsec:gluing} and \ref{subsec:segment}, respectively.
		\begin{prop} \label{prop:pinch_cond} An octahedron $o_k$ is pinched if and only if  one of $\tau_{k,a},\tau_{k,b},\tau_{k,c},\tau_{k,d}$ is $1$. In this case, we actually have $\tau_{k,a}=\tau_{k,b}=\tau_{k,c}=\tau_{k,d}=1$.
		\end{prop}
		 \begin{proof} An octahedron $o_k$ is pinched if and only if $z_{k,T}=z_{k,B}$. We may   assume $z_{k,T}=\infty$ and then a simple cross-ratio computation gives that $\tau_{k,a} = \dfrac{z_{k,B}-z_{k,i_1}}{z_{k,B}-z_{k,i_2}}$ where $i_1$ and $i_2$ are the indices of the segments around the crossing $c_k$ attached to the region $r_a$. Also the non-degeneracy condition for $\Tcal_{5D}$ gives $z_{k,i_1} \neq z_{k,i_2}$  and thus $\tau_{k,a}=1$ if and only if $z_{k,B}=\infty$. Applying the same argument to other sides of $o_k$, we conclude the proposition.
		 \end{proof}		
		\begin{prop} \label{prop:col_reg} An octahedron $o_k$ is pinched if and only if 
    	    \begin{equation*}
            	\left\{
            	\begin{array}{cc}
		        \sqrt{m}^{-1}w_a- w_b+ \sqrt{m} w_c - w_d=0 & \textrm{ for Figure \ref{fig:reg_crossing}(a)} \\[10pt]
                \sqrt{m} w_a - w_b+ \sqrt{m}^{-1} w_c - w_d=0 & \textrm{ for Figure \ref{fig:reg_crossing}(b)}.
                \end{array}
            	\right.
	        \end{equation*}
		\end{prop}
		\begin{proof} Equations (\ref{eqn:tau_w_a}) and (\ref{eqn:tau_w_b}) and the non-degeneracy condition give a direct proof. For instance, we have $\tau_{k,a}-1=\dfrac{\sqrt{m} w_a( \sqrt{m} w_a- w_b + \sqrt{m}^{-1} w_c -  w_d)}{w_b w_d- w_aw_c}$ from the equation (\ref{eqn:tau_w_b}) and the proof follows by Proposition \ref{prop:pinch_cond}.
		\end{proof}
		\begin{prop} \label{lem:pinched_region} Suppose that a region of $D$ has $n$ corner crossings and $n-1$ octahedra among them are pinched. Then the remaining octahedron is also pinched.
		\end{prop}
		\begin{proof} The gluing equation of the regional edge corresponding to the region, the product of $n$ $\tau$'s which come from each corner crossing of the region is equal to $1$, tells us that any $\tau$ among them becomes $1$ whenever the others are $1$.
		\end{proof}
		\begin{exam} Suppose a diagram $D$ has a kink. Considering the region bounded by the kink with Proposition \ref{prop:pinch_cond}, the octahedron on the kink is pinched. We will give another ``non-trivial'' example in Section \ref{subsec:85}.
        \end{exam}

 	\section{Holonomy representation} \label{sec:Holonomy}

	Let segment variables $z=(z_1,\cdots,z_{2N})$ be a boundary parabolic solution to $\Tcal_{4D}$ and let $o_1,\cdots, o_N$ be hyperbolic ideal octahedra determined by $z$. Let $\rho$ be the  holonomy representation associated to the solution. In this section, we present an explicit formula to compute $\psl$-matrices  for the $\rho$-image  of the Wirtinger generators. We also present a formula for the cusp shape of the pseudo-hyperbolic structure.

	We will use the base work that we did in Section \ref{subsec:developing} and we briefly recall it here with further assumptions. Let $P_k$(resp., $\widehat{P}_k$) be the intersection point of $l_k$ (resp., $\widehat{l}_k$) and the central edge of $o_k$ and we choose the base point of $\pi_1(M)$ to be $P_1$. Let $\lambda_\circ$ be a loop obtained by pushing the knot to the right-hand side and we may assume that $\lambda_\circ$ passes through every $P_k$ and $\widehat{P}_k$. See Figure \ref{fig:figure_eight_dev}. 
	We denote two developing images of $o_k$ by $O_k$ and $\widehat{O}_k$ which appear when we consider the developing map along $\lambda_\circ$.
	We also denote the projection image of $O_k$ and $\widehat{O}_k$ to $\Cbb$ by $L_k$ and $\widehat{L}_k$, respectively. See Figure \ref{fig:figure_eight_OctaDev}. We may assume that the bottom and the top vertices of $O_1$ are placed at the origin and $\infty$, respectively, and the ``vertical'' diagonal of $L_1$ is of length $1$.  Choosing the meridian $\mu$ to be the Wirtinger generator winding the over-arc of $c_1$ and the lifting of the base point $P_1$ in $O_1$, these assumptions imply that the holonomy action of $\mu$ is $z \mapsto z+1$ and hence the ``vertical'' diagonal of every $L_k$ and $\widehat{L}_k$ is of length $1$. See Figure \ref{fig:segm_defn}. See Section \ref{subsec:developing} for details.

		 Following the oriented diagram $D$ from the over-arc of $c_1$, let $k_n$ be the index of the $n$-th under-passing crossing and $m_{k_n}$ be the Wirtinger generator winding the over-arc of the crossing $c_{k_n}$ for $1 \leq n \leq N$. We will compute  the $\rho$-image of Wirtinger generators in the order $\rho(m_{k_1}),\rho(m_{k_2}),\cdots, \rho(m_{k_N})$.  
		 
	 We define a loop $\underline{m}_{k_n} \in \pi_1(M)$ as follows : (1) Follow the loop $\lambda_\circ$ from the base point $P_1$ to $\widehat{P}_{k_n}$. (2) Wind the over-arc of $c_{k_n}$ in the right-handed orientation. (3) Retrace $\lambda_\circ$ from $\widehat{P}_{k_n}$ to $P_1$. One can depict $\underline{m}_{k_n}$ as in Figure \ref{fig:mu_loop}, where $e_{k_i}$ denotes the sign of the crossing $c_{k_i}$.
	\begin{figure}[!h] 		
		\centering		
		\begin{subfigure}[H]{0.9\textwidth}
			\centering
			\scalebox{1}{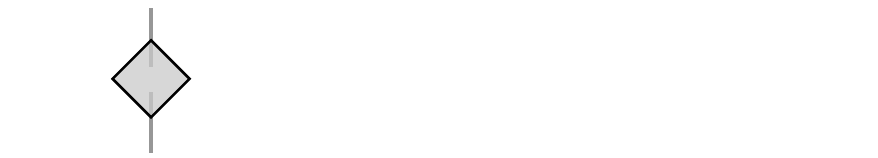}
			\caption{The loop $\underline{m}_{k_n}$ on the diagram.}
			\end{subfigure}

		\begin{subfigure}[H]{0.9\textwidth}
			\centering
			\scalebox{1}{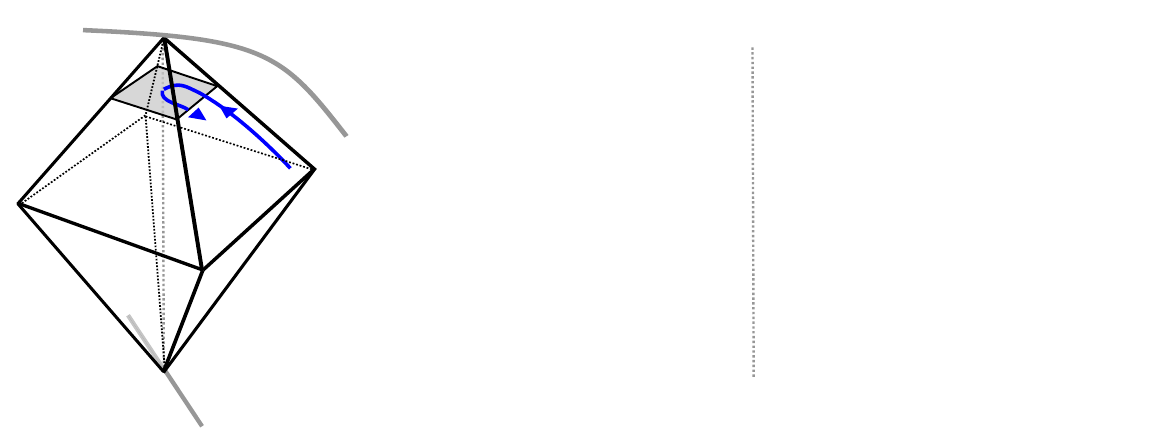}
			\caption{The loop $\underline{m}_{k_n}$ : $P_1 \overset{\lambda_\circ}{\rightarrow} \widehat{P}_{k_n}\rightarrow P_{k_n} \rightarrow D_{k_n} =B_{k_n}\rightarrow P_{k_n} \rightarrow \widehat{P}_{k_n} \overset{\lambda_\circ^{-1}}{\rightarrow} P_1$.}
		\end{subfigure}
		\caption{The loop $\underline{m}_{k_n}$.}
		\label{fig:mu_loop}
	\end{figure} 	
	\begin{lem} \label{lem:mu} The loop $\underline{m}_{k_{n}}$ can be expressed as
	\begin{equation*}
		\begin{array}{ccc} 
			\underline{m}_{k_{n}} &=& (m_{k_1}^{e_{k_1}}  \cdots m_{k_{n-1}}^{e_{k_{n-1}}}) \, m_{k_{n}} \, (m_{k_1}^{e_{k_1}}  \cdots m_{k_{n-1}}^{e_{k_{n-1}}})^{-1}
		\end{array}
	\end{equation*}
	\end{lem}
	\begin{proof}  Pulling the loop $\underline{m}_{k_n}$ over the diagram $D$, the part of $\underline{m}_{k_{n}}$ given in step (1) is $m_{k_1}^{e_{k_1}} \cdots m_{k_{n-1}}^{e_{k_{n-1}}}$ (see Figure \ref{fig:mu_loop}(a)) and thus we obtain the desired expression.
	\end{proof}	
	
	\begin{lem} \label{lem:exact_hol}  The holonomy action of $\underline{m}_{k_n}$  sends the lifting of the  hypotenuse $\beta_{k_n}$ (resp., $\delta_{k_n}$) to that of   $\delta_{k_n}$ (resp., $\beta_{k_n}$)  in $\widehat{O}_{k_n}$ at a positive crossing (resp., negative crossing), i.e., it sends  the hypotenuse on the right side   to the left side of the over-arc at $c_{k_n}$. In particular, it fixes the bottom vertex of $\widehat{O}_{k_n}$.
	\end{lem}
	\begin{proof} 
		Let $B_{k_n}$(resp., $D_{k_n}$) be the intersection point of  $\beta_{k_n}$(resp., $\delta_{k_n}$) and $l_{k_n}$ as in Figure \ref{fig:mu_loop}(b). Note that $B_{k_n}$ and $D_{k_n}$ are identified in $\Ocal_D$. Then we can rewrite $\underline{m}_{k_n}$ as $$\underbrace{P_1 \overset{\lambda_\circ}{\rightarrow}}_{(1)} \underbrace{\widehat{P}_{k_n} \rightarrow P_{k_n} \rightarrow D_{k_n} = B_{k_n} \rightarrow P_{k_n} \rightarrow}_{(2)} \underbrace{\widehat{P}_{k_n} \overset{\lambda_\circ^{-1}}{\rightarrow} P_1}_{(3)}.$$ Since we follow the path $P_1\rightarrow \widehat{P}_{k_n} \rightarrow P_{k_n}$ in the beginning and trace it back in the end, the holonomy action of $\underline{m}_{k_n}$ sends the lifting of $B_{k_n}$ in $\widehat{O}_{k_n}$ to that of $D_{k_n}$. 
		Recall that holonomy actions on the developing image are cellular.
		Hence the action sends the lifting of $\beta_{k_n}$ in $\widehat{O}_{k_n}$ to that of $\delta_{k_n}$. 

		The same argument also holds for  negative crossings except changing $\beta_{k_n}$ and $\delta_{k_n}$.
	\end{proof}
	Since hyperbolic octahedra $\{o_1,\cdots,o_N\}$ form a boundary parabolic solution, $\rho(m_{k_n})$ is indeed a parabolic element and so is $\rho(\underline{m}_{k_n})$  by Lemma \ref{lem:mu}. Therefore, Lemma \ref{lem:exact_hol} characterizes it uniquely.

	 Now we carry out an explicit computation for  $\rho(\underline{m}_{k_n})$ in segment variables.
	 We first compute the coordinates of the bottom vertex of $\widehat{O}_{k_n}$. To do this, we compute the difference of the coordinates between the bottom vertices of consecutive octahedra in the developing map along $\lambda_\circ$. The computation will be elementary and thus we only compute for the case of Figure \ref{fig:hyp_segm}(a) : the difference of the coordinates between the bottom and the east vertices of $O_{k}$ is $\frac{z_c}{z_a-z_b}$. See Figure \ref{fig:segm_defn}. 
	 Note that we divide it by $z_a-z_b$ to make the ``vertical'' diagonal of $L_k$ to be of length $1$. On the other hand, the difference of the coordinates between the bottom of the north vertices of $\widehat{O}_{k+1}$ is $\frac{z_d}{z_d-z_e}$. 
	 Note that $\widehat{O}_{k+1}$ can be obtained  by flipping $O_{k+1}$ through $ z \mapsto \frac{1}{z}$ and then normalized so that the ``vertical'' diagonal of $\widehat{L}_{k+1}$ becomes length $1$. 
	 Therefore, the difference of the coordinates between the bottom vertices of $O_k$ and $\widehat{O}_{k+1}$ is $\frac{z_c}{z_a-z_b} - \frac{z_d}{z_d-z_e}$. 
	\begin{figure}[!h]
		\centering
		\scalebox{0.8}{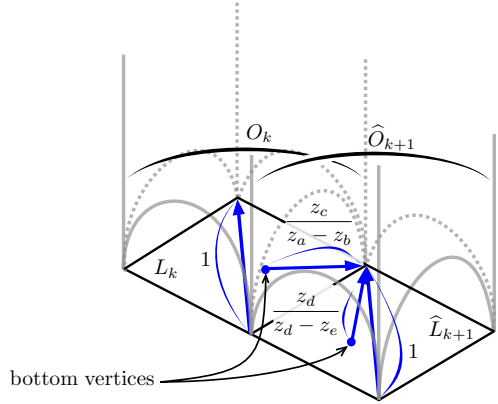}
		\caption{Developing map along $\lambda_\circ$ for Figure \ref{fig:hyp_segm}(a).}
		\label{fig:segm_defn}
	\end{figure}
	We compute the difference of the coordinates for other cases similarly and define these value as \emph{segment labels}.
	\begin{defn}\label{defn:segm_label} For a segment of $D$ we define a \textit{segment label} by 
		\begin{equation*} 
			\left\{
			\begin{array}{cc}
				\dfrac{z_c}{z_a-z_b} - \dfrac{z_d}{z_d-z_e} & \textrm{for  Figure \ref{fig:hyp_segm}(a)} \\[10pt]
				\dfrac{z_a}{z_a-z_b} - \dfrac{z_c}{z_d-z_e} & \textrm{for  Figure \ref{fig:hyp_segm}(b)} \\[10pt]
				\dfrac{z_c}{z_a-z_b} - \dfrac{z_d}{z_d-z_e} & \textrm{for  Figure \ref{fig:hyp_segm}(c)} \\[10pt]
				\dfrac{z_a}{z_a-z_c} - \dfrac{z_c}{z_d-z_e} & \textrm{for  Figure \ref{fig:hyp_segm}(d)}
			\end{array}
			\right.
		\end{equation*}
	\end{defn}
	Since we assume that the bottom vertex of $O_1$ is the origin, the coordinate of the bottom vertex of $\widehat{O}_{k_n}$ 
	is the sum of segment labels over every segment which appears when we follow the over-arc of $c_1$ from the crossing $c_1$ until we under-pass the crossing $c_{k_n}$.
	\begin{defn} For a crossing $c_k$ of $D$ as in  figure \ref{fig:reg_crossing}, we define a \textit{crossing label} $\Lambda_k$ by 
		\begin{equation*} 
				(z_d-z_b) \left( \dfrac{1}{z_c}-\dfrac{1}{z_a} \right).
		\end{equation*}
	\end{defn}
	\begin{rmk}\label{rmk:non_zero_crossing} Proposition \ref{prop:SegmRule} tells us that a crossing label is non-zero and an edge label is well-defined.
	\end{rmk}
		
	\begin{rmk}
		The terms ``segment label'' and ``crossing label'' were introduced in \cite{tsvietkova_hyperbolic_2012}, and changed to \emph{translation parameter} and \emph{intercusp parameter} in \cite{neumann_intercusp_2016}, respectively. To consider translation (resp., intercusp) parameters, we need to choose three (resp., two) ideal points with horospheres centered at each point. See Figures~2 and 3 in \cite{neumann_intercusp_2016}. We choose the horosphere centered at $\infty$ to be the plane of height $1$ in $\Hbb^3$, and the other horosphere centered at the top or the bottom vertex of $O_k$ or $\widehat{O}_k$, $\rho$-equivariently. Then the segment label in this paper coincides with the translation parameter in \cite{neumann_intercusp_2016} for the top vertex of $O_k$, and the bottom vertices of $O_k$ and $\widehat{O}_{k+1}$ in Figure \ref{fig:segm_defn} with respect to the chosen horospheres. Also the crossing label $\Lambda_k$ in this paper coincides with the intercusp parameter for the top and the bottom vertices of $O_k$.
	\end{rmk}
	\begin{lem}\label{lem:mu_repn} The $\rho$-image of $\underline{m}_{k_n}$ is 
		\begin{equation*}
			\rho(\underline{m}_{k_n})= \begin{pmatrix} 1+\Lambda_{k_n} \Sigma_{k_n} & -\Lambda_{k_n} \Sigma_{k_n}^2 \\[3pt] \Lambda_{k_n}  & 1-\Lambda_{k_n} \Sigma_{k_n} \end{pmatrix}
		\end{equation*}
		where $\Lambda_{k_n}$ is the crossing label of $c_{k_n}$ and $\Sigma_{k_n}$ is the sum of segment labels over every segment appearing when we follow the over-arc of $c_1$ from the crossing $c_1$ until we under-pass the crossing $c_{k_n}$. 
	\end{lem}
	\begin{proof} As we discussed before, the coordinate of the bottom vertex of $\widehat{O}_{k_n}$ is $\Sigma_{k_n}$. Since $\rho(\underline{m}_{k_n})$ is a parabolic element fixing $\Sigma_{k_n}$, $$\rho(\underline{m}_{k_n})=\begin{pmatrix} 1 & \Sigma_{k_n} \\ 0 & 1 \end{pmatrix} \begin{pmatrix} 1 & 0 \\ \Lambda & 1 \end{pmatrix}\begin{pmatrix} 1 & -\Sigma_{k_n} \\ 0 & 1 \end{pmatrix}$$ for some $\Lambda$. One can check that $\Lambda= \Lambda_{k_n}$ from Lemma \ref{lem:exact_hol} and thus obtain the lemma.
	\end{proof}

	Combining Lemmas \ref{lem:mu} and \ref{lem:mu_repn} together, we obtain :
	\begin{thm} \label{thm:HolonomyRep} For the Wirtinger generator $m_{k_n}$ winding the over-arc of the crossing $c_{k_n}$ we have
	$$			\rho(m _{k_n})= \rho(m_{k_1}^{e_{k_1}} \cdots m_{k_{n-1}}^{e_{k_{n-1}}})^{-1} \arraycolsep=4pt\begin{pmatrix} 1+\Lambda_{k_n} \Sigma_{k_n} & -\Lambda_{k_n} \Sigma_{k_n}^2 \\[3pt] \Lambda_{k_n} & 1-\Lambda_{k_n} \Sigma_{k_n}\end{pmatrix} \rho(m_{k_1}^{e_{k_1}} \cdots  m_{k_{n-1}}^{e_{k_{n-1}}})$$ where $\Lambda_{k_n}$ is the crossing label of $c_{k_n}$ and $\Sigma_{k_n}$ is the sum of segment labels over every segment appeared when we follow the over-arc of $c_1$ from the crossing $c_1$ until we under-pass the crossing $c_{k_n}$.
	\end{thm}

	    \begin{exam}[Figure-eight knot] \label{exam:figure_hol} We label crossings and segments as in Figure \ref{fig:FigureEightSeg}. Let $\sigma_i$ be a segment label of a segment $s_i$ and let $\Lambda_k$ be a crossing label of a crossing $c_k$. From the boundary parabolic solution given in Example \ref{ex:fig_seg_var}, we have	$\sigma_1=\sigma_5=0,\ \sigma_3=\sigma_7=-1,\ \sigma_2=\sigma_4=\sigma_6=\sigma_8= -\Lambda$ and $\Lambda_3=\Lambda_4=-\Lambda,\ \Lambda_1 =\Lambda_2= 1+\Lambda$ where $\Lambda^2 + \Lambda +1 =0$. From the diagram we have $(k_1,k_2,k_3,k_4)=(2,4,1,3)$ and compute $\rho(m_2), \rho(m_4), \rho(m_1)$ and $\rho(m_3)$ in this order.

		\begin{itemize}
			\item $\Sigma_2=\sigma_1=0,\ e_2=1$ and thus $\rho(m_2) = \arraycolsep=4pt\begin{pmatrix} 1+\Lambda_2 \Sigma_2 & -\Lambda_2 \Sigma_2 ^2  \\ \Lambda_2 & 1-\Lambda_2 \Sigma_2 \end{pmatrix}= \begin{pmatrix} 1 & 0 \\ 1+ \Lambda & 1 \end{pmatrix}$. 
			\item $\Sigma_4=\sigma_1+\sigma_2+\sigma_3 = -1-\Lambda,\ e_4=-1$ and thus $$\rho(m_4)= \rho(m_2)^{-1} \arraycolsep=4pt\begin{pmatrix} 1+\Lambda_4 \Sigma_4 & -\Lambda_4 \Sigma_4 ^2  \\ \Lambda_4 & 1-\Lambda_4 \Sigma_4 \end{pmatrix}\rho(m_2)=  \begin{pmatrix} -\Lambda & -1-\Lambda \\ 1+\Lambda & 2+\Lambda \end{pmatrix}.$$
			\item $\Sigma_1=\sigma_1+\sigma_2+\cdots+\sigma_5 = -1-2\Lambda,\ e_1=1$ and thus $$\rho(m_1)= \rho(m_2 \, m_4^{-1})^{-1} \arraycolsep=4pt\begin{pmatrix} 1+\Lambda_1 \Sigma_1 & -\Lambda_1 \Sigma_1^2 \\ \Lambda_1 & 1-\Lambda_1 \Sigma_1 \end{pmatrix} \rho(m_2 \, m_4^{-1}) = \begin{pmatrix} 1 & 1 \\ 0 & 1 \end{pmatrix}.$$
			\item $\Sigma_3=\sigma_1+\sigma_2+\cdots +\sigma_7=-2-3\Lambda,\ e_3=-1$ and thus $$\rho(m_3)= \rho(m_2 \, m_4^{-1} \, m_1)^{-1} \begin{pmatrix} 1+\Lambda_3 \Sigma_3 & -\Lambda_3 \Sigma_3^2 \\ \Lambda_3 & 1-\Lambda_3 \Sigma_3 \end{pmatrix} \rho(m_2 \, m_4^{-1} \, m_1) = \begin{pmatrix} 2+\Lambda & 1 \\ -\Lambda & -\Lambda \end{pmatrix}.$$
		\end{itemize}
		Remark that we have ``free'' parameters $p,q$ and $r$ in the solution but the holonomy representation $\rho$ does not depend on them as we computed. In this reason, we don't really care about these parameters.

    \end{exam}    		
	 \begin{prop} \label{prop:length_trace}  Let $\widehat{m}_{k_n}$ be the Wirtinger generator winding the incoming under-arc of $c_{k_n}$. Then we have $\textrm{Tr}\big(\widetilde{\rho}(m_{k_n} \widehat{m}_{k_n})\big) = 2+ \Lambda_{k_n}$ where $\widetilde{\rho}$ is a $\sl$-lifting of $\rho$.
	 \end{prop}
	 Note that the trace in the above equation does not depend on the choice of $\sl$-lifting.
   	\begin{proof}
   	 From Figure \ref{fig:mu_loop}(a) we have $$\widehat{m}_{k_n}=(m_{k_1}^{e_{k_1}} \cdots m_{k_{n-1}}^{e_{k_{n-1}}})^{-1} \, m_1 \, (m_{k_1}^{e_{k_1}} \cdots m_{k_{n-1}}^{e_{k_{n-1}}}).$$ Combining Lemmas \ref{lem:mu} and \ref{lem:mu_repn}, 
   	\begin{equation*}
   	\begin{array}{ccl}
   	\textrm{Tr}\big(\widetilde{\rho}(m_{k_n} \widehat{m}_{k_n})\big) &=& \textrm{Tr}\big(\widetilde{\rho}(m_{k_1}^{e_{k_1}} \cdots m_{k_{n-1}}^{e_{k_{n-1}}})^{-1} \, \widetilde{\rho}(\underline{m}_{k_n} m_1) \, \widetilde{\rho}(m_{k_1}^{e_{k_1}} \cdots m_{k_{n-1}}^{e_{k_{n-1}}})\big) \\[3pt]
   	&=& \textrm{Tr}\big(\widetilde{\rho}(\underline{m}_{k_n} m_1)\big) \\[3pt]
   	&=& \textrm{Tr}\left( \arraycolsep=4pt\begin{pmatrix} 1+\Lambda_{k_n} \Sigma_{k_n} & -\Lambda_{k_n} \Sigma_{k_n}^2 \\ \Lambda_{k_n}  & 1-\Lambda_{k_n} \Sigma_{k_n} \end{pmatrix}\begin{pmatrix} 1 & 1 \\ 0  & 1 \end{pmatrix}\right) \\[1pt]
   	&=& 2+ \Lambda_{k_n}
   	\end{array} 
   	\end{equation*} Here we use the fact that the Wirtinger generator $m_1$ coincides with  the meridian $\mu$ chosen in Section \ref{subsec:developing} whose holonomy action is $z \mapsto z+1$.	
   	\end{proof} 

     \begin{rmk} \label{rmk:hat_wirtinger}It is convenient to keep the holonomy action of $\widehat{m}_{k_n}$ in mind. For instance, one can verify that $\rho(\widehat{m}_1)$ sends the lifting of the hypotenuse $\widehat{\beta}_1$ in $O_1$ to that of $\widehat{\delta}_1$. In particular, it fixes the bottom vertex $O_1$ as in Figure~\ref{fig:hol_action} and a simple computation gives $\rho(\widehat{m}_1)= \arraycolsep=4pt\begin{pmatrix} 1 & 0 \\ \Lambda_1 & 1 \end{pmatrix}$ where $\Lambda_1$ is the crossing-label of the crossing $c_1$. Furthermore, conjugating $\widehat{m}_{k_n}$ appropriately as we did to $\underline{m}_{k_n}$, we can obtain a loop whose holonomy action sends the lifting of $\widehat{\beta}_{k_n}$ (resp., $\widehat{\delta}_{k_n}$) in $O_{k_n}$ to that of $\widehat{\delta}_{k_n}$ (resp., $\widehat{\beta}_{k_n}$) fixing the bottom of vertex of $O_{k_n}$ at a positive crossing (resp,. negative crossing). (This would be a ``hat-version'' of Lemma~\ref{lem:exact_hol}.)
   		\begin{figure}[!h]	
   			\centering
   			\scalebox{1}{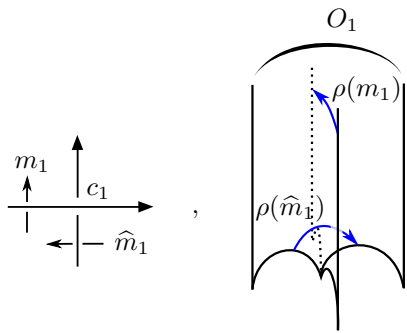}
   			\caption{The holonomy actions of $m_1$ and $\widehat{m}_1$ on $O_1$.}
   			\label{fig:hol_action}
   		\end{figure}
     \end{rmk}
     
         \begin{rmk} In this paper, we only present the formula in segment variables. However, the argument in this section can also be applied to the five-term triangulation. In this case, we need to care about pinched octahedra. In particular, Lemma \ref{lem:mu_repn} should be revised, since the fixed point  of $\rho(\underline{m}_{k_n})$ may be $\infty$ and in that case, $\rho(\underline{m}_{k_n})$ is of the form $\left(\begin{array}{cc} 1 & * \\ 0 & 1 \end{array} \right)$.  The exact formula will be given in \cite{KYptolemy}.
         \end{rmk}
         \begin{rmk} \label{rmk:hol_link} One can simply extend Theorem \ref{thm:HolonomyRep} to the case of a link as follows : consider a base point for each component of a link and then the argument in this section will be valid for each component. Then one can put them together through appropriate conjugation, which is determined by a path joining two base points. We will discuss the precise formula in \cite{KYptolemy}.
         \end{rmk}

   		Now we consider the cusp shape of $\rho$. Recall that the cusp shape of a hyperbolic knot is the Euclidean similarity structure on the link of the knot. For the holonomy representation of the hyperbolic structure, the \emph{cusp shape} is defined by the translation length of the holonomy action of the longitude when we fix the holonomy of the meridian by $z \mapsto z+1$. Here we take the longitude as a null-homologous one in $S^3\setminus K$.   In this manner, we define the cusp shape for any boundary parabolic representation by the translation length of the holonomy action of the longitude.
	\begin{thm} \label{thm:Longitude} The $\rho$-representation of the longitude $\lambda$ is $$\rho(\lambda)=\arraycolsep=4pt\begin{pmatrix} 1&\Sigma_\circ-w(D) \\[3pt] 0 & 1 \end{pmatrix}$$ where $\Sigma_\circ$ is the sum of segment labels over all segments in $D$ and $w(D)$ is the writhe of $D$. In particular, the cusp shape of $\rho$ is $\Sigma_\circ-w(D)$.
	\end{thm}
	\begin{proof} Since we fix the holonomy action of the meridian $\mu$ by $z \mapsto z+1$, that of any loop in the link of the knot is also a translation. In particular, the holonomy action of the loop $\lambda_\circ$ is $z \mapsto z+\Sigma_\circ$. As one verifies that the loop $\lambda_\circ$ winds the knot $w(D)$ times, the longitude $\lambda$ is $\lambda_\circ  \mu^{-w(D)}$ and hence we conclude the theorem.
	\end{proof}
	\begin{cor} \label{cor:cusp} The cusp shape of $\rho$ is
		$$\sum_\textrm{crossing $c_k$} \lambda_k$$ where $\lambda_k$ at the crossing $c_k$ is $\dfrac{z_c-z_a}{z_d-z_b}-1$ for Figure \ref{fig:reg_crossing}(a) and $\dfrac{z_c-z_a}{z_b-z_d}+1$ for Figure \ref{fig:reg_crossing}(b). Here $z_i$ is a segment variable assigned to the segment $s_i$ as in Figure \ref{fig:reg_crossing}.
	\end{cor}
	\begin{proof}
		Note that a segment label of a segment is defined by the sum of two terms where each term comes from the crossing attached to the segment, i.e., the first and the second terms in Definition \ref{defn:segm_label} come from the crossings $c_k$ and $c_{k+1}$ in Figure \ref{fig:hyp_segm}, respectively. Thus each crossing contributes four terms to  $\Sigma_\circ$, the sum of all segment labels. Precisely, these four terms are $-\frac{z_c}{z_c-z_a},\frac{z_c}{z_c-z_a},-\frac{z_a}{z_d-z_b}$ and $\frac{z_c}{z_d-z_b}$ for the crossing in Figure \ref{fig:reg_crossing}(a) and the sum of them is $\frac{z_c-z_a}{z_d-z_b}$. Similarly, we obtain $\frac{z_c-z_a}{z_b-z_d}$ for Figure \ref{fig:reg_crossing}(b). To manage the term $-w(D)$ in Theorem \ref{thm:Longitude}, we revise these sums by $-1$ and $1$ for Figure \ref{fig:reg_crossing}(a) and (b), respectively.
	\end{proof}	
	\begin{exam}[Figure-eight knot] As we computed in Example \ref{exam:figure_hol}, the sum of all segment labels is $-4\Lambda -2$ for the figure-eight knot diagram in Example \ref{ex:fig_seg_var} where $\Lambda^2+\Lambda+1=0$. Since the diagram has 0 writhe, the cusp shape is $\pm 2 \sqrt{-3}$. 
	\end{exam}
	\begin{rmk} \label{rmk:cusp_link} For  the case of a link, one can compute the cusp shape for each component of a link using the same formula give in Theorem \ref{thm:Longitude}. Specifically, summing segment labels of the segments in a component of a link and then revising it by the writhe of the component, it gives the cusp shape for the component.
	\end{rmk} 
	\section{The $T(2,N)$ torus knot and the $J(N,M)$ knot} \label{sec:Solution}
		In this section, we present all boundary parabolic solutions for the $T(2,N)$ torus knot $T(2,N)$ and the $J(N,-M)$ knot with respect to the diagrams given in Figure \ref{fig:torus_twist}. We then compute the complex volume and cusp shape by applying the solution to Theorem 1.2 in \cite{CKK_2014} and Theorem \ref{thm:Longitude}, respectively.
	\begin{figure}[!h]
			\centering
			\scalebox{1}{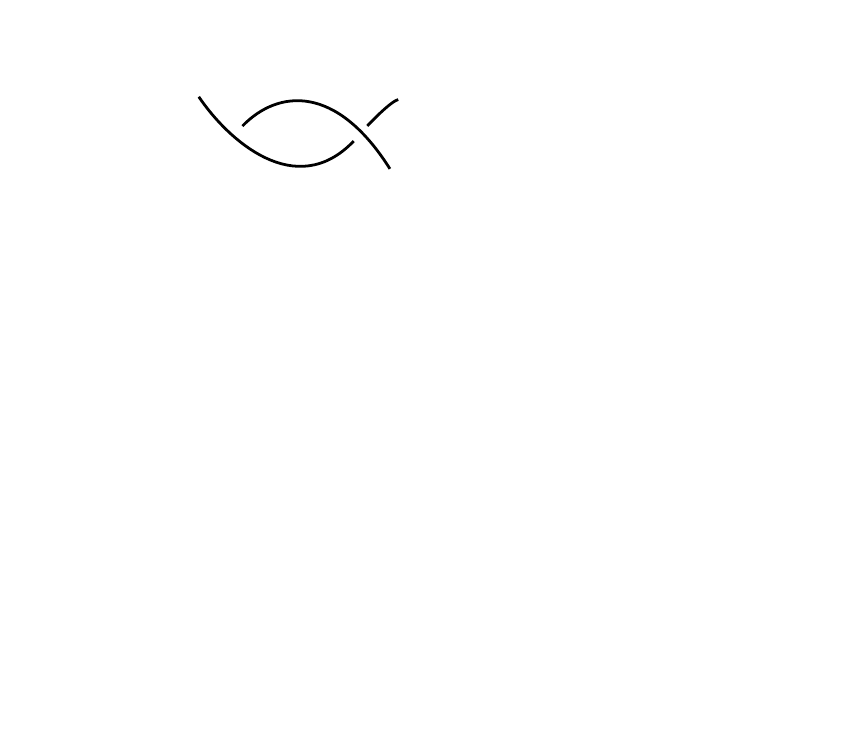}
			\caption{Diagrams of $T(2,N)$ and $J(N,-M)$.}
			\label{fig:torus_twist}
	\end{figure}

	Any non-trivial boundary parabolic representation arises as the holonomy of a pseudo-developing associated to a boundary parabolic solution for the five-term triangulation \cite{cho_optimistic_2016}. On the other hand, applying Proposition \ref{lem:pinched_region} to the given diagrams, one can check that all the octahedra have to be  pinched whenever one of them is pinched. In fact, in this case, we obtain an abelian representation as the holonomy which is not our interest.  Therefore, the boundary parabolic solutions for the four-term triangulation present all the non-abelian boundary parabolic representations. 
	
		\subsection{Boundary parabolic solutions for the $T(2,N)$ torus knot} \label{subsec:torus}
		Let us assign segment variables $z=(z_1,z_2,\cdots,z_{2N+1},z_{2N+2} )$ as in Figure \ref{fig:torus_twist}(a). We doubly label the top segment by $z_1, z_{2N+1}$ and the bottom segment by $z_2, z_{2N+2}$ for notational convenience. Also, we will confuse a segment and its segment variable for simplicity. We will compute a boundary parabolic solution in the following order.
		\begin{enumerate}
			\item[(1)]  Solve the hyperbolicity equations of $z_3,z_4,\cdots,z_{2N}$.
			\item[(2)] Solve $z_1=z_{2N+1}$ and $z_2=z_{2N+2}$.
			\item[(3)] Solve the hyperbolicity equations of $z_1(=z_{2N+1})$ and  $z_2(=z_{2N+2})$. 
		\end{enumerate}		
		\textbf{(1) Solve the hyperbolicity equations of  $z_3,z_4,\cdots,z_{2N}$.} \\[3pt]        
        We first define the notion of a \emph{$W$-Fibonacci sequence} which plays a key role in computation.
		\begin{defn} A sequence $(F_i)_{i \in \Zbb}$ is called a \emph{$W$-Fibonacci} for a complex number $W$ if it satisfies the recurrsion relation $F_{i+1} = W \cdot F_i + F_{i-1}$ for all $i$. The $W$-Fibonacci sequence $(B_i)_{i \in \Zbb}$ with initial conditions $B_0=0$ and $B_1=1$ is called the \emph{base $W$-Fibonacci sequence}.
		\end{defn}
		The following lemmas are quite easy but will be useful in the computation.
        \begin{lem} \label{lem:BasicFibo} Let $(F_i)_{i \in \Zbb}$ and $(G_i)_{i \in \Zbb}$ be $W$-Fibonacci sequences and $(B_i)_{i \in \Zbb}$ be the base $W$-Fibonacci sequence. Then
        	\\[3pt]
            (a) $F_i=F_0  B_{i-1} + F_1 B_{i}$ for all $i$;\\
            (b) For any $n \in \Zbb$, $$(-1)^i \cdot(F_iG_{i+n}-F_{i-1}G_{i+n+1})$$ does not depend on $i$. In particular, we have $B_i^2 -B_{i-1}B_{i+1}=(-1)^{i+1}$ for all $i$;\\
            (c) $B_{i+1} = \displaystyle\sum_{0 \leq j \leq \frac{i}{2}} \binom{i-j}{j} W^{i-2j}$ for all $i \geq 0$. 
		\end{lem} 
        \begin{proof} (a) Let $H_i:=F_0 B_{i-1} + F_i B_i$. Then $(H_i)_{i\in \Zbb}$ is the $W$-Fibonacci sequence satisfying $H_0=F_0$ and $H_1=F_1$. Since any $W$-Fibonacci sequence is uniquely determined by two consecutive terms, we have $H_i=F_i$ for all $i$. \\
	   (b) We have 			    $F_iG_{i+n}-F_{i-1}G_{i+n+1}=(WF_{i-1}+F_{i-2})G_{i+n}-F_{i-1}(WG_{i+n}+G_{i-1+n}) = -(F_{i-1}G_{i-1+n} - F_{i-2}G_{i+n})$.\\        	        
        (c) One can check that the right-hand side of the equation is a $W$-Fibonacci sequence with the same initial condition as $(B_i)_{i \in \Zbb}$.
		\end{proof}
		\begin{lem}\label{lem:key_lem} Let $D$ be a knot diagram containing a local diagram with segment variables $z_i$ assigned as in Figure \ref{fig:local_twisting}. Suppose $z_1,z_2,z_3$ and $z_4$ satisfy \begin{equation} \label{eqn:key_eqn}
		z_{2i-1} = \dfrac{F_i}{G_i},\ z_{2i} = \dfrac{F_{i-1}}{G_{i+1}}
		\end{equation} for some $W$-Fibonacci sequences $(F_i) _{i \in \Zbb}$ and $(G_i) _{i \in \Zbb}$, for $i=1,2$. Then the hyperbolicity equations of $z_3,z_4,\cdots,z_{2n}$ hold if and only if the equation (\ref{eqn:key_eqn}) holds for $1 \leq i \leq n+1$.
        \begin{figure}[!h]
			\centering
			\scalebox{1}{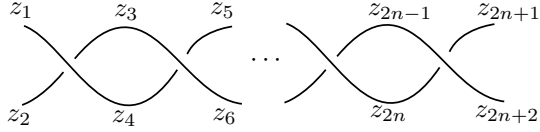}
			\caption{Local diagram with twisting.}
			\label{fig:local_twisting}
		\end{figure}
	\end{lem}
	\begin{proof} Rewriting the hyperbolicity equations of $z_{2i+1}$ and $z_{2i+2}$, we have 
\begin{equation*}
	\left\{ 
		\begin{array}{ll} 
			z_{2i+1}=z_{2i-1}+z_{2i} - \dfrac{z_{2i-1} z_{2i} }{z_{2i-3}} \\[10pt]
			z_{2i+2}=\dfrac{z_{2i-1} z_{2i}}{z_{2i-1}+z_{2i}-z_{2i-2}} 
		\end{array}
		\right.
\end{equation*} for $1 \leq i \leq n-1$ and one can prove the lemma by induction.
	\end{proof}
	\begin{rmk} \label{rmk:key_mirror} By Proposition \ref{prop:mirror_seg}, if one considers a diagram containing the mirror image of Figure \ref{fig:local_twisting}, then the equation (\ref{eqn:key_eqn}) is replaced by $$z_{2i-1} = \dfrac{F_{i-1}}{G_{i+1}},\  z_{2i} = \dfrac{F_{i}}{G_{i}}.$$
	\end{rmk}
	We now apply Lemma \ref{lem:key_lem} to $T(2,N)$ as follows. Let  $$\Lambda:=(z_2-z_3)\left(\dfrac{1}{z_1} - \dfrac{1}{z_4}\right)$$ and we define $\sqrt{\Lambda}$-Fibonacci sequences $(F_i)$ and $(G_i)$ by initial conditions $F_0= z_1 z_2 \sqrt{\Lambda}$, $F_1=z_1(z_3-z_2)$ and $G_1 =(z_3-z_2)$, $G_2= z_1\sqrt{\Lambda}$. Then the equation (\ref{eqn:key_eqn}) holds trivially for $i=1,2$ and hence for $1 \leq i \leq N+1$ by Lemma \ref{lem:key_lem}. We may rewrite $z_1,\cdots,z_{2N+2}$ using the base $\sqrt{\Lambda}$-Fibonacci sequence $(B_i)$ : 
	\begin{equation}\label{eqn:torus_sol}
	z_{2i-1} = \dfrac{z_1 z_2 \sqrt{\Lambda} B_{i-1}+z_1(z_3-z_2) B_i }{(z_3-z_2)B_{i-2}+z_1\sqrt{\Lambda}B_{i-1}},\ z_{2i} = \dfrac{z_1 z_2 \sqrt{\Lambda} B_{i-2}+z_1(z_3-z_2) B_{i-1} }{(z_3-z_2)B_{i-1}+z_1\sqrt{\Lambda}B_{i}}
	\end{equation} for $1 \leq i \leq N+1$. Here in the denominator we use $G_i=G_0B_{i-1}+G_1 B_i = G_1 B_{i-2}+G_2 B_{i-1}$. Using the explicit form of $B_i$ as in Lemma \ref{lem:BasicFibo}(c), the expressions of $z_{2i-1}$ and $z_{2i}$ in the equation (\ref{eqn:torus_sol}) are rational polynomials in $z_1,z_2,z_3$ and $\Lambda$ (not $\sqrt{\Lambda}$). \\[8pt]
	\textbf{(2) Solve $z_1=z_{2N+1}$ and $z_2=z_{2N+2}$.} \\[3pt]
	 To solve $z_1=z_{2N+1} \Leftrightarrow F_1G_{N+1}-F_{N+1}G_1=0$ and $z_2=z_{2N+2} \Leftrightarrow F_0G_{N+2}-F_{N}G_2= 0$ let $H_{i}:=F_1 G_{i+1} -F_{i+1}G_1$ and $H'_{i}:=F_0 G_{i+2} -F_{i}G_2$. Then $(H_{i})$ and $(H'_{i})$ are $\sqrt{\Lambda}$-Fibonacci sequences satisfying $H_0=H'_0=0$. Thus, by Lemma \ref{lem:BasicFibo}(a), we have $H_{N}=H_1 B_{N}$ and $H'_{N}=H'_1  B_{N}$. Also,  $H_1$ and $H'_1$ are non-zero, otherwise we have either $z_1=z_3$ or  $z_2=z_4$ and this violates Proposition \ref{prop:SegmRule}. Here we also use $\Lambda \neq 0$.  Therefore, we conclude that both $z_1=z_{2N+1}$ and $z_2=z_{2N+2}$ hold if and only if $B_N=0$. \\[8pt]	 
	\textbf{(3) Solve the hyperbolicity equations of $z_1(=z_{2N+1})$ and  $z_2(=z_{2N+2})$.} \\[3pt]
	To solve the hyperbolicity equation of $z_1$,  $$\dfrac{z_1-z_2}{z_1-z_3}\cdot \dfrac{z_2(z_1-z_{2N-1})}{z_{2N-1}(z_1-z_2)}=1 \Longleftrightarrow z_1 z_2 G_N -(z_1+z_2-z_3)F_N=0,$$ let $H''_i:=z_1 z_2 G_i -(z_1+z_2-z_3)F_i$. Then $(H''_i)$ is a $\sqrt{\Lambda}$-Fibonacci sequence satisfying $H''_0=0$ and $H''_1 \neq 0$.  Hence the hyperbolicity equation of $z_1$ holds if and only if $B_N=0$. Moreover, Proposition \ref{prop:redundant} tells us that we do not need to solve the hyperbolicity equation of $z_2$.
	
	Replacing $z_1,z_2,$ and $z_3$ by $p,q,$ and $r$, respectively, we obtain :
	\begin{thm} \label{thm:SolTorus} Let $z=(z_1,\cdots,z_{2N})$ be segment variables assigned to the torus knot $T(2,N)$ as in Figure \ref{fig:torus_twist}(a). Let $(B_i)_{i \in \Zbb}$ be the base $\sqrt{\Lambda}$-Fibonacci sequence. The segment variables $z$ form a boundary parabolic solution if and only if they are of the form
		$$z_{2i-1} = \dfrac{p q \sqrt{\Lambda}B_{i-1}+p(r-q) B_i }{(r-q)B_{i-2}+p\sqrt{\Lambda}B_{i-1}},\ z_{2i} = \dfrac{p q \sqrt{\Lambda} B_{i-2}+p(r-q) B_{i-1} }{(r-q)B_{i-1}+p\sqrt{\Lambda}B_{i}}
	$$ for $1 \leq i \leq N$ for some $p,q,r$ where $\Lambda$ is chosen to satisfy $B_N=0$. 
	\end{thm}
	Suppose $N=2n+1$ is odd so that $T(2,N)$ is a knot, then $B_N$ is a monic polynomial in $\Lambda$ of degree $n$ with no repeated root \cite{riley_parabolic_1972} and hence we have $n$ choices for $\Lambda$. We also can choose $p,q,$ and $r$ freely but satisfying the non-degeneracy condition. 
	\begin{rmk}\label{rmk:tours_rep} One can check that the segment labels and crossing labels do not depend on $p,q,$ and $r$. Thus, by Theorem \ref{thm:HolonomyRep}, the holonomy representation associated to the solution $z$ depends only on the choice of $\Lambda$, not on $p,q,r$. Therefore, the torus knot $T(2,2n+1)$ has $n$ non-abelian boundary parabolic representations.
	\end{rmk}
	\begin{exam}[Trefoil knot] Let $z=( z_1,\cdots, z_6 )$ be segment variables assigned to the trefoil knot, $T(2,3)$, as in Figure \ref{fig:torus_twist}(a). Then we have $B_3=\Lambda+1=0$ and hence $$(z_1,\cdots,z_6)=\left(p,\; q,\; r,\;  \dfrac{p(q-r)}{p+q-r},\; \dfrac{pq}{p+q-r},\; \dfrac{p r}{r-q}\right).$$
	\end{exam}
	\subsection{Boundary parabolic solutions for the $J(N,-M)$ knot} \label{subsec:JMN}
	We use the same strategy  for the $J(N,-M)$ knot as the previous subsection. We assign segment variables $z_1,\cdots,z_{2N+2}$ and $z'_1,\cdots,z'_{2M+2}$ as in Figure \ref{fig:torus_twist}(b). Note that there are $4$ double-labelings: $z_{2N+2}=z'_1,\  z'_2=z_2,\ z_{2N+1}=z'_{2M+1}$, and $z_1=z'_{2M+2}$. We will compute a boundary parabolic solution in the following order.
	\begin{enumerate}	
		\item[(1)] Apply Lemma \ref{lem:key_lem} to solve the hyperbolicity equations of $z_3,\cdots,z_{2N}$.
		\item[(2)] Solve the hyperbolicity equations of $z_{2N+2}$ and $z_{2}$ to obtain $z'_3$ and $z'_4$, respectively.
		\item[(3)] Set $z'_1:=z_{2N+2},\ z'_2:=z_2$ and apply Lemma \ref{lem:key_lem} again to solve the hyperbolicity equations of $z'_3,\cdots,z'_{2M}$.
		\item[(4)] Solve $z_{2N+1}=z'_{2M+1},\ z_{1}=z'_{2M+2}$ and the hyperbolicty equations of them.
	\end{enumerate}
	\textbf{(1) Apply Lemma \ref{lem:key_lem} to solve the hyperbolicity equations of $z_3,\cdots,z_{2N}$.} \\[3pt]
	Repeating the first step of the previous subsection, we have 
	\begin{equation}\label{eqn:jnm_sol1} z_{2i-1} = \dfrac{F_i}{G_i},\  z_{2i} = \dfrac{F_{i-1}}{G_{i+1}}\end{equation}for $1 \leq i \leq N+1$ where  $(F_i)$ and $(G_i)$ are $\sqrt{\Lambda}$-Fibonacci sequences with initial conditions $F_0= z_1 z_2 \sqrt{\Lambda}$, $F_1=z_1(z_3-z_2)$ and $G_1 =(z_3-z_2)$, $G_2= z_1\sqrt{\Lambda}$. Recall that the expressions of the equation (\ref{eqn:jnm_sol1}) consists of rational polynomials in $z_1,z_2,z_3,$ and $\Lambda=(z_2-z_3)\left(\frac{1}{z_1} - \frac{1}{z_4}\right)$. \\[8pt]
	\textbf{(2) Solve the hyperbolicity equations of $z_{2N+2}$ and $z_{2}$ to obtain $z'_3$ and $z'_4$, respectively.} \\[3pt]
	Plugging the equation (\ref{eqn:jnm_sol1}) into the hyperbolicity equation of $z_{2N+2}$ we have  
	\begin{equation}\label{eqn:y3}
		\begin{array}{ll}
		&\dfrac{z_{2N+2}-z_{2N+1}}{z_{2N+2}-z_{2N}}\cdot \dfrac{z'_3(z_{2N+2}-z_2)}{z_2(z_{2N+2}-z'_3)}=1 \\[15pt] 
		\Leftrightarrow&	\dfrac{F_0}{z'_3}=G_2+ \dfrac{\sqrt{\Lambda}B_N(F_0G_{N+2}-F_NG_2)}{F_NG_{N+1}-F_{N-1}G_{N+2}} \cdot G_{N+2}.
	\end{array}
	\end{equation} By Lemma \ref{lem:BasicFibo}(a) we have $F_0G_{N+2}-F_NG_2=( F_0G_3- F_1G_2) \cdot B_N$ where $(B_i)$ is the base $\sqrt{\Lambda}$-Fibonacci sequence. Also by Lemma \ref{lem:BasicFibo}(b) we have $F_NG_{N+1}-F_{N-1}G_{N+2}=(-1)^{N-1} \cdot (F_1G_2-F_0G_3)$. Hence the equation (\ref{eqn:y3}) can be simplified to \begin{equation} \label{eqn:z3}
	z'_3=\dfrac{F_0}{G_2+(-1)^{N}\sqrt{\Lambda}B_N\cdot G_{N+2}}.
	\end{equation} Similarly, the hyperbolicity equation of $z_2$ is simplified to \begin{equation} \label{eqn:z4} z'_4= \dfrac{F_N + \sqrt{\Lambda} B_N \cdot F_0}{G_{N+2}}. \end{equation}
	\textbf{(3) Set $z'_1:=z_{2N+2},\ z'_2:=z_2$ and apply Lemma \ref{lem:key_lem} again to solve the hyperbolicity equations of $z'_3,\cdots,z'_{2M}$.} \\[3pt]
	Set $z'_1:=z_{2N+2}=\dfrac{F_N}{G_{N+2}}$ and $z'_2:=z_2=\dfrac{F_0}{G_2}$. Let $\Lambda':=(-1)^{N} \Lambda B_N^2$ and define $\sqrt{\Lambda'}$-Fibonacci sequences $(F'_i)$ and $(G'_i)$ by initial conditions $F'_0 =(-1)^{\frac{N}{2}} F_N,\ F'_1=F_0$ and $G'_1 =G_2$, $G'_2=(-1)^{\frac{N}{2}} G_{N+2}$. Then one can check that
	\begin{equation}\label{eqn:jnm_sol2} z'_{2i-1}=\dfrac{F'_{i-1}}{G'_{i+1}},\ z'_{2i}=\dfrac{F'_{i}}{G'_{i}}\end{equation} holds for $i=1,\ 2$ by the equations (\ref{eqn:z3}) and (\ref{eqn:z4}), and thus it holds also for $i=1,2,\cdots, M+1$ by Lemma \ref{lem:key_lem} and Remark \ref{rmk:key_mirror}. We remark that the expressions for $z'_{2i-1}$ and $z'_{2i}$ in the equation (\ref{eqn:jnm_sol2}) are rational polynomials in $z_1,z_2,z_3,$ and $\Lambda$.\\[8pt]
	\textbf{(4) Solve $z_{1}=z'_{2M+2},\ z_{2N+1}=z'_{2M+1}$ and the hyperbolicity equations of them. }\\[3pt]
	Let $(B'_i)$ be the base $\sqrt{\Lambda'}$-Fibonacci sequence. Then we have
    \begin{equation*}
       	\begin{array}{ccl}
           	F'_i&=&(-1)^{\frac{N}{2}}F_N B'_{i-1}+z_1 z_2 \sqrt{\Lambda} B'_i \\[5pt]
            G'_i&=&z_1\sqrt{\Lambda}  B'_{i-2}+(-1)^{\frac{N}{2}}G_{N+2} B'_{i-1}
        \end{array}
     \end{equation*} for all $i$. One can check that $z_1=z'_{2M+2}$ holds if and only if $B'_{M+1} +(-1)^\frac{N}{2} B'_M B_{N-1}=0$ using the following equalities. 
\begin{equation*}
	\begin{array}{ccl}
	 F'_{M+1}-z_1 G'_{M+1} &=&(-1)^{\frac{N}{2}} B'_{M} (F_N-z_1 G_{N+2}) + z_1 z_2 \sqrt{\Lambda} B'_{M+1} - z_1^2 \sqrt{\Lambda} B'_{M-1} \\[3pt]
	&=&(-1)^{\frac{N}{2}} B'_{M} (z_1z_2 \sqrt{\Lambda} B_{N-1} - z_1^2 \sqrt{\Lambda} B_{N+1}) \\[3pt]
	& & \quad \quad \quad \quad \quad \quad + z_1 z_2 \sqrt{\Lambda} B'_{M+1} - z_1^2 \sqrt{\Lambda} B'_{M-1} \\[5pt]
      &=&z_1z_2 \sqrt{\Lambda}(B'_{M+1} + (-1)^\frac{N}{2} B'_M B_{N-1}) \\[3pt]
     & & \quad \quad \quad \quad \quad \quad - z_1^2 \sqrt{\Lambda} (B'_{M-1}+(-1)^\frac{N}{2} B'_M B_{N+1})\\[3pt]	&=&z_1 (z_2 -z_1) \sqrt{\Lambda}(B'_{M+1} + (-1)^\frac{N}{2} B'_M B_{N-1}). 
	\end{array}
\end{equation*} Here in the last equality, we use $B'_{M+1}-B'_{M-1} = (-1)^\frac{N}{2}(B_{N+1}-B_{N-1}) B'_M$.

By similar computation one can check that $B'_{M+1} +(-1)^\frac{N}{2} B'_M B_{N-1}=0$ is also a necessary and sufficient condition for $ z_{2N+1}=z'_{2M+1}$ and the hyperbolicity equations of $z_1$ and $z_{2N+1}$.
\begin{thm} \label{thm:SolJMN} Let $z=(z_1,\cdots,z_{2N+2},z'_3,\cdots,z'_{2M} )$ be segment variables assigned to the $J(N,-M)$ knot diagram as in Figure \ref{fig:torus_twist}(b). Let $(B_i)_{i \in \Zbb}$ be the base $\sqrt{\Lambda}$-Fibonacci sequence and $(B'_i)_{i \in \Zbb}$ be the base $\big((-1)^\frac{N}{2} \sqrt{\Lambda} B_N\big)$-Fibonacci sequence. The segment variables $z$ form a boundary parabolic solution if and only if they are of the form 
	 $$ z_{2i-1} = \dfrac{p q \sqrt{\Lambda} B_{i-1}+p(r-q) B_i}{(r-q)B_{i-2}+p\sqrt{\Lambda}B_{i-1}},\ z_{2i} = \dfrac{p q \sqrt{\Lambda} B_{i-2}+p(r-q) B_{i-1}}{(r-q)B_{i-1}+p\sqrt{\Lambda}B_{i}}
	$$ for $1 \leq i \leq N+1$,  and  $$	z'_{2i-1} = \dfrac{(-1)^\frac{N}{2} F_N B'_{i-2}+pq\sqrt{\Lambda} B'_{i-1}}{p\sqrt{\Lambda}B'_{i-1}+(-1)^\frac{N}{2}G_{N+2}B'_{i}},\ z'_{2i} = \dfrac{(-1)^\frac{N}{2} F_N B'_{i-1}+pq\sqrt{\Lambda} B'_i}{p\sqrt{\Lambda}B'_{i-2}+(-1)^\frac{N}{2}G_{N+2}B'_{i-1}}
	$$  for $2 \leq i \leq M$ where $F_N= p q \sqrt{\Lambda} B_{N-1}+p(r-q) B_N$ and $G_{N+2}=(r-q)B_{N}+p\sqrt{\Lambda}B_{N+1}$ and $\Lambda$ is chosen to satisfy $B'_{M+1} + (-1)^\frac{N}{2} B'_M B_{N-1} = 0$.
\end{thm}
	\begin{rmk}\label{rmk:jnm_rep}If we assume either $N$ or $M$ is even so that $J(N,-M)$ is a knot, then $B'_{M+1} + (-1)^\frac{N}{2} B'_M B_{N-1}$ turns out to be a monic polynomial in $\Lambda$ of degree $\frac{NM}{2}$ with no repeated roots and therefore the $J(N,-M)$ knot has $\frac{NM}{2}$ non-abelian boundary parabolic representations. The polynomial $B'_{M+1} + (-1)^\frac{N}{2} B'_M B_{N-1}$ coincides with \emph{Riley's polynomial} $\Lambda(y)$ in \cite{riley_parabolic_1972}.
	\end{rmk}
	\subsection{Cusp shapes and complex volumes of $T(2,N)$}
    	We fix $N$ to be odd so that $T(2,N)$ is a knot. Let segment variables $z=(z_1, \cdots, z_{2N})$ be a boundary parabolic solution given in Theorem \ref{thm:SolTorus}. As we mentioned in Remark \ref{rmk:tours_rep}, the holonomy representation  only depends on the choice of $\Lambda$. We denote the solution $z$ by $z_\Lambda$ to stress the choice of $\Lambda$ and the corresponding holonomy representation by $\rho_\Lambda$.
    	
    	We first compute the cusp shape of $\rho_\Lambda$ using Theorem \ref{thm:Longitude}. To do this, we compute the segment label $\sigma_i$ of a segment $z_i$.
    	\begin{itemize}
        	\item[(1)] For $1 \leq i \leq {N-1}$ we have
            \begin{equation*}
            	\begin{array}{ccl}
                	\sigma_{2i+1} &=&\dfrac{z_{2i-1}}{z_{2i-1}-z_{2i+2}}-\dfrac{z_{2i+1}}{z_{2i+3}-z_{2i+2}}\\[9pt]
                    &=& \dfrac{F_i/G_i}{F_{i}/G_{i}-F{i}/G_{i+2}} - \dfrac{F_{i+1}/G_{i+1}}{F_{i+2}/G_{i+2}-F{i}/G_{i+2}} \\[9pt]
                    &=& \dfrac{G_{i+2}}{\sqrt{\Lambda} G_{i+1}} - \dfrac{G_{i+2}}{\sqrt{\Lambda} G_{i+1}}=0. 
                \end{array}
            \end{equation*}
            \item[(2)] For $i=1$ we have \begin{equation*}
            	\begin{array}{ccl}
                	\sigma_1&=& \dfrac{z_{2N-1}}{z_{2N-1}-z_{2N+2}}-\dfrac{z_1}{z_3-z_2} \\[7pt]
                    &=& \dfrac{F_{N}/G_{N}}{F_{N}/G_{N}-F_{N}/G_{N+2}} - \dfrac{F_{1}/G_{1}}{F_{2}/G_{2}-F_{0}/G_{2}} \\[9pt]
                    &=& \dfrac{G_{N+2}}{\sqrt{\Lambda} G_{N+1}} - \dfrac{G_{2}}{\sqrt{\Lambda} G_{1}} \\[9pt]
                    &=&\dfrac{G_{N+2}G_1 - G_{N+1}G_2}{\sqrt{\Lambda}G_{N+1}G_1}=0. 
                \end{array}
            \end{equation*} In the last equation, we use Lemma \ref{lem:BasicFibo}(a): $G_{N+2}G_1 - G_{N+1} G_2 = (G_3G_1-G_2^2) B_N$.
            \item[(3)] Similar computation gives $\sigma_{2i} =-1$ for $1 \leq i \leq {N}$.
        \end{itemize}
        Since the diagram in Figure~\ref{fig:torus_twist}(a) has $-N$ writhe, we have :
        \begin{prop} Let $z_\Lambda=(z_1,\cdots,z_{2N})$ be a boundary parabolic solution of the $T(2,N)$-torus knot given in Theorem \ref{thm:SolTorus} with a prefered choice of $\Lambda$. Let $\rho_\Lambda$ be the holonomy representation associated to $z_\Lambda$. Then the cusp shape of $\rho_\Lambda$ is $-2N$.
	    \end{prop}
        
        \begin{rmk} This agrees with Proposition 5 in \cite{riley_parabolic_1972}. Note that Riley used the longitude which is the inverse of ours.
        \end{rmk}
        
	Now we compute the complex volume of $\rho_\Lambda$ using the $z$-solution and the formula given by Theorem 1.2 in \cite{CKK_2014}. The volume of $\rho_\Lambda$ is $0$ as expected and we plot Chern-Simons invariants (mod $\pi^2$) up to $N=31$. 
    	\begin{figure}[!h]
    		\centering
    		\def\svgwidth{0.8\columnwidth}
    		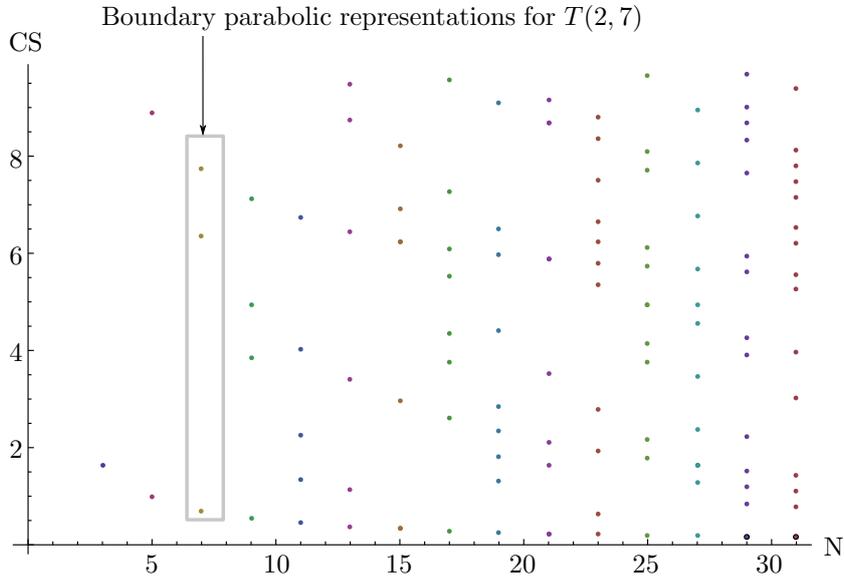 
    		\caption{Chern-Simons invariants of $T(2,N)$ (mod $\pi^2$).}
    		\label{fig:ChernSimonTorus}
    	\end{figure} 
    For instance, $3$ dots on the vertical line $N=7$ are the Chern-Simons invariants (mod $\pi^2$) of $T(2,7)$.
    It is interesting to observe that the graph without reducing modulo $\pi^2$ has more natural patterns as in Figure \ref{fig:ChernSimonTorusB}. Further work would be required to investigate this phenomenon more precisely.
     	\begin{figure}[!h]
    		\centering
    		\def\svgwidth{0.8\columnwidth}
    		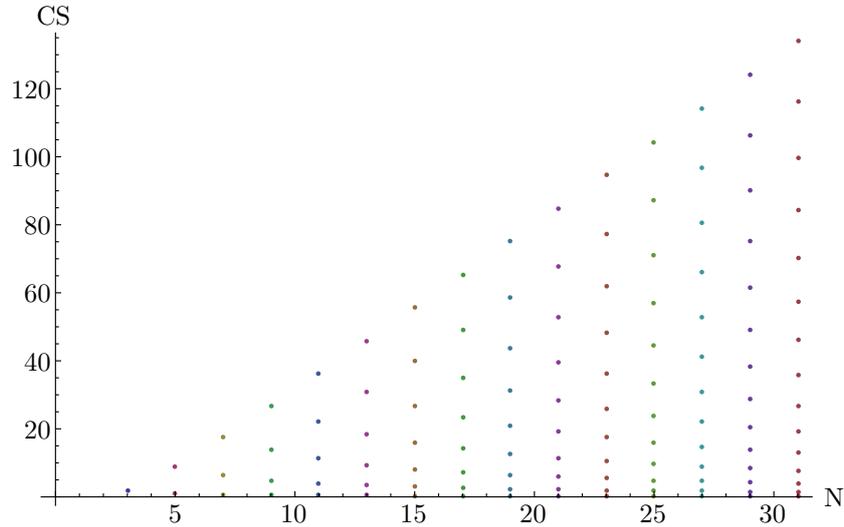 
    		\caption{Chern-Simons invariants of $T(2,N)$ without reducing modulo $\pi^2$.}
    		\label{fig:ChernSimonTorusB}
    	\end{figure}
    \subsection{Cusp shapes and complex volumes of the $J(N,-M)$ knot.} 
       We assume that $N$ is even so that  $J(N,-M)$ is a knot. Let segment variables $z_\Lambda=(z_1,\cdots,z_{2N+2},z'_3,\cdots,z'_{2M} )$ be a boundary parabolic solution given in Theorem \ref{thm:SolJMN} with the preferred choice of $\Lambda$. Let $\rho_\Lambda$ be the holonomy rerpresentation associated to $z_\Lambda$.
       
       To give the cusp shape formula we compute the segment labels $\sigma_i$ and $\sigma'_{j}$. We first assume that $M$ is odd.
        \begin{itemize}
        	\item[(1)] The same computation as in the previous subsection gives $\sigma_{2i+1}=0$ and $\sigma_{2i+2}=-1$ for $1 \leq i \leq N-1$.
        	\item[(2)] Similarly, $\sigma'_{j}=0$ if $j \equiv 1,2 \ (\textrm{mod } 4)$ and  $\sigma'_{j}=1$ if $j \equiv 0,3 \ (\textrm{mod } 4)$ for $3 \leq j \leq 2M$.
        	\item[(3)] For $\sigma_{2N+2}$ we have
        	\begin{equation*}
        		\begin{array}{ccl}
        		\sigma_{2N+2}&=&\dfrac{z_{2N+2}}{z_{2N+1}-z_{2N}}-\dfrac{z'_3}{z'_3-z'_2} \\[10pt]
        		&=&\dfrac{G_{N+1}}{\sqrt{\Lambda}G_{N+2}}+\dfrac{G'_{1}}{\sqrt{\Lambda'}G'_{2}} \\[10pt]
        		&=&\dfrac{1}{\sqrt{\Lambda}B_N} \cdot \dfrac{(z_3-z_2)B_{N-1}B_N + z_1 \sqrt{\Lambda}(B_N^2+1)}{(z_3-z_2)B_N+z_1 \sqrt{\Lambda} B_{N+1}} \\[10pt]
        		&=&\dfrac{B_{N-1} }{\sqrt{\Lambda}B_N} = \dfrac{B_{N-1} }{B_{N+1}-B_{N-1}}.
        		\end{array}
        	\end{equation*} The fourth equality follows from Lemma \ref{lem:BasicFibo}(b): $B_N^2+1=B_{N-1}B_{N+1}$. Similarly, we have $\sigma_2=\dfrac{B_{N-1} }{B_{N+1}-B_{N-1}}$.
        	\item[(4)] For $\sigma_{1}$ we have
        	$$\dfrac{G'_M}{G'_{M+1}} = \dfrac{(z_3-z_2)(-1)^\frac{N}{2} B_NB'_{M-1} + z_1 \sqrt{\Lambda}(B'_{M-2}+(-1)^\frac{N}{2}B_{N+1}B'_{M-1})}{(z_3-z_2) (-1)^\frac{N}{2} B_NB'_M + z_1 \sqrt{\Lambda}(B'_{M-1}+(-1)^\frac{N}{2}B_{N+1}B'_M)}$$ and $B'_{M+1}+(-1)^\frac{N}{2}B'_MB_{N-1}=B'_{M-1}+(-1)^\frac{N}{2}B'_MB_{N+1}=0$. Therefore,
				\begin{equation*}
					\begin{array}{ccl}
						\sigma_1 &=& \dfrac{z'_{2M+1}}{z'_{2M+1}-z'_{2M}}-\dfrac{z_1}{z_3-z_2} \\[10pt]
						&=&\dfrac{G'_M}{-\sqrt{\Lambda'}G'_{M+1}}-\dfrac{z_1}{z_3-z_2} \\[10pt]
						&=& -\dfrac{(z_3-z_2) (-1)^\frac{N}{2} B_NB'_{M-1} + z_1 \sqrt{\Lambda}(B'_{M-2}+(-1)^\frac{N}{2}B_{N+1}B'_{M-1})}{(z_3-z_2) (-1)^\frac{N}{2} \sqrt{\Lambda'}B_NB'_M}-\dfrac{z_1}{z_3-z_2}	\\[5pt]	
				&=& -\dfrac{(z_3-z_2) (-1)^\frac{N}{2} B_NB'_{M-1}}{(z_3-z_2) (-1)^\frac{N}{2} \sqrt{\Lambda'}B_NB'_M} = -\dfrac{B'_{M-1}}{B'_{M+1}-B'_{M-1}} = \dfrac{B_{N+1}}{B_{N+1}-B_{N-1}}.
					\end{array}
				\end{equation*} We use $B_N^2+1=B_{N-1}B_{N+1}$ in the fourth equality and $B'_{M+1}+(-1)^\frac{N}{2}B'_MB_{N-1}=B'_{M-1}+(-1)^\frac{N}{2}B'_MB_{N+1}=0$ in the last equality. Similar computation gives $\sigma_{2N+1}=\dfrac{B_{N+1}}{B_{N+1}-B_{N-1}}$.
        \end{itemize}
         For an even $M$, we have
        \begin{itemize} 
        	\item[(1)] $\sigma_i =0 $ if $ i  \equiv 1,2 \textrm{ (mod 4)}$ and $\sigma_i = -1 $ if $i \equiv 0,3 \textrm{ (mod 4)}$ for $3 \leq i \leq 2N$.
        	\item[(2)] $\sigma'_{j} =1 $ if $ j  \equiv 1,2 \textrm{ (mod 4)}$ and $\sigma'_{j} = 0 $ if $j \equiv 0,3 \textrm{ (mod 4)}$ for $3 \leq j \leq 2M$. 
        	\item[(3)] $\sigma_1=\sigma_2=\sigma_{2N}=\sigma_{2N+1} = \dfrac{B_{N+1}}{B_{N+1}-B_{N-1}}$.	
        \end{itemize}
        Combining both cases, Theorem \ref{thm:Longitude} gives :
        \begin{prop} Let $z_\Lambda$ be a boundary parabolic solution of the $J(N,-M)$ knot given in Theorem \ref{thm:SolJMN} with a prefered choice of $\Lambda$. Let $\rho_\Lambda$ be the holonomy representation associated to $z_\Lambda$. Then the cusp shape of $\rho_\Lambda$ is $$((-1)^{M}-1)\cdot N+\dfrac{2(B_{N+1}+B_{N-1})}{B_{N+1}-B_{N-1}}.$$
        \end{prop} 

	    We plot the cusp shapes of $J(2,-M)$ for $M=10,20$, and $30$ in Figure \ref{fig:cusp_graph} as examples. (Some cusp shapes for $J(2,-20)$ and $J(2,-30)$ are omitted in the graph.)  As their Riley polynomials are integral polynomials, their cusp shapes are in symmetry with respect to the real axis as in the graph. Also we can observe that the cusp shapes of geometric representations approach $2+2i$ which is the cusp shape of Whitehead link. We showed in Section \ref{subsec:white} that each component of the Whitehead link has cusp shape $2+2i$ for the geometric representation.
      	\begin{figure}[!h]
      		\centering
      		\def\svgwidth{0.8\columnwidth}
			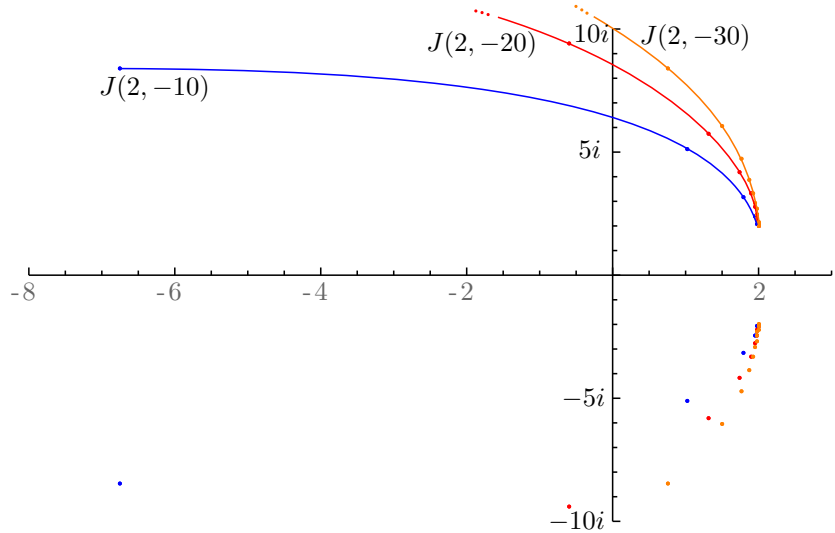
      		\caption{Cusp shapes for $J(2,-10),J(2,-20)$ and $J(2,-30)$}
      		\label{fig:cusp_graph}
      	\end{figure}
      	    	
        We finally compute the complex volume of $J(N,-M)$ as before. Here we only present graphs for the case $N=2$, i.e., for the twist knot case, as an example.
        \begin{figure}[!h]
    		\centering
    		\begin{subfigure}[t]{0.9\textwidth}
    		\centering
    		\def\svgwidth{0.9\columnwidth}
    		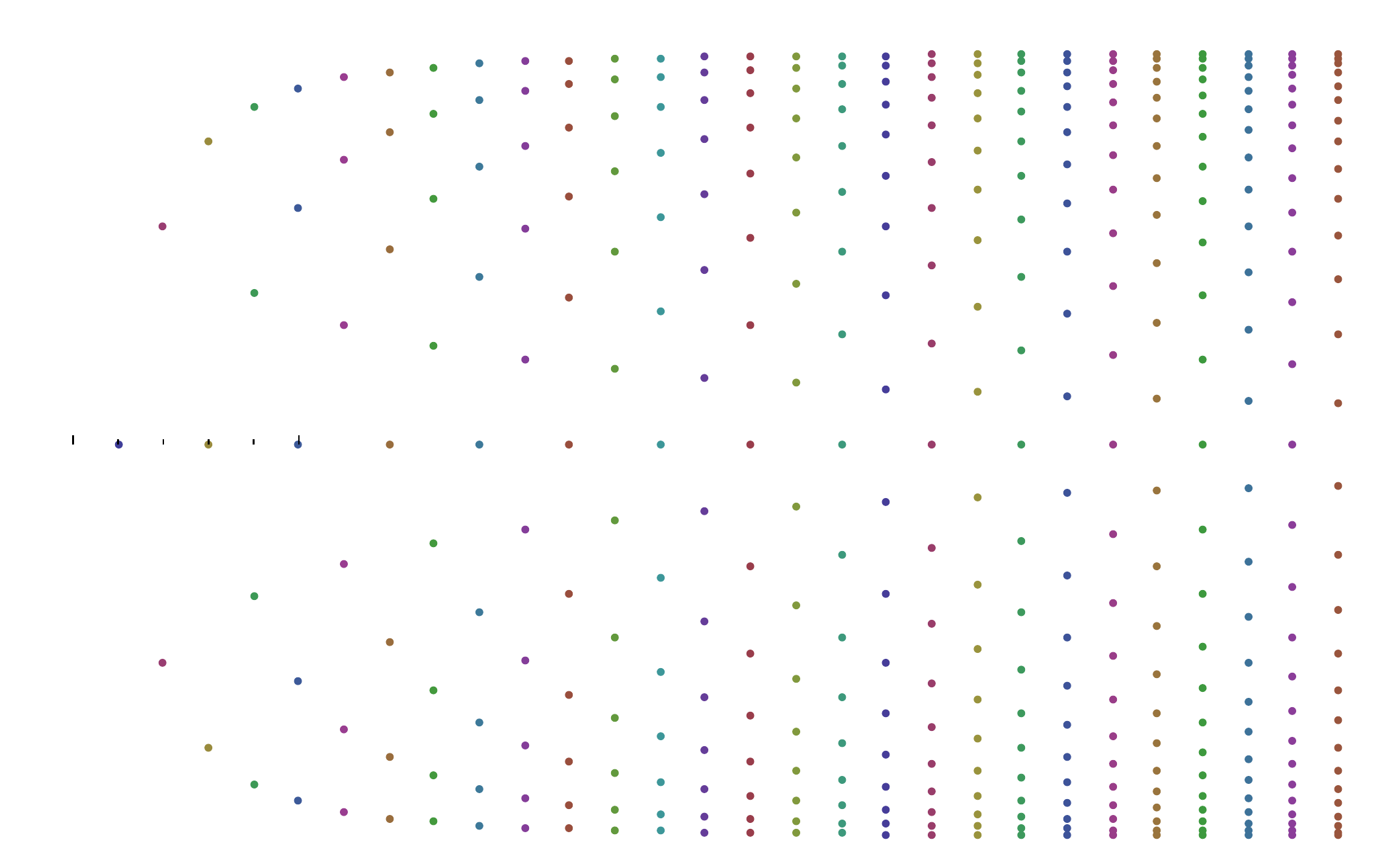 
    		\caption{Volumes of representations of $J(2,-M)$.}
    		\label{fig:ComplexVolumeJ2}
    		\end{subfigure}
    	\end{figure}
    	\begin{figure}[!h]
    		\centering
    		\ContinuedFloat
    		\begin{subfigure}[t]{0.9\textwidth}
    		\centering
    		\def\svgwidth{0.88\columnwidth}
    		~~~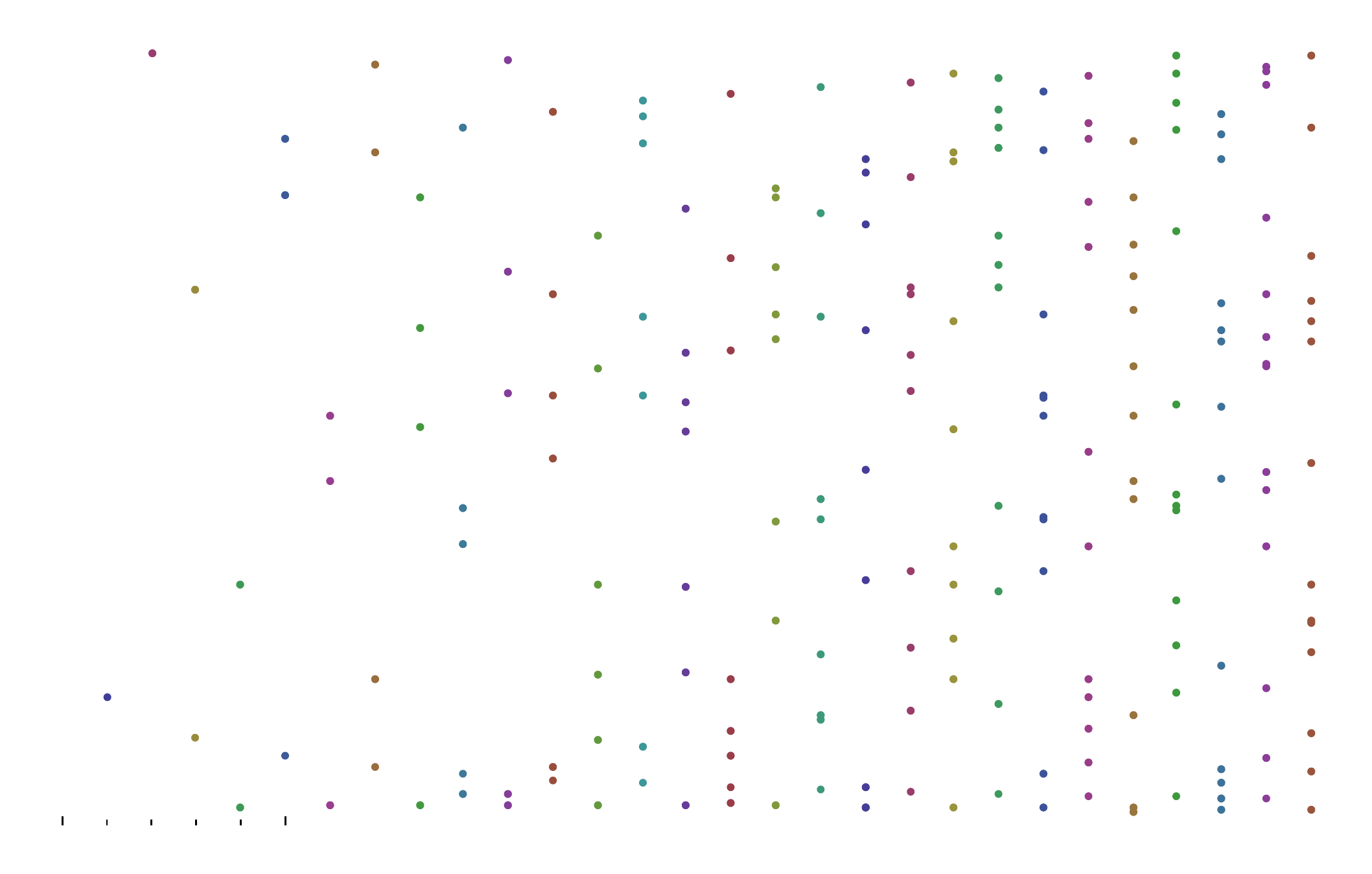 
    		\caption{Chern-Simons invariants of $J(2,-M)$ (mod $\pi^2$).}
    		\label{fig:ComplexVolumeJ2b}
    		\end{subfigure}
	   \end{figure}
	   \begin{figure}[!h]
		   	\centering
		   	\ContinuedFloat
		   	\begin{subfigure}[t]{0.9\textwidth}
    		\centering
    		\def\svgwidth{0.9\columnwidth}
    		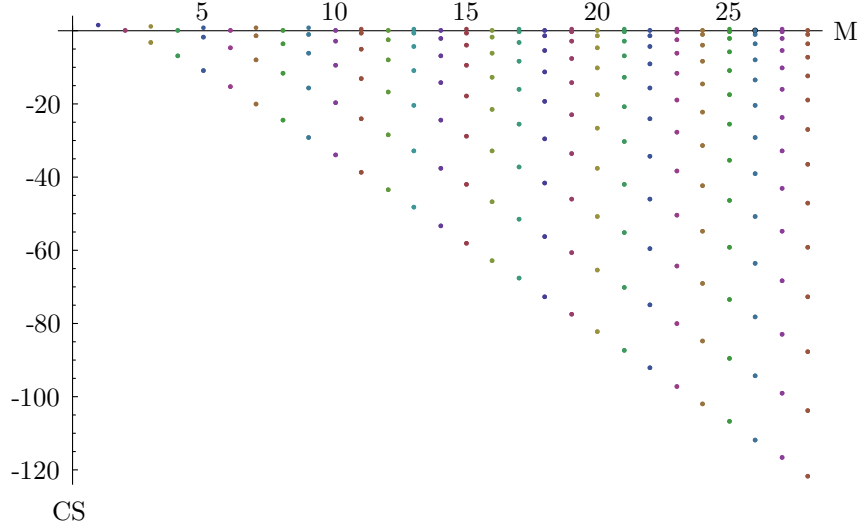 
    		\caption{Chern-Simons invariants of $J(2,-M)$ without reducing modulo $\pi^2$.}
    		\label{fig:ComplexVolumeJ2c}
    		\end{subfigure}
    	\caption{Complex volumes of $J(2,M)$.}
    	\end{figure}     
    
    	It is well-known that the maximal volume of $J(2,-M)$ tends to the volume of the Whitehead link as $M$ goes to infinity. The graph in Figure \ref{fig:ComplexVolumeJ2} tells us that  the second largest volume of $J(2,-M)$ also have the same limit and so does the third largest volume, and so on.  We also remark that $J(2,-M)$ has a representation of volume $0$ if $M$ is odd, as its Riley polynomial has a real solution.
    	
	\section{Further Examples}\label{sec:Exam}   
	\subsection{Granny and square knots}
	Connect-summing two trefoil knots, we obtain the granny knot. If we take one of them mirror-image, the connected-sum results in the square knot. We assign segment variables $z=(z_1,\cdots,z_{12})$ and $z'=(z'_1, \cdots, z'_{12})$ as follows.
	\begin{figure}[!h]
		\centering
		\scalebox{1}{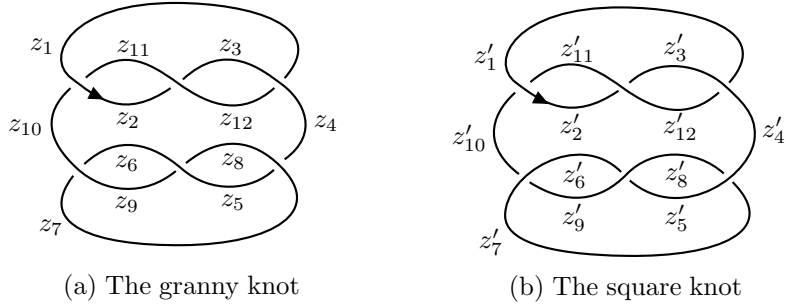}
		\caption{The granny and the square knots}
		\label{fig:square_granny}
	\end{figure}

	First, we compute a boundary parabolic solution for the granny knot. From the previous section, $(z_1,z_{10},z_{11},z_{2},z_3,z_{12})$ and $( z_{10},z_7,z_6,z_9,z_8,z_5)$ are of the form
	$$(z_1,z_{10},z_{11},z_{2},z_3,z_{12})=\left(p,\; q,\;r,\;\dfrac{p(q-r)}{p+q-r},\;\dfrac{pq}{p+q-r},\;\dfrac{pr}{r-q} \right)$$
	and
	$$(z_{10},z_7,z_6,z_9,z_8,z_5)=\left(p',\;q',\;r',\;\dfrac{p'(q'-r')}{p'+q'-r'},\;\dfrac{p'q'}{p'+q'-r'},\;\dfrac{p'r'}{r'-q'} \right)$$ for some $p,q,r$ and $p',q',r'$. Note that $z_{10}=z_4$ automatically holds as we have seen in the step (2) of Section \ref{subsec:torus}. Then we immediately obtain $p'=q$ from $z_{10}$ and $q'=\dfrac{p q^2-(p-q)(q-r) r'}{q^2+(p-q)r}$ from the hyperbolicity equations of $z_{10}$ and $z_4$. We hence obtain a solution in $p,q,r$ and $s$ (here we substitute $s$ for $r'$ for consistency) : 
	\begin{equation*}
	\begin{array}{ccl}
	(z_1,\cdots,z_{12}) &=& \left(p,\ \dfrac{p (q-r)}{p+q-r},\ \dfrac{p q}{p+q-r},\ q,\ -\dfrac{s\left(p r+q^2-q r\right)}{p (q-s)} \right.\\[12pt] 
	&&\quad \quad ,\ s, \ \dfrac{p q^2-s(p-q) (q-r)}{p r+q^2-q r},\ \dfrac{p q^2-s (p-q) (q-r)}{p (q+r-s)+q (q-r)}\\[12pt]
	&& \quad \quad \left. ,\ \dfrac{p q (q-s)}{p (q+r-s)+q (q-r)},\ q,\ r,\ \dfrac{p r}{r-q}\right).
	\end{array}
	\end{equation*}
	Similar computation using Proposition \ref{prop:mirror_seg} gives a solution for the square knot:
	\begin{equation*}
	\begin{array}{ccl}
	(z'_1,\cdots,z'_{12}) &=& \left(p,\ \dfrac{p (q-r)}{p+q-r},\ \dfrac{p q}{p+q-r},\ q,\ \dfrac{p q (q s-1)}{s \left(pr+q^2-q r'\right)} \right.\\[12pt] 
	&&\quad \quad ,\ \dfrac{1}{s}, \ \dfrac{p r+q^2-q r}{p-s (p-q) (q-r)},\ \dfrac{p \left(-q^2 s+q+r\right)+q (q-r)}{p-s (p-q) (q-r)}\\[12pt]
	&& \quad \quad \left. ,\ \dfrac{p r+q^2-q r}{p-p q s}+q,\ q,\ r,\ \dfrac{p r}{r-q}\right).
	\end{array}
	\end{equation*}
	\begin{rmk}	We remark that both the above solutions have four ``free'' parameters, $p,q,s$ and $r$. The number of free parameters was three in Example \ref{ex:fig_seg_var} and also in Section \ref{sec:Solution}. The appearance of one additional free parameter shows the existence of a $1$-parameter family of irreducible boundary parabolic representations coming from the connected-sum \cite{cho_erratum_2016}.
	\end{rmk} 
	 One can check using the solution above that these representations for the granny knot all have the same cusp shape $-12$ while those of the sqaure knot have $0$ cusp shape. This difference of cusp shapes again shows the difference of their peripheral structure. Note that there are  more irreducible representations for a connected sum when one of summands has an abelian representation.
	\subsection{The $8_5$ knot} \label{subsec:85}
	 Let us consider the $8_5$ knot diagram with segment variables $z=(z_1,\cdots ,z_{16} )$ and region variables $w=( w_1, \cdots, w_{10})$ as in Figure \ref{fig:eight_five}.
	\begin{figure}[!h]
		\centering
		\scalebox{1}{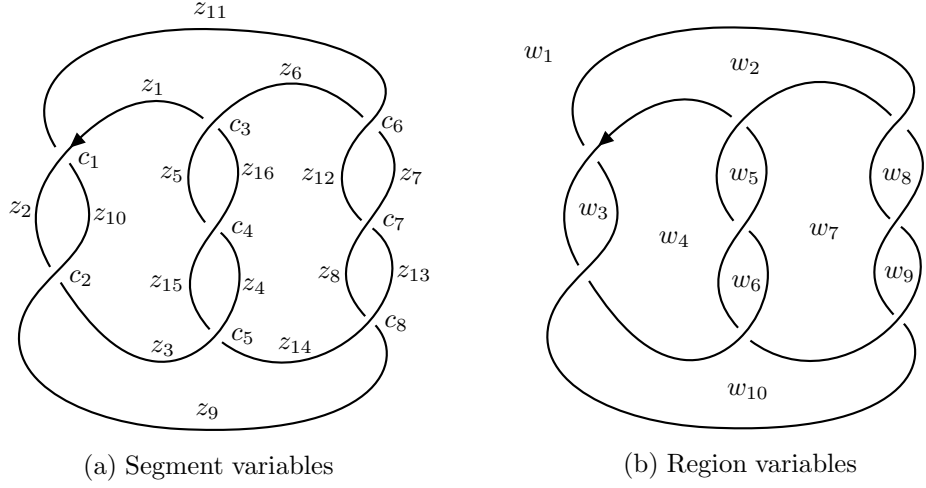}
		\caption{The $8_5$ knot}
		\label{fig:eight_five}
	\end{figure} 	
	We first give a boundary parabolic solution in segment variables, which should be a solution without pinched octahedra. The computation will be done by applying Lemma~\ref{lem:key_lem} in parallel for the three twists. We equate the variables in the ends of the twists appropriately and solve hyperbolicity equations for these variables as we did in Section~\ref{subsec:torus}. Using Mathematica, we obtain :
	\begingroup
	\allowdisplaybreaks
	\begin{align*}
			z_1 &= r \\
			z_2 &= \dfrac{-p(r-q)+q(r-p)}{-p+q(\Lambda ^3-2 \Lambda)} \\
			z_3 &= \dfrac{p(r-q)(\Lambda^3-\Lambda-1)-q(r-p)\Lambda}{-p+q(\Lambda ^3-2 \Lambda)}\\
			z_4 &= 
			\dfrac{- p(r-q)(\Lambda^4-\Lambda^3-\Lambda^2+1)+q(r-p)(\Lambda^2-\Lambda) }{-p +q(\Lambda ^2-\Lambda)} \\
			z_5 &= \dfrac{p(r-q) \left(\Lambda ^2-\Lambda -1\right)}{-p +q(\Lambda ^2-\Lambda)}\\
			z_6 &= r-q \\
			z_7 &= \dfrac{p(r-q)+q(r-p)(\Lambda^3-2\Lambda-1)}{p} \\
			z_8 &= \dfrac{p(r-q)(\Lambda-1)+q(r-p)(\Lambda^4-\Lambda^3-2\Lambda^2+\Lambda+1)}{
				p(\Lambda^6-\Lambda^5-3\Lambda^4+\Lambda^3+3\Lambda^2+2\Lambda-1)+q(\Lambda^4-\Lambda^3-2\Lambda^2+\Lambda+1)} \\
			z_9 &= \dfrac{p(r-q)(\Lambda^2-\Lambda)-q(r-p)}{p(\Lambda^2-\Lambda)-q(\Lambda^5-\Lambda^4-2\Lambda^3+\Lambda^2+\Lambda+1)}\\
			z_{10} &= \dfrac{p(r-q)(\Lambda^2-\Lambda)-q(r-p)}{q \left(\Lambda ^2-\Lambda -1\right)}\\
			z_{11}&= r-p \\
			z_{12}&= \dfrac{q(r-p) (\Lambda -1)}{p(\Lambda^3-\Lambda^2-\Lambda)+q(\Lambda-1)}\\
			z_{13} &= \dfrac{p(r-q)(\Lambda^3-\Lambda^2-\Lambda)+q(r-p)(\Lambda^6-\Lambda^5-3\Lambda^4+\Lambda^3+3\Lambda^2+2\Lambda-1)}{p(\Lambda^3-\Lambda^2-\Lambda)+q(\Lambda-1)}\\
			z_{14}&=\dfrac{p(r-q)(\Lambda^5-\Lambda^4-2\Lambda^3+\Lambda^2+\Lambda+1)-q(r-p)(\Lambda^3-\Lambda^2-\Lambda+1)}{p-q(\Lambda^3-\Lambda^2-\Lambda+1)}\\
			z_{15}&=\dfrac{-p(r-q)(\Lambda^3-\Lambda^2-1)+q(r-p)(\Lambda-1)}{p-q(\Lambda^3-\Lambda^2-\Lambda+1)} \\
			z_{16}&=\dfrac{p(r-q)(\Lambda^3-\Lambda^2-1)-q(r-p)(\Lambda-1)}{p(\Lambda-1)-q(\Lambda-1)}
	\end{align*}
	\endgroup
		where $\Lambda$ is chosen to satisfy $$\left(\Lambda ^5-\Lambda ^4-3 \Lambda ^3+2 \Lambda ^2+2 \Lambda +1\right) \left(\Lambda ^6-\Lambda ^5-2 \Lambda ^4+2 \Lambda ^2+2 \Lambda -1\right)=0.$$ Hence we obtain $11$ boundary parabolic representations. 
		We compute the cusp shape and complex volume using the above $z$-solutions for each holonomy representation.
	\begin{center}
	\begin{tabular}{c|c|c|c}
		$\Lambda$ & Volume & CS(mod $\pi^2$) & Cusp shape \\[2pt]
		\hline
		$-1.4978440$ & 0.0 &  8.514816 & -10.28576 \\[2pt]
		$-0.331409 - 0.386277i$ & $-1.138823$ &  0.373180& $-4.718078 + 6.0545 i	$\\[2pt]
		$-0.331409 + 0.386277i$&   $1.138823$&0.373180&$-4.718078 -6.0545i$  \\[2pt]
		$1.5803315 - 0.282555i$ & 6.997189 &3.594081&$-11.13904 - 3.54683i$ \\[2pt]
		$1.5803315 + 0.282555i$ & $-6.997189$ & 3.594081&$-11.13904 + 3.54683i$ \\
		\hline
		\hline
		$-1.1341553$ & 0.0 & 7.693190 & $-3.01951$ \\[2pt]
		$0.37927761$ & 0.0& 2.176413 &1.01951\\[2pt]
		$-0.592989 - 0.8475437i$ & $- 2.828122$ &$3.555607$ & $-9.50976 + 2.97945 i$\\[2pt]
		$-0.592989 + 0.8475437i$ &  2.828122 & 3.555607 & $-9.50976 - 
		2.97945 i$\\[2pt]
		$1.4704279 - 0.1026820i$ & 2.828122 & 6.313996 &$-9.50976 - 2.97945i$ \\[2pt]
		$1.4704279 + 0.1026820i$ & $-2.828122$ & 6.313996&$-9.50976 + 
		2.97945i$ 
	\end{tabular}
\end{center}

	On the other hand, the $8_5$ knot has 
	a boundary parabolic solution with pinched octahedra, which can not be obtained from segment variables. Thus we use region variables $w=(w_1,\cdots,w_{10})$. Applying Proposition \ref{lem:pinched_region} to the diagram, we have $3$ possibilities : 
	\begin{itemize}
		\item[(1)] The octahedra $o_1$ and $o_2$ are pinched.
		\item[(2)] The octahedra $o_3,o_4,$ and $o_5$ are pinched.
		\item[(3)] The octahedra $o_6,o_7,$ and $o_8$ are pinched.
	\end{itemize}
	Note that two possibilities among the above three can not occur at the same time, since in that case  every octahedron becomes pinched. 
	
	Suppose the octahedra $o_1$ and $o_2$ are pinched. Then Proposition \ref{prop:col_reg} gives additional equations $w_1+w_4=w_2+w_3$ and $w_1+w_4=w_3+w_{10}$. With these equations, we are able to obtain a solution as follows.
	\begin{equation*}
		\begin{array}{ccl}
		(w_1,\cdots,w_{10}) &=&	\left(-\dfrac{1}{p+q-pqr}+\dfrac{1}{p}+\dfrac{1}{q},\;-\dfrac{1}{p+q-pqr}+\dfrac{1}{p}+r\right.\\[13pt] 
		& &\quad \quad ,\; -\dfrac{1}{p (-q) r+p+q}+\dfrac{1}{p}+\dfrac{2}{q}-r,\; -\dfrac{1}{p+q-pqr}+\dfrac{1}{p}+\dfrac{1}{q}\\[13pt]
		& &\quad \quad ,\;-\dfrac{1}{p+q-pqr}+\dfrac{1}{q}+r,\; \dfrac{1}{p}+\dfrac{1}{q},\; r,\; \dfrac{1}{p}+\dfrac{1}{q}\\[13pt]
		& &\quad \quad \left.,\; -\dfrac{1}{p+q-pqr}+\dfrac{1}{q}+r,\; -\dfrac{1}{p+q-pqr}+\dfrac{1}{p}+r\right).
		\end{array}
	\end{equation*} In this case, we obtain the holonomy representation $\rho$,
	\begin{equation*}
		\begin{array}{lll}
			\rho(m_1)=\left(\begin{array}{cc}2 & 1 \\ -1 & 0 \end{array}\right)&,& \rho(m_2)=\left(\begin{array}{cc}2 & 1 \\ -1 & 0 \end{array}\right), \\[13pt]
			\rho(m_3)=\left(\begin{array}{cc}1 & 1 \\ 0 & 1 \end{array}\right)&,&
			\rho(m_4)=\left(\begin{array}{cc}1 & 0 \\ -1 & 1 \end{array}\right), \\[13pt]
			\rho(m_5)=\left(\begin{array}{cc}2 & -1 \\ 1 & 0 \end{array}\right)&,&
			\rho(m_6)=\left(\begin{array}{cc}0 & -1 \\ 1 & 2 \end{array}\right), \\[13pt]
			\rho(m_7)=\left(\begin{array}{cc}1 & 0 \\ -1 & 1 \end{array}\right)&,&
			\rho(m_8)=\left(\begin{array}{cc}1 & 1 \\ 0 & 1 \end{array}\right)
		\end{array}
	\end{equation*} where $m_k$ is the Wirtinger generator winding the over-arc of a crossing $c_k$. The volume of $\rho$ is $0$ and the Chern-Simons invariant (mod $\pi^2$) is $3.28987$.
	We can check that other possibilities (2) and (3) do not have a solution that induces a  non-abelian representation. Therefore, we obtain one additional representation from $w$-variables and conclude that the $8_5$ knot has exactly $12$ non-abelian boundary parabolic representations. 
	We note that there is a missing  irreducible boundary parabolic representation in the table  of  \emph{Curve project}  \cite{curve_project}.
	\subsection{The Whitehead link}\label{subsec:white}
	 Let us consider the Whitehead link diagram with segment variables as in Figure \ref{fig:whitehead}. 
    	\begin{figure}[!h]	
    		\centering
    		\scalebox{1}{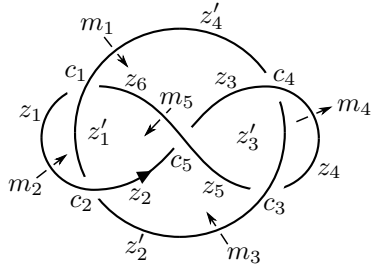}
    		\caption{The Whitehead link.}
    		\label{fig:whitehead}
    	\end{figure}
    
    	By symbolic computation using computer, one can obtain the boundary parabolic solutions to the hyperbolicity equations for the Whitehead link as follows.
    	\begin{equation*}
    	\begin{array}{ccl}
    	(z_1,\cdots,z_6) &=& \left( q+(1+\Lambda)p -\Lambda \dfrac{pq}{r},\ -\dfrac{(1-\Lambda) p r}{q}+p+ \Lambda r,\ p+r \right. \\[5pt]
    	& & \quad \quad \quad \quad \quad \left. ,\dfrac{q(p+r)}{(1+\Lambda)q-\Lambda r},\ -\dfrac{\Lambda r (p+q)}{(1+\Lambda) q-\Lambda r},\ p+q \right) \\[10pt]
    	(z'_1,\cdots,z'_4) &=& \left(\dfrac{p (q-r)}{q-\Lambda r},\ p,\ \dfrac{\Lambda p q-p r-(1-\Lambda) q r}{q-r},\ \dfrac{p (r-\Lambda q)+(1-\Lambda) q r}{q-\Lambda r}\right)
    	\end{array}
    	\end{equation*} where $\Lambda^2=-1$.
    	
    	Let $\rho_\Lambda$ be the holonomy representation corresponding to the choice of $\Lambda$. We first compute $\rho_\Lambda(m_k)$ of the Wirtinger generator $m_k$ winding the over-arc of a crossing $c_k$. Let $\Lambda_k$ be the crossing label of $c_k$. Let $K$ be the component of the Whitehead link which passes $c_5$ and let $K'$ be the other component. Let us choose a base point $P$ below the crossing $c_2$ and follow the component $K$. Then as in Section \ref{sec:Holonomy} we have the followings under the assumption $\rho_\Lambda(m_2) =  \begin{pmatrix} 1 & 1\\0 & 1\end{pmatrix}$. 
	   	\begin{itemize}
	   		\item $\rho_\Lambda(m_5) = \arraycolsep=4pt\begin{pmatrix}  1 +\Lambda_5 \Sigma_5 & -\Lambda_5 \Sigma_5^2\\\Lambda_5 & 1-\Lambda_5 \Sigma_5 \end{pmatrix} =\begin{pmatrix} -	\Lambda & -1 \\ 2\Lambda & \Lambda+2 \end{pmatrix}$ where $\Sigma_5$ is the segment label of $z_2$. 
	   		\item $\rho_\Lambda(m_3) = \rho_\Lambda(m_5) \arraycolsep=4pt\begin{pmatrix} 1 +\Lambda_3 \Sigma_3& -\Lambda_3 \Sigma_3^2 \\ \Lambda_3 & 1-\Lambda_3 \Sigma_3 \end{pmatrix} \rho_\Lambda(m_5^{-1})=\begin{pmatrix} 1 & 0 \\ \Lambda-1 & 1 \end{pmatrix}$ where $\Sigma_3$ is the sum of the segment labels of $z_2,z_3,$ and $z_4$.
	   		\item $\rho_\Lambda(m_1) = \rho_\Lambda(m_3^{-1}m_5) \arraycolsep=4pt\begin{pmatrix} 1 +\Lambda_1 \Sigma_1& -\Lambda_1 \Sigma_1^2 \\ \Lambda_1 & 1-\Lambda_1 \Sigma_1 \end{pmatrix} \rho_\Lambda(m_5^{-1}m_3)=\begin{pmatrix} 2-\Lambda & 1-\Lambda \\ \Lambda-1 & \Lambda \end{pmatrix}$ where $\Sigma_1$ is the sum of the segment labels of $z_2,z_3,z_4,z_5,$ and $z_6$.
	   		\item At the crossing $c_5$ we have $m_4=m_5 m_2 m_5^{-1}$ and hence $\rho_\Lambda(m_4) =\begin{pmatrix} -1 & -1 \\ 4 & 3 \end{pmatrix}$.
	   	\end{itemize}

		We can also compute the cusp shape of $\rho_\Lambda$ for each component as in Remark \ref{rmk:cusp_link}. The cusp shape for the component $K$ is $2+2\Lambda$, which is the sum of the segment labels of $z_1,z_2,\cdots,z_6$ with writhe correction $-1$. Similarly, the cusp shape for the component $K'$ is also $2+2\Lambda$, which is the sum of the segment labels of $z'_1,z'_2,z'_3$ and $z'_4$ with writhe correction $0$.   Finally, Theorem 1.2 in \cite{CKK_2014} gives that the complex volume (mod $\pi^2i$) is $\pm3.66386 -2.4674 i$.
		
		One can claim that the above representations are all of the non-abelian boundary parabolic representations of the Whitehead link if there is no solution with pinched octahedra. In fact, this is true since Proposition \ref{lem:pinched_region} tells us that the octahedron in the crossing $c_5$ is the only possibility for a pinched octahedron. In that case, we can compute the cross-ratios satisfying the gluing equations and verify that they give the trivial representation.

\end{document}